\documentclass[12pt]{amsart}

\usepackage{amsmath}
\usepackage{amsthm}
\usepackage{amsopn}
\usepackage{amssymb}
\usepackage[pdftex]{graphicx}
\usepackage[utf8]{inputenc}
\usepackage[T1]{fontenc}
\usepackage[all]{xy}
\usepackage{mathrsfs}
\usepackage{enumerate}
\usepackage{etoolbox}
\usepackage{textcomp}
\usepackage{mathtools}
\usepackage{bbm}

\newcommand{\ZZ}{\mathbb{Z}}
\newcommand{\CC}{\mathbb{C}}
\newcommand{\RR}{\mathbb{R}}
\newcommand{\TT}{\mathbb{T}}
\newcommand{\FF}{\mathcal{F}}
\newcommand{\BB}{\mathcal{B}}
\newcommand{\HH}{\mathcal{H}}

\newcommand{\OO}{\mathscr{O}}
\newcommand{\SLASH}{\char`\\}
\newcommand{\spacedash}{\hspace{.06cm} - \hspace{.06cm}}
\newcommand{\NN}{\mathbb{N}}
\newcommand{\SSS}{\mathbb{S}}
\newcommand{\defeq}{\vcentcolon=}
\newcommand{\Fin}{\mathcal{F}in}
\newcommand{\slant}{\mathcal{S}}

\DeclareMathOperator*{\hocoeq}{hocoeq}
\DeclareMathOperator*{\colim}{colim}
\DeclareMathOperator{\Res}{Res}
\DeclareMathOperator{\Ind}{Ind}

\newtheorem{thm}[equation]{Theorem}
\newtheorem{lem}[equation]{Lemma}
\newtheorem{prop}[equation]{Proposition}
\newtheorem{cor}[equation]{Corollary}
\theoremstyle{definition} \newtheorem{definition}[equation]{Definition}

\numberwithin{equation}{section}

\newlength\tindent
\renewcommand{\indent}{\hspace*{\tindent}}

\title{Symmetric Powers and Norms of Mackey Functors}
\author{John Ullman}

\begin{document}

\begin{abstract}
In this paper we give detailed algebraic descriptions of the derived symmetric power and norm constructions on categories of Mackey functors, as well as the derived $G$-symmetric monoidal structure. We build on the results of~\cite{UTamb}, in which it is shown that every Tambara functor over a finite group $G$ arises as the zeroth stable homotopy group of a commutative ring $G$-spectrum. The norm / restriction adjunctions on categories of Tambara functors promised in~\cite{UTamb} are demonstrated algebraically. Finally, we give a new characterization of Tambara functors in terms of multiplicative push forwards of Mackey functors, and use this to obtain an appealing new description of the free Tambara functor on a Mackey functor which closely matches the structure of equivariant extended powers.
\end{abstract}

\maketitle

\section{Introduction}\label{sec:intro}

Let $G$ be a finite group and $n > 1$. We can define the \emph{$n$'th symmetric power} construction on Mackey functors over $G$ by
\begin{align*}
	Sym_n (\underline{M}) \defeq \underline{\pi}_0 ((\HH\underline{M})^{\wedge n}/\Sigma_n).
\end{align*}

Note that this is \emph{not} equal to $\underline{M}^{\otimes n}/\Sigma_n$. However, the results of~\cite{UTamb} make it possible to give an algebraic description of $Sym_n$. We restate the relevant results here for convenience, using $\CC$ to denote the derived free commutative ring $G$-spectrum functor, and recalling that $\underline{\pi}_0$ of a commutative ring $G$-spectrum is a Tambara functor (see~\cite{Brun} or~\cite{Stri}).

\begin{thm}\label{thm:mapstoemcomm}
[Theorem 5.2 of \cite{UTamb}] If $X$ and $\HH\underline{R}$ are commutative ring $G$-spectra, with $X$ $(-1)$-connected and $\HH\underline{R}$ Eilenberg MacLane, then $\underline{\pi}_0$ induces a bijection between maps $X \to \HH\underline{R}$ in the homotopy category of commutative ring $G$-spectra and maps $\underline{\pi}_0 X \to \underline{R}$ of Tambara functors.
\end{thm}

\begin{thm}\label{thm:homemcomm}
[Theorem 5.3 of \cite{UTamb}] The functor $\underline{\pi}_0$ induces an equivalence from the homotopy category of Eilenberg MacLane commutative ring $G$-spectra to the category of Tambara functors.
\end{thm}

\begin{cor}\label{cor:freetambara}
[Corollary 5.8 of \cite{UTamb}] If $\underline{M}$ is any Mackey functor then $\underline{\pi}_0 \CC (\HH\underline{M})$ is the free Tambara functor on $\underline{M}$.
\end{cor}

Thus, for any Mackey functor $\underline{M}$ we have that $Sym_n (\underline{M})$ is the $n$'th degree part of the free Tambara functor on $\underline{M}$. We give a detailed algebraic description of this in Section~\ref{sec:sympow}. The precise statement is given by Corollary~\ref{cor:identsympow}, and the geometric fixed points are given by Corollary~\ref{cor:geomfixpsympow}.\\
\indent Next, let $H$ be a proper subgroup of $G$. Letting $N_H^G$ denote the derived norm functor of~\cite{HHR}, we can define a \emph{norm} construction sending $H$-Mackey functors to $G$-Mackey functors by
\begin{align*}
	N_H^G (\underline{M}) \defeq \underline{\pi}_0 N_H^G (\HH\underline{M}).
\end{align*}

Using our description of free Tambara functors, and the fact that the norm functor gives the left adjoint of restriction from commutative ring $G$-spectra to commutative ring $H$-spectra, we give an algebraic description of this norm construction in Section~\ref{sec:norm}. The precise statement is given by Theorem~\ref{thm:normimage}, and the geometric fixed points are given by Proposition~\ref{prop:geomfixpnorm}.\\
\indent Next we put a \emph{$G$-symmetric monoidal structure} on Mackey functors: for each finite $G$-set $T$ and Mackey functor $\underline{M}$, we let
\begin{align*}
	\underline{M}^{\otimes T} \defeq \underline{\pi}_0 (\HH\underline{M})^{\wedge T}.
\end{align*}

We give an algebraic description of this in Section~\ref{sec:gsymmon}. The precise statement is given by Theorem~\ref{thm:identgsymmon}.\\
\indent These first descriptions of the norm and $G$-symmetric monoidal structure are unsatisfying, since they do not give intrinsic generators and relations. We give intrinsic descriptions in Section~\ref{sec:intrinsic}, along with several structure maps. Then we use \emph{multiplicative push forwards} to cleanly match our descriptions to topology in Section~\ref{sec:multpush}.\\
\indent In~\cite{UTamb} it is shown that the norm construction $N_H^G$ on Mackey functors gives the left adjoint of restriction on Tambara functors. We restate the result here for convenience.
\begin{cor}\label{cor:utambnormresadj}
[Corollary 5.13 of~\cite{UTamb}] If $H$ is a subgroup of $G$, then the left adjoint of the restriction functor from $G$-Tambara functors to $H$-Tambara functors coincides with $N_H^G$ on underlying commutative Green functors.
\end{cor}

In Section~\ref{sec:normresadj} we give an algebraic demonstration of this adjunction, obtaining along the way a similar adjunction for semi-Tambara functors.\\
\indent Finally, in Section~\ref{sec:multpushmackey} we define \emph{multiplicative push forwards of Mackey functors} and give an algebraic description. We use these to give an alternative characterization of Tambara functors as \emph{multiplicative Mackey functors}, and to give an alternative formula for the free Tambara functor on a Mackey functor which closely matches the structure of equivariant extended powers.\\
\indent We work throughout this paper with orthogonal spectra. We denote the category of orthogonal $G$-spectra by $Sp_G$ and the category of commutative ring $G$-spectra by $comm_G$. We utilize the $\SSS$ model structures of~\cite{Stolz}, so that induction and restriction functors are Quillen pairs, and always index on a complete $G$-universe. Recall that we can pull back the positive model structure to $comm_G$. We refer to~\cite{Stolz} and Section~A.4 of~\cite{Ull} for background on these model structures, which build on the classical ones from~\cite{MM}.\\
\indent We will use the notation $Mack(G)$ for the category of $G$-Mackey functors, and $sMack(G)$ for the category of semi-Mackey functors (that is, Mackey functors without additive inverses).

\section{Symmetric Powers of Mackey Functors}\label{sec:sympow}

In this section we give an algebraic description of symmetric powers of Mackey functors. As explained in the introduction, this comes down to describing the free Tambara functor on a Mackey functor. Hence, we begin by giving a definition of (semi-)Tambara functors, after some preliminaries. Let $G$ be a finite group, and let $\Fin_G$ denote the category of finite $G$-sets. Also let $Set_{\neq \emptyset}$ denote the category of nonempty sets. Let $i : X \to Y$ and $j : Y \to Z$ be maps in $\Fin_G$. Let
\begin{align*}
	\textstyle \prod_{i,j} X \defeq \{ (z, s) : z \in Z, s : j^{-1} (z) \to X, i \circ s = Id \}
\end{align*}

be the set of sections of $i$ defined on fibers of $j$, with $G$ acting by conjugation. There is an obvious $G$-map
\begin{align*}
	p : \textstyle \prod_{i,j} X &\to Z \\
	(z, s) &\mapsto z
\end{align*}

as well as an evaluation $G$-map as below.
\begin{align*}
	e : Y \times_Z \textstyle \prod_{i,j} X &\to X \\
	(y, (z, s)) &\mapsto s(y)
\end{align*}

Observe that the diagram below commutes.
\begin{align*}
\xymatrix{
 Y \times_Z \textstyle \prod_{i,j} X \ar[d]_-{e} \ar[rr]^-{\pi_2} & & \textstyle \prod_{i,j} X \ar[d]^-{p} \\
 X \ar[r]_-{i} & Y \ar[r]_-{j} & Z }
\end{align*}

An \emph{exponential diagram} is any diagram in $\Fin_G$ which is isomorphic to one of the above form. If the diagram
\begin{align*}
\xymatrix{
 A \ar[d]_-{f} \ar[rr]^-{g} & & B \ar[d]^-{h} \\
 X \ar[r]_-{i} & Y \ar[r]_-{j} & Z }
\end{align*}

is an exponential diagram we will say that $(f,g,h)$ is a \emph{distributor} for $(i,j)$. We can now define semi-Tambara functors. Our definition is equivalent to that of Tambara's "semi-TNR functors" in~\cite{Tambara}.

\begin{definition}\label{def:tambara}
A \emph{semi-Tambara functor} $\underline{M}$ is a triplet of functors
\begin{align*}
	\underline{M}^* : \Fin_G^{op} &\to Set_{\neq \emptyset} \\
	\underline{M}_* : \Fin_G &\to Set_{\neq \emptyset} \\
	\underline{M}_{\star} : \Fin_G &\to Set_{\neq \emptyset}
\end{align*}
with common object assignment $X \mapsto \underline{M} (X)$ such that
\begin{enumerate}[(i)]
\item if $X \xrightarrow{i} Z \xleftarrow{j} Y$ is a coproduct in $\Fin_G$ then
\begin{align*}
	\underline{M} (X) \xleftarrow{\underline{M}^* (i)} \underline{M} (Z) \xrightarrow{\underline{M}^* (j)} \underline{M} (Y)
\end{align*}
is a product in $Set_{\neq \emptyset}$,
\item for any pullback diagram
\begin{align*}
\xymatrix{
 P \ar[d]_-{p} \ar[r]^-{q} & Y \ar[d]^-{i} \\
 X \ar[r]_-{j} & Z }
\end{align*}
we have the two relations $\underline{M}^* (j) \circ \underline{M}_* (i) = \underline{M}_* (p) \circ \underline{M}^* (q)$ and $\underline{M}^* (j) \circ \underline{M}_{\star} (i) = \underline{M}_{\star} (p) \circ \underline{M}^* (q)$, and
\item for any exponential diagram
\begin{align*}
\xymatrix{
 A \ar[d]_-{f} \ar[rr]^-{g} & & B \ar[d]^-{h} \\
 X \ar[r]_-{i} & Y \ar[r]_-{j} & Z }
\end{align*}
we have $\underline{M}_{\star} (j) \circ \underline{M}_* (i) = \underline{M}_* (h) \circ \underline{M}_{\star} (g) \circ \underline{M}^* f$.
\end{enumerate}
A map of semi-Tambara functors $\underline{M} \to \underline{N}$ is a collection of maps of sets $\underline{M} (X) \to \underline{N} (X)$ which forms a triplet of natural transformations $\underline{M}^* \to \underline{N}^*$, $\underline{M}_* \to \underline{N}_*$, $\underline{M}_{\star} \to \underline{N}_{\star}$. We denote the category of semi-Tambara functors by $sTamb(G)$.
\end{definition}

The third condition above is called the \emph{distributive law}. If $\underline{M}$ is a semi-Tambara functor the structure maps $\underline{M}^* (f)$ are called \emph{restrictions}, the $\underline{M}_* (f)$ are called \emph{transfers} and the $\underline{M}_{\star} (f)$ are called \emph{norms}. When the choice of $\underline{M}$ is clear we will denote these by $r_f$, $t_f$ and $n_f$, respectively. \\
\indent Now for any $X \in \Fin_G$, the composite
\begin{align*}
	\underline{M} (X) \times \underline{M} (X) \cong \underline{M} (X \textstyle \coprod X) \xrightarrow{\underline{M}_* (Id_X \coprod Id_X)} \underline{M} (X)
\end{align*}

defines an operation making $\underline{M}(X)$ into a commutative monoid. We call this addition; the unit (zero) comes from the unique transfer
\begin{align*}
	\underline{M}_* : \underline{M} (\emptyset) \to \underline{M} (X).
\end{align*}

A \emph{Tambara functor} is a semi-Tambara functor $\underline{M}$ such that these monoids are abelian groups. We denote by $Tamb(G)$ the category of Tambara functors. Note that we obtain analogous definitions of semi-Mackey and Mackey functors by deleting the norms from the above definition.\\
\indent Next, using norm maps instead, we obtain a second operation which distributes over the first. We call this multiplication. With these commutative semi-ring structures the restrictions become maps of rings, the transfers are maps of modules, and the norms are maps of multiplicative monoids. Thus a Tambara functor defines a commutative Green functor. We also obtain forgetful functors $sTamb(G) \to sMack(G)$ and $Tamb(G) \to Mack(G)$ by neglect of the norms.\\
\indent The Grothendieck group construction gives left adjoints
\begin{align*}
	sMack(G) &\to Mack(G)\\
	sTamb(G) &\to Tamb(G)
\end{align*}

to the appropriate forgetful functors. For Mackey functors this is trivial; for Tambara functors see~\cite{Tambara} (or, alternatively, Section~13 of~\cite{Stri}). Next we give a definition of free Tambara functors.
\begin{definition}\label{def:freetambara}
Let $\underline{M}$ be a Mackey functor. A \emph{free Tambara functor} on $\underline{M}$ is a Tambara functor $\TT(\underline{M})$ together with a map of Mackey functors $\underline{M} \to \TT(\underline{M})$ which is initial among maps from $\underline{M}$ to Tambara functors. A \emph{free semi-Tambara functor} on a semi-Mackey functor $\underline{M}$ is a semi-Tambara functor $s\TT(\underline{M})$ together with a map of semi-Mackey functors $\underline{M} \to s\TT(\underline{M})$ which is initial among maps from $\underline{M}$ to semi-Tambara functors.
\end{definition}

Of course, free (semi-)Tambara functors are unique up to unique isomorphism, so we may speak of \emph{the} free (semi-)Tambara functor on a (semi-)Mackey functor.\\
\indent We now give a construction of the free semi-Tambara functor on a semi-Mackey functor $\underline{M}$. Let $X \in \Fin_G$. We define $s\TT_0 (\underline{M}) (X)$ to be the set of equivalence classes of pairs $(U \xrightarrow{i} V \xrightarrow{j} X, u \in \underline{M} (U))$, where $i$ and $j$ are maps in $\Fin_G$ and $(U \xrightarrow{i} V \xrightarrow{j} X, u \in \underline{M} (U))$ is equivalent to $(U' \xrightarrow{i'} V' \xrightarrow{j'} X, u' \in \underline{M} (U'))$ if and only if there is a commutative diagram
\begin{align*}
\xymatrix{
 U \ar[r]^-{i} & V \ar[r]^-{j} & X \\
 U' \ar[u]^-{f}_-{\cong} \ar[r]_-{i'} & V' \ar[u]^-{g}_-{\cong} \ar[ru]_-{j'} & }
\end{align*}

such that $f$ and $g$ are isomorphisms and $r_f (u) = u'$. Such a pair represents a transfer of a norm of an element from $\underline{M}$. We define restrictions, transfers and norms for $s\TT_0 (\underline{M})$ as follows. If $f : X \to Y$ is a map in $\Fin_G$ we define $t_f ((U \xrightarrow{i} V \xrightarrow{j} X, u))$ to be $(U \xrightarrow{i} V \xrightarrow{f \circ j} Y, u)$. If instead $f : Y \to X$, we form the diagram below, where the squares are pullbacks,
\begin{align*}
\xymatrix{
 U \ar[r]^-{i} & V \ar[r]^-{j} & X \\
 Q \ar[u]^-{g} \ar[r]_-{h} & P \ar[u]^-{q} \ar[r]_-{p} & Y \ar[u]_-{f} }
\end{align*}

and define $r_f ((U \xrightarrow{i} V \xrightarrow{j} X, u))$ to be $(Q \xrightarrow{h} P \xrightarrow{p} Y, r_g (u))$. It is simple to verify that we now have a semi-Mackey functor. Addition is achieved by taking disjoint unions of the $U$'s and $V$'s; the zero element consists of the diagram with $U$ and $V$ empty and the element $0 \in \underline{M}(\emptyset)$. Now suppose again that $f : X \to Y$; to define the corresponding norm we take our cue from the distributive law. Form the diagram below, where the rectangle is exponential and the square is a pullback,
\begin{align*}
\xymatrix{
 U \ar[r]^-{i} & V \ar[r]^-{j} & X \ar[r]^-{f} & Y \\
 P \ar[u]^-{p} \ar[r]_-{q} & A \ar[u]_-{g} \ar[rr]_-{h} && B \ar[u]_-{k} }
\end{align*}

and define $n_f ((U \xrightarrow{i} V \xrightarrow{j} X, u))$ to be $(P \xrightarrow{h \circ q} B \xrightarrow{k} Y, r_p (u))$. The fact that $s\TT_0 (\underline{M})$ is a semi-Tambara functor follows from the following three lemmas, which we state without proof.

\begin{lem}\label{lem:distrpullback}
(Commutation of norms and restrictions) Suppose given maps of finite $G$-sets as below, where the squares are pullbacks.
\begin{align*}
\xymatrix{
 X \ar[r]^-{i} & Y \ar[r]^-{j} & Z \\
 Q \ar[u] \ar[r]_-{g} & P \ar[u] \ar[r]_-{h} & W \ar[u]_-{k} }
\end{align*}
Then the pullback over $k$ of the exponential diagram for $i,j$ is the exponential diagram for $g,h$.
\end{lem}

\begin{lem}\label{lem:distrlaw}
(Distributive law) Suppose given a commutative diagram of finite $G$-sets as below.
\begin{align*}
\xymatrix{
 & & C \ar[dl]_-{g} \ar[r]^-{f} & D \ar[dd]^-{p} \\
 & P \ar[dl]_-{h} \ar[dr] & & \\
 V \ar[dr]_-{i} & & A \ar[dl] \ar[r] & B \ar[dd]^-{q} \\
 & X \ar[dr]_-{j} & & \\
 & & Y \ar[r]_-{k} & Z }
\end{align*}
If the square is a pullback and the two interior pentagons are exponential, then the outer pentagon is exponential; that is, the maps $hg$, $f$ and $qp$ form a distributor for $ji, k$.
\end{lem}

\begin{lem}\label{lem:functorialnorm}
(Functorality of norm) Suppose given a commutative diagram of finite $G$-sets as below.
\begin{align*}
\xymatrix{
 Q \ar[d]_-{g} \ar[r]^-{p} & C \ar[d] \ar[r]^-{q} & D \ar[ddd]^-{r} \\
 A \ar[d]_-{h} \ar[r] & B \ar[dd] & \\
 V \ar[d]_-{i} & & \\
 X \ar[r]_-{j} & Y \ar[r]_-{k} & Z }
\end{align*}
If the square is a pullback and the two interior rectangles are exponential, then the outer rectangle is exponential; that is, the maps $hg$, $qp$ and $r$ form a distributor for $i,kj$.
\end{lem}

Now $s\TT_0 (\underline{M})$ is \emph{not} the free semi-Tambara functor on $\underline{M}$. (In fact, it is the free semi-Tambara functor on $\underline{M}$ regarded as an object with restrictions but not transfers.) To obtain the free semi-Tambara functor on $\underline{M}$ we must impose an equivalence relation representing the distributive law, as follows. Given $U \xrightarrow{i} V \xrightarrow{j} X$ in $\Fin_G$, let $k : W \to U$ be a map in $\Fin_G$ and form the commutative diagram below, where the rectangle is exponential.
\begin{align*}
\xymatrix{
 W \ar[r]^-{k} & U \ar[r]^-{i} & V \ar[r]^-{j} & X \\
 A \ar[u]^-{f} \ar[rr]_-{g} && B \ar[u]_-{h} \ar[ur]_-{j \circ h} & }
\end{align*}

We define $s\TT (\underline{M}) (X)$ to be the quotient of $s\TT_0 (\underline{M}) (X)$ by the smallest equivalence relation $\sim$ such that
\begin{align*}
	(U \xrightarrow{i} V \xrightarrow{j} X, t_k (w)) \sim (A \xrightarrow{g} B \xrightarrow{j \circ h} X, r_f (w))
\end{align*}

for all $W \xrightarrow{k} U \xrightarrow{i} V \xrightarrow{j} X$ in $\Fin_G$ and all $w \in \underline{M} (W)$. It is clear that the transfer maps for $s\TT_0 (\underline{M})$ descend to this quotient, and the restriction maps descend by Lemma~\ref{lem:distrpullback}. The fact that the norms descend is more difficult. We prove it below.

\begin{lem}\label{lem:normsdescend}
The norm maps on $s\TT_0 (\underline{M})$ descend to $s\TT (\underline{M})$.
\end{lem}
\begin{proof}
We must prove that, if two elements in $s\TT_0 (\underline{M})$ are equivalent, then so are their norms. Unfortunately, this requires a large amount of notation. Consider a commutative diagram of the form below, where the pentagons are exponential and the squares are pullbacks.
\begin{align*}
\xymatrix{
 & L \ar[d]_-{l_1} \ar[r]^-{l_2} & W \ar[d]_-{w_1} \ar[r]^-{w_2} & Y \ar[ddr]^-{y} & \\
 & J \ar[dl]_-{j_1} \ar[r]^-{j_2} & K \ar[dl]_-{k_1} \ar[dr]^-{k_2} & & \\
 F \ar[d]_-{f_1} \ar[r]^-{f_2} & V \ar[dr]^-{v} && H \ar[dl]_-{h_1} \ar[r]^-{h_2} & I \ar[d]^-{i} \\
 A \ar[r]_-{a} & B \ar[r]_-{b} & C \ar[r]_-{c} & D \ar[r]_-{d} & E }
\end{align*}
Next form the diagram below, where the squares are pullbacks.
\begin{align*}
\xymatrix{
 P \ar[d]_-{p_1} \ar[r]^-{p_2} & Q \ar[d]^-{q_1} \ar[r]^-{q_2} & H \ar[d]^-{h_1} \\
 A \ar[r]_-{a} & B \ar[r]_-{b} & C }
\end{align*}
Now Lemma~\ref{lem:distrpullback} tells us that the exponential diagram for $a,b$ pulls back over $h_1$ to the exponential diagram for $p_2,q_2$. Thus we obtain an exponential diagram
\begin{align*}
\xymatrix{
 J \ar[d]_-{j_3} \ar[rr]^-{j_2} && K \ar[d]^-{k_2} \\
 P \ar[r]_-{p_2} & Q \ar[r]_-{q_2} & H }
\end{align*}
and a commutative diagram as below.
\begin{align*}
\xymatrix{
 F \ar[d]_-{f_1} & J \ar[d]^-{j_3} \ar[l]_-{j_1} \\
 A & P \ar[l]^-{p_1} }
\end{align*}
Now every generating relation on $s\TT_0 (\underline{M}) (F)$ is of the form
\begin{align*}
	(B \xrightarrow{b} C \xrightarrow{c} D, t_a (x)) \sim (F \xrightarrow{f_2} V \xrightarrow{c \circ v} D, r_{f_1} (x))
\end{align*}
for some $x$, $A$, $B$, etc. First we calculate
\begin{align*}
	n_d ((B \xrightarrow{b} C \xrightarrow{c} D, t_a (x))) = (Q \xrightarrow{h_2 \circ q_2} I \xrightarrow{i} E, r_{q_1} t_a (x)).
\end{align*}
Next, since $k_1 \circ w_1, w_2, i \circ y$ is a distributor for $c \circ v, d$ by Lemma~\ref{lem:distrlaw}, we obtain
\begin{align*}
	n_d ((F \xrightarrow{f_2} V \xrightarrow{c \circ v} D, r_{f_1} (x))) = (L \xrightarrow{w_2 \circ l_2} Y \xrightarrow{i \circ y} E, r_{l_1} r_{j_1} r_{f_1} (x)).
\end{align*}
Now $r_{q_1} t_a = t_{p_2} r_{p_1}$, and Lemma~\ref{lem:functorialnorm} implies that $j_3 \circ l_1, w_2 \circ l_2, y$ is a distributor for $p_2,h_2 \circ q_2$, so we obtain the following.
\begin{align*}
	(Q \xrightarrow{h_2 \circ q_2} I \xrightarrow{i} E, r_{q_1} t_a (x)) &= (Q \xrightarrow{h_2 \circ q_2} I \xrightarrow{i} E, t_{p_2} r_{p_1} (x)) \\
	                                                                                           &\sim (L \xrightarrow{w_2 \circ l_2} Y \xrightarrow{i \circ y} E, r_{l_1} r_{j_3} r_{p_1} (x))
\end{align*}
But since $p_1 \circ j_3 = f_1 \circ j_1$, we have $r_{l_1} r_{j_3} r_{p_1} (x) = r_{l_1} r_{j_1} r_{f_1} (x)$, so that
\begin{align*}
	n_d ((B \xrightarrow{b} C \xrightarrow{c} D, t_a (x))) \sim n_d((F \xrightarrow{f_2} V \xrightarrow{c \circ v} D, r_{f_1} (x))).
\end{align*}
\end{proof}

It is now clear that $s\TT (\underline{M})$ satisfies every part of Definition~\ref{def:tambara} except part (i). However, this is easily seen to hold when we note that disjoint unions of exponential diagrams are exponential. This implies that the equivalence relation on
\begin{align*}
	s\TT_0 (\underline{M}) (X_1 \textstyle \coprod \displaystyle X_2) \cong s\TT_0 (\underline{M}) (X_1) \times s\TT_0 (\underline{M}) (X_2)
\end{align*}

is the product of the two equivalence relations on $s\TT_0 (\underline{M}) (X_1)$ and $s\TT_0 (\underline{M}) (X_2)$.\\
\indent We must now show that $s\TT (\underline{M})$ is actually a free semi-Tambara functor on $\underline{M}$. First we require a map of semi-Mackey functors from $\underline{M}$ to $s\TT (\underline{M})$. Let $X$ be a finite $G$-set. We define a function as below.
\begin{align*}
	\theta_{\underline{M}} (X) : \underline{M} (X) &\to s\TT (\underline{M}) (X) \\
	                                                                        x &\mapsto (X \xrightarrow{=} X \xrightarrow{=} X, x)
\end{align*}

We now verify that this is a map of semi-Mackey functors.

\begin{lem}\label{lem:univmapsemitamb}
The functions $\theta_{\underline{M}} (X)$ determine a map
\begin{align*}
	\theta_{\underline{M}} : \underline{M} \to s\TT (\underline{M})
\end{align*}
of semi-Mackey functors.
\end{lem}
\begin{proof}
Since pullbacks of identity maps are identity maps, it is clear that $\theta_{\underline{M}}$ commutes with restrictions. Now let $f : X \to Y$ be a map in $\Fin_G$ and let $x \in \underline{M} (X)$. We have the following.
\begin{align*}
	\theta_{\underline{M}} (t_f (x)) &= (Y \xrightarrow{=} Y \xrightarrow{=} Y, t_f (x)) \\
	                                                    &\sim (X \xrightarrow{=} X \xrightarrow{f} Y, x) \\
	                                                    &= t_f ((X \xrightarrow{=} X \xrightarrow{=} X, x)) \\
	                                                    &= t_f (\theta_{\underline{M}} (x))
\end{align*}
The second line above is equivalent to the first because the diagram below is exponential.
\begin{align*}
\xymatrix{
 X \ar[d]_-{=} \ar[rr]^-{=} && X \ar[d]^-{f} \\
 X \ar[r]_-{f} & Y \ar[r]_-{=} & Y }
\end{align*}
Thus, $\theta_{\underline{M}}$ commutes with transfers as well.
\end{proof}

Finally we can show that $s\TT (\underline{M})$ is a free semi-Tambara functor on $\underline{M}$.

\begin{thm}\label{thm:freesemitamb}
Let $\underline{M}$ be a semi-Mackey functor. For any semi-Tambara functor $\underline{R}$, the function shown below is a bijection.
\begin{align*}
	Hom_{sTamb(G)} (s\TT (\underline{M}), \underline{R}) &\to Hom_{sMack(G)} (\underline{M}, \underline{R}) \\
	                                                                                     F &\mapsto F \circ \theta_{\underline{M}}
\end{align*}
That is, $s\TT (\underline{M})$ is the free semi-Tambara functor on $\underline{M}$.
\end{thm}
\begin{proof}
Firstly, for any $X \in \Fin_G$ it is readily verified that
\begin{align*}
	(U \xrightarrow{i} V \xrightarrow{j} X, u) = t_j n_i (\theta_{\underline{M}} (u))
\end{align*}
in $s\TT (\underline{M})$, so the map is injective. Now suppose we are given a map $F_0 : \underline{M} \to \underline{R}$ of semi-Mackey functors. For any $X \in \Fin_G$ we define a function as below.
\begin{align*}
	F(X) : s\TT_0 (\underline{M}) (X) &\to \underline{R} (X) \\
	(U \xrightarrow{i} V \xrightarrow{j} X, u) &\mapsto t_j n_i (F_0 (u))
\end{align*}
It is easy to see that this is a well-defined function, and that these form a map
\begin{align*}
	F : s\TT_0 (\underline{M}) \to \underline{R}
\end{align*}
of semi-Tambara functors. Finally, one easily checks that $F$ descends to a map on the quotient $s\TT (\underline{M})$ and that $F \circ \theta_{\underline{M}} = F_0$.
\end{proof}

Next, we decompose $s\TT (\underline{M})$ as a semi-Mackey functor. For any finite $G$-set $X$ and $n \geq 0$, let $s\TT^n_0 (\underline{M}) (X)$ denote the set of equivalence classes in $s\TT_0 (\underline{M}) (X)$ represented by $(U \xrightarrow{i} V \xrightarrow{j} X, u)$ such that $i^{-1} (v)$ has exactly $n$ elements for each $v \in V$. Since pullbacks preserve fibers, these sets form a sub-semi-Mackey functor of $s\TT_0 (\underline{M})$. Now any map of $G$-sets $i : U \to V$ can be decomposed as
\begin{align*}
	\textstyle \coprod_{n \geq 0} i_n : \coprod_n U^n \to \coprod_n V^n,
\end{align*}

where $V^n \defeq \{ v \in V : \# i^{-1} (v) = n \}$ and $U^n \defeq i^{-1} (V^n)$, and this decomposition is isomorphism invariant, so we get a direct sum decomposition as below.
\begin{align*}
	s\TT_0 (\underline{M}) = \bigoplus_{n \in \NN} s\TT^n_0 (\underline{M})
\end{align*}

Next, consider the equivalence relation on $s\TT_0 (\underline{M})$. Suppose that the diagram below is exponential.
\begin{align*}
\xymatrix{
 A \ar[d] \ar[rr]^-{f} && B \ar[d] \\
 W \ar[r]_-{k} & U \ar[r]_-{i} & V }
\end{align*}

Then $f$ is a pullback of $i$, so if all the fibers of $i$ have $n$ elements then all the fibers of $f$ have $n$ elements. Combining this with the fact that disjoint unions of exponential diagrams are exponential, we see that the equivalence relation on $s\TT_0 (\underline{M})$ is the direct sum of its restrictions to the $s\TT^n_0 (\underline{M})$. Thus we get another direct sum decomposition, as below.
\begin{align*}
	s\TT (\underline{M}) = \bigoplus_{n \in \NN} s\TT^n (\underline{M})
\end{align*}

Here, $s\TT^n (\underline{M})$ consists of the elements of $s\TT (\underline{M})$ that are represented by pairs $(U \xrightarrow{i} V \xrightarrow{j} X, u)$ such that every fiber of $i$ has $n$ elements. The $n = 0$ and $n = 1$ components are easy to identify. For the statement below, we denote by $s\underline{A}$ the Burnside semi-Mackey functor, and note that it is the initial semi-Tambara functor.

\begin{prop}\label{prop:freestamb01}
The unique map $s\underline{A} \to s\TT (\underline{M})$ in $sTamb(G)$ induces an isomorphism
\begin{align*}
	s\underline{A} \xrightarrow{\cong} s\TT^0 (\underline{M}).
\end{align*}
The universal map $\theta_{\underline{M}} : \underline{M} \to s\TT (\underline{M})$ induces an isomorphism
\begin{align*}
	\underline{M} \xrightarrow{\cong} s\TT^1 (\underline{M}).
\end{align*}
\end{prop}
\begin{proof}
Firstly, the unique map $\iota : s\underline{A} \to s\TT (\underline{M})$ in $sTamb(G)$ is a map of commutative semi-Green functors, so it preserves multiplicative units. Now the multiplicative unit in $s\TT (\underline{M}) (X)$ is $(\emptyset \to X \xrightarrow{=} X, 0)$, so $\iota$ maps the span $Y \xleftarrow{f} X \to \ast$ to
\begin{align*}
	t_f ((\emptyset \to X \xrightarrow{=} X, 0)) = (\emptyset \to X \xrightarrow{f} Y, 0).
\end{align*}
It is also clear that $s\TT^0 (\underline{M}) (Y)$ consists of the elements represented by the pairs $(\emptyset \to X \xrightarrow{f} Y, 0)$, and that two of these are equivalent if and only if their maps $f : X \to Y$ are isomorphic objects of $\Fin_G / Y$. The first part now follows by inspection.\\
\indent Now consider $s\TT^1 (\underline{M})$. All of the fibers of a map $i : U \to V$ in $\Fin_G$ have $1$ element exactly when $i$ is an isomorphism, and in this case $i$ identifies $U$ with $V$. Thus we may describe $s\TT^1_0 (\underline{M}) (X)$ as the set of isomorphism classes of pairs $(V \xrightarrow{j} X, v \in \underline{M} (V))$. Now diagrams of the form below are exponential,
\begin{align*}
\xymatrix{
 W \ar[d]_-{=} \ar[rr]^-{=} && W \ar[d]^-{k} \\
 W \ar[r]_-{k} & V \ar[r]_-{=} & V }
\end{align*}
so the equivalence relation on $s\TT^1_0 (\underline{M}) (X)$ is generated by
\begin{align*}
	(V \xrightarrow{j} X, t_k (x)) \sim (W \xrightarrow{j \circ k} X, x).
\end{align*}
The equivalence relation \emph{defining} $s\TT^1_0 (\underline{M}) (X)$ is a special case of the above, where we restrict ourselves to \emph{isomorphisms} $k : W \to V$. We can now identify $s\TT^1 (\underline{M}) (X)$ as below.
\begin{align*}
	s\TT^1 (\underline{M}) (X) = \colim_{(V \to X) \in \Fin_G / X} \underline{M}_* (V)
\end{align*}
Since $Id_X$ is the terminal object of $\Fin_G / X$, we have an isomorphism as below.
\begin{align*}
	\colim_{f : V \to X} \underline{M}_* (f) : s\TT^1 (\underline{M}) (X) \xrightarrow{\cong} \underline{M} (X)
\end{align*}
Examining this map, we see that it sends $(X \xrightarrow{=} X \xrightarrow{=} X, x)$ to $x$, so its composite with $\theta_{\underline{M}} (X)$ is the identity.
\end{proof}

\indent Next, suppose that $\underline{M}$ is a Mackey functor and $\underline{R}$ is a Tambara functor. Then we have the following.
\begin{align*}
	Hom_{Mack(G)} (\underline{M}, \underline{R}) &= Hom_{sMack(G)} (\underline{M}, \underline{R}) \\
	                                                                            &\cong Hom_{sTamb(G)} (s\TT (\underline{M}), \underline{R})
\end{align*}

Thus, defining $\TT (\underline{M})$ to be the additive completion of $s\TT (\underline{M})$ and letting
\begin{align*}
	\theta_{\underline{M}} : \underline{M} \to \TT (\underline{M})
\end{align*}

be the composite of $\theta_{\underline{M}} : \underline{M} \to s\TT (\underline{M})$ with the completion map, we have the following.

\begin{cor}\label{cor:algfreetamb}
Let $\underline{M}$ be a Mackey functor. For any Tambara functor $\underline{R}$ the function shown below is a bijection.
\begin{align*}
	Hom_{Tamb(G)} (\TT (\underline{M}), \underline{R}) &\to Hom_{Mack(G)} (\underline{M}, \underline{R}) \\
	                                                                                  F &\mapsto F \circ \theta_{\underline{M}}
\end{align*}
That is, $\TT (\underline{M})$ is the free Tambara functor on $\underline{M}$.
\end{cor}

Now let $\TT^n (\underline{M})$ be the additive completion of $s\TT^n (\underline{M})$. We obtain a direct sum decomposition as below.
\begin{align*}
	\TT (\underline{M}) = \bigoplus_{n \in \NN} \TT^n (\underline{M})
\end{align*}

The following is immediate from Proposition~\ref{prop:freestamb01}.

\begin{cor}\label{cor:freetamb01}
The unique map $\underline{A} \to \TT (\underline{M})$ in $Tamb(G)$ induces an isomorphism
\begin{align*}
	\underline{A} \xrightarrow{\cong} \TT^0 (\underline{M}).
\end{align*}
The universal map $\theta_{\underline{M}} : \underline{M} \to \TT (\underline{M})$ induces an isomorphism
\begin{align*}
	\underline{M} \xrightarrow{\cong} \TT^1 (\underline{M}).
\end{align*}
\end{cor}

We our now in a position to identify our symmetric powers $Sym_n$. Letting $\CC$ denote the free commutative ring spectrum functor, and taking $\HH\underline{M}$ to be positive cofibrant, we have the following.
\begin{align*}
	\CC (\HH\underline{M}) = S \vee \HH\underline{M} \vee (\HH\underline{M})^{\wedge 2}/\Sigma_2 \vee ...
\end{align*}

The inclusion of the wedge summand $\HH\underline{M}$ induces a map
\begin{align*}
	\underline{M} \to \underline{\pi}_0 \CC (\HH\underline{M})
\end{align*}

of Mackey functors, which then induces a map of Tambara functors
\begin{align*}
	\psi_{\underline{M}} : \TT (\underline{M}) \to \underline{\pi}_0 \CC (\HH\underline{M}).
\end{align*}

Corollary~\ref{cor:freetambara} says that $\psi_{\underline{M}}$ is an isomorphism. We need one more simple fact. In the following, note that the wedge sum decomposition of $\CC (\HH\underline{M})$ induces a direct sum decomposition of $\underline{\pi}_0 \CC (\HH\underline{M})$.

\begin{lem}\label{lem:gradedmap}
The map $\psi_{\underline{M}}$ is a graded map of Mackey functors.
\end{lem}
\begin{proof}
Any element of $\TT^n (\underline{M}) (X)$ is a difference of transfers of norms of elements in the image of $\underline{M}$, where the norms are taken along maps $U \to V$ such that all the fibers have $n$ elements. Now $\CC (\HH\underline{M})$ is a \emph{graded} ring spectrum with $\HH\underline{M}$ in degree $1$, so it is clear from the definition (see~\cite{Stri}) that such norms end up in the $n$'th wedge summand.
\end{proof}

The following is immediate.

\begin{cor}\label{cor:identsympow}
Let $n \geq 2$. There is a natural isomorphism of Mackey functors
\begin{align*}
	\psi_{\underline{M}}^n : \TT^n (\underline{M}) \xrightarrow{\cong} Sym_n (\underline{M}).
\end{align*}
Thus, for any finite $G$-set $X$, $Sym_n (\underline{M}) (X)$ is the quotient of the free abelian group on the pairs $(U \xrightarrow{i} V \xrightarrow{j} X, u)$, with $i$ and $j$ maps of finite $G$-sets and $u \in \underline{M} (U)$ such that $i$ has degree $n$, by the relations
\begin{enumerate}[(i)]
\item $(U \xrightarrow{i} V \xrightarrow{j} X, u) = (U' \xrightarrow{i'} V' \xrightarrow{j'} X, u')$ whenever there is a commutative diagram
\begin{align*}
\xymatrix{
 U \ar[r]^-{i} & V \ar[r]^-{j} & X \\
 U' \ar[u]^-{f}_-{\cong} \ar[r]_-{i'} & V' \ar[u]^-{g}_-{\cong} \ar[ru]_-{j'} & }
\end{align*}
such that $f$ and $g$ are isomorphisms and $r_f (u) = u'$,
\item $(U_1 \textstyle \coprod U_2 \xrightarrow{i_1 \coprod i_2} V_1 \coprod V_2 \xrightarrow{j_1 \coprod j_2} X, (u_1, u_2)) =$ \\
$(U_1 \xrightarrow{i_1} V_1 \xrightarrow{j_1} X, u_1) + (U_2 \xrightarrow{i_2} V_2 \xrightarrow{j_2} X, u_2)$, and
\item $(U \xrightarrow{i} V \xrightarrow{j} X, t_k (w)) = (A \xrightarrow{g} B \xrightarrow{j \circ h} X, r_f (w))$ whenever the diagram
\begin{align*}
\xymatrix{
 W \ar[r]^-{k} & U \ar[r]^-{i} & V \ar[r]^-{j} & X \\
 A \ar[u]^-{f} \ar[rr]_-{g} && B \ar[u]_-{h} \ar[ur]_-{j \circ h} & }
\end{align*}
is commutative and the rectangle is exponential.
\end{enumerate}
Transfers are determined on the generators by composition, while restrictions are determined by pullback.
\end{cor}

\begin{cor}\label{cor:sympowsomecolim}
For each $n$ the functor $Sym_n$ preserves direct limits and reflexive coequalizers.
\end{cor}
\begin{proof}
We know that $Sym_n$ is a retract of $\TT$. Since Tambara functors are closed under direct limits and reflexive coequalizers (of Mackey functors), it is formal that $\TT$ preserves these.
\end{proof}

Next we point out that these same symmetric powers of Mackey functors are obtained by arbitrary $(-1)$-connected spectra.

\begin{prop}\label{prop:-1connsamesympow}
Let $n \geq 2$. For $(-1)$-connected spectra $X$, $\underline{\pi}_0$ of the derived $n$'th symmetric power $X^{\wedge n} / \Sigma_n$ is naturally isomorphic to $Sym_n (\underline{\pi}_0 X)$.
\end{prop}
\begin{proof}
Let $f : X \to Y$ be a map of positive cofibrant, $(-1)$-connected spectra that induces an isomorphism on $\underline{\pi}_0$. Now $\CC (X)$ and $\CC (Y)$ are $(-1)$-connected, so for any Tambara functor $\underline{R}$ we have isomorphisms
\begin{align*}
	Hom_{Tamb(G)} (\underline{\pi}_0 \CC (X), \underline{R}) &\cong Hom_{Ho(comm_G)} (\CC(X), \HH\underline{R}) \\
	                                                                                              &\cong Hom_{Ho(Sp_G)} (X, \HH\underline{R}) \\
	                                                                                              &\cong Hom_{Mack(G)} (\underline{\pi}_0 X, \underline{R})
\end{align*}
by Theorems~\ref{thm:mapstoemcomm} and~\ref{thm:homemcomm}, and similarly for $Y$. It follows that $\CC (f)$ induces an isomorphism on $\underline{\pi}_0$. The result follows by applying this to positive cofibrant models for the natural maps $X \to Post^0 X$.
\end{proof}

Finally, we can identify the geometric fixed points of the free Tambara functor on a Mackey functor $\underline{M}$. We use the notation $\Phi^G \underline{N}$ for the geometric fixed points of $\underline{N}$, and for any subgroup $H$ we use $\Phi^H \underline{N}$ to denote $\Phi^H \Res_H^G \underline{N}$. This is given by
\begin{align*}
	\Phi^H \underline{N} = \underline{N} (G/H) / \big( \sum_{K \subsetneq H} t_K^H (\underline{N} (G/K)) \big).
\end{align*}

Recall that these functors are symmetric monoidal. For any subgroup $H$ we may define a homomorphism as below. Here $W(H)$ denotes the Weyl group $Aut_G(G/H)$ of $H$ in $G$.
\begin{align*}
	\Xi_{\underline{M}} (H) : (\Phi^H \underline{M}) / W(H) &\to \Phi^G \TT^{|G/H|} (\underline{M}) \\
	                                                                                     [x] &\mapsto [n_H^G \theta_{\underline{M}} (x)]
\end{align*}

The fact that this formula induces a homomorphism on $\underline{M} (G/H)$ follows from the distributive law: $n_H^G (a + b)$ is equal to $n_H^G a + n_H^G b$ modulo transfers. The distributive law also implies that $n_H^G z$ is a sum of transfers when $z$ is. The element $n_H^G \theta_{\underline{M}} (x)$ is represented by $(G/H \to \ast \xrightarrow{=} \ast, x)$, so it is in $\TT^{|G/H|} (\underline{M})$. Finally we note that, for any $f \in W(H)$, $(G/H \to \ast \xrightarrow{=} \ast, x)$ and $(G/H \to \ast \xrightarrow{=} \ast, r_f (x))$ are equivalent under the relation which defines $s\TT_0 (\underline{M}) (G/G)$. \\
\indent Now let $H_1, ..., H_r$ be a list of subgroups containing exactly one from each conjugacy class. The maps $\Xi_{\underline{M}} (H)$ determine a homomorphism of graded rings
\begin{align*}
	\Xi_{\underline{M}} : Sym\big(\bigoplus_{i = 1}^{r} (\Phi^{H_i} \underline{M}) / W(H_i)\big) \to \Phi^G \TT (\underline{M}),
\end{align*}

where $Sym(\spacedash)$ denotes the symmetric algebra functor and the piece $(\Phi^{H_i} \underline{M}) / W(H_i)$ is in degree $|G/H_i|$. We have the following.

\begin{thm}\label{thm:geomfixpfreetamb}
The map $\Xi_{\underline{M}}$ is an isomorphism of graded rings. That is, the geometric fixed points of the free Tambara functor on $\underline{M}$ is the free symmetric algebra on the geometric fixed points of $\underline{M}$ with respect to all subgroups, quotiented by the actions of their Weyl groups.
\end{thm}
\begin{proof}
First we show that $\Xi_{\underline{M}}$ is surjective. A typical element of $\TT (\underline{M}) (G/G)$ is a difference of pairs $(U \to V \to \ast, u)$. Decomposing the $V$'s into orbits to obtain sums, we see that it suffices to show that $\Xi_{\underline{M}}$ hits the equivalence classes of the pairs with $V$ an orbit. Now if $V$ is not the trivial orbit, this pair is a transfer, and so represents zero. Thus we consider pairs $(U \to \ast \xrightarrow{=} \ast, u)$. We may express $U$ as a union
\begin{align*}
	U \cong \coprod_{j = 1}^m G/H_{i_j}
\end{align*}
and $u$ as a product of elements $u = (u_1, ... , u_m)$ with $u_j \in \underline{M} (G/H_{i_j})$. We then have
\begin{align*}
	[(U \to \ast \xrightarrow{=} \ast, u)] &= \prod_{j = 1}^m [n_{H_{i_j}}^G \theta_{\underline{M}} (u_j)] \\
	                                       &= \Xi_{\underline{M}} (\otimes_{j=1}^m [u_j]),
\end{align*}
so we see that $\Xi_{\underline{M}}$ is surjective.\\
\indent We now show that $\Xi_{\underline{M}}$ is injective by constructing a left inverse $\xi_{\underline{M}}$. Consider a pair $(U \xrightarrow{i} V \to \ast, u)$. Decompose $i^{-1} (V^G)$ into orbits to obtain
\begin{align*}
	i^{-1} (V^G) \cong \coprod_{v \in V^G} \coprod_{j = 1}^{m_v} G/H_{i_{v,j}}
\end{align*}
and $u = \prod_v (u_{v,1}, ... , u_{v,m_v})$ with $u_{v,j} \in \underline{M} (G/H_{i_{v,j}})$. Note that for each $v \in V^G$ the $u_{v,j}$'s are unique up to permutations within the various $\underline{M} (G/H_i)$ and the actions by the corresponding Weyl groups $W(H_i)$. Thus we define a homomorphism as below.
\begin{align*}
	\xi_{\underline{M}} : s\TT_0 (\underline{M}) (G/G) &\to Sym\big(\bigoplus_{i = 1}^{r} (\Phi^{H_i} \underline{M}) / W(H_i)\big)\\
	                              (U \xrightarrow{i} V \to \ast, u) &\mapsto \sum_{v \in V^G} \otimes_{j = 1}^{m_v} [u_{v,j}]
\end{align*}
Now consider the equivalence relation on $s\TT_0 (\underline{M}) (G/G)$. Suppose we are given a diagram as below, where the rectangle is exponential, and take $u = t_k (w)$ with $w \in \underline{M} (W)$.
\begin{align*}
\xymatrix{
 A \ar[d]_-{e} \ar[rr]^-{\pi} && B \ar[d]_-{p} \ar[dr] & \\
 W \ar[r]_-{k} & U \ar[r]_-{i} & V \ar[r] & \ast }
\end{align*}
Points of $B^G$ correspond to elements $v \in V^G$ and equivariant sections $s : i^{-1} (v) \to W$ of $k$. Of course, a map of orbits has an equivariant section if and only if it is an isomorphism, in which case it has exactly one equivariant section. Let $W_{iso}$ denote the union of the orbits of $W$ that are mapped injectively by $k$ into $U$. Now we can choose isomorphisms successively as below,
\begin{align*}
	i^{-1} (V^G) &\cong \coprod_{v \in V^G} \coprod_{j = 1}^{m_v} G/H_{i_{v,j}} \\
	k^{-1} i^{-1} (V^G) \cap W_{iso} &\cong \coprod_{v \in V^G} \coprod_{j = 1}^{m_v} \coprod_{l=1}^{n_{v,j}} G/H_{i_{v,j}}
\end{align*}
so that $k : k^{-1} i^{-1} (V^G) \cap W_{iso} \to i^{-1} (V^G)$ becomes the appropriate fold map. Meanwhile, each orbit of $k^{-1} i^{-1} (V^G)$ that is not in $W_{iso}$ contributes a transfer to one of the $u_{v,j}$'s, so it does not affect the equivalence classes of the $u_{v,j}$'s in the appropriate geometric fixed points. The fact that $\xi_{\underline{M}}$ descends to a map on $s\TT (\underline{M}) (G/G)$ now follows from the fact that the expression
\begin{align*}
	\otimes_{j = 1}^{m_v} [u_{v,j}]
\end{align*}
is distributive in each variable $u_{v,j}$. Taking the additive completion, we get a map on $\TT (\underline{M}) (G/G)$. Finally, it is clear that this map descends to $\Phi^G \TT (\underline{M})$, since transfers of elements of $\TT (\underline{M})$ from proper subgroups are differences of pairs $(U \to V \to \ast, u)$ such that $V^G = \emptyset$. To show that $\xi_{\underline{M}} \circ \Xi_{\underline{M}}$ is the identity, it suffices to check this on elements of the form $\otimes_{j=1}^m [u_j]$, with $u_j \in \underline{M} (G/H_{i_j})$. We have
\begin{align*}
	\Xi_{\underline{M}} (\otimes_{j=1}^m [u_j]) = [(\textstyle \coprod_j G/H_{i_j} \to \ast \xrightarrow{=} \ast, (u_1, ..., u_m))]
\end{align*}
and
\begin{align*}
	\xi_{\underline{M}} ([(\textstyle \coprod_j G/H_{i_j} \to \ast \xrightarrow{=} \ast, (u_1, ..., u_m))]) = \otimes_{j=1}^m [u_j],
\end{align*}
by definition.
\end{proof}

Combining this with Corollary~\ref{cor:identsympow}, we obtain the following.

\begin{cor}\label{cor:geomfixpsympow}
Let $n \geq 2$, and let $H_1, ..., H_r$ be a list of subgroups of $G$ containing exactly one from each conjugacy class. Then there is a natural isomorphism
\begin{align*}
	\Phi^G (\psi_{\underline{M}}^n) \circ \Xi_{\underline{M}}^n : Sym^n \big(\bigoplus_{i = 1}^{r} (\Phi^{H_i} \underline{M}) / W(H_i)\big) \xrightarrow{\cong} \Phi^G Sym_n (\underline{M}).
\end{align*}
Here, $Sym^n$ denotes the $n$'th degree part of the symmetric algebra functor, and we take the summand $(\Phi^{H_i} \underline{M}) / W(H_i)$ to have degree $|G/H_i|$.
\end{cor}

\indent \emph{Remark:} The conjecture that there is an isomorphism such as $\Xi_{\underline{M}}$ is due to Mike Hill. It was, in fact, inferred topologically by inspecting the structure of equivariant extended powers.

\section{Norms of Mackey Functors}\label{sec:norm}

In this section we give an algebraic description of the norm construction on Mackey functors. Let $G$ be a finite group, and let $H$ be a proper subgroup. We describe the norm functor $N_H^G : Sp_H \to Sp_G$ of~\cite{HHR} as follows. Let $\{ g_i \}$ be a set of coset representatives for $G/H$, with $1$ being the representative of the identity coset. As an ordinary orthogonal spectrum we set
\begin{align*}
	N_H^G X \defeq \bigwedge_{G/H} X.
\end{align*}

Fix $g \in G$, and define $g_{j_i}$ by $g^{-1} \cdot g_i H = g_{j_i} H$. Then we indicate the $G$-action on $N_H^G X$ schematically by
\begin{align*}
	g \cdot (\wedge_{g_i H} x_i) = \wedge_{g_i H}  (g_i^{-1} g g_{j_i}) \cdot x_{j_i}.
\end{align*}

This functor is clearly symmetric monoidal. In Section I.5 of \cite{Ull} it is shown that this functor preserves cofibrancy and weak equivalences between cofibrant objects. It follows that $N_H^G$ is left derivable, so we can define a norm construction on Mackey functors as below.
\begin{align*}
	N_H^G : Mack(H) &\to Mack(G) \\
	      \underline{M} &\mapsto \underline{\pi}_0 (N_H^G \HH\underline{M})
\end{align*}

Next, Corollary~I.5.8 of~\cite{Ull} implies that if $X$ is $(-1)$-connected, then the derived norm $N_H^G X$ is too. Furthermore, application of Theorem~I.5.9 of~\cite{Ull} to the cofiber sequences $Post_1 Z \to Z \to Post^0 Z$ yields the following.

\begin{lem}\label{lem:norm-1conn}
If $f : X \to Y$ is a map of cofibrant, $(-1)$-connected spectra which induces an isomorphism on $\underline{\pi}_0$, then $N_H^G (f)$ induces an isomorphism on $\underline{\pi}_0$.
\end{lem}

\begin{cor}\label{cor:norm-1der}
For cofibrant, $(-1)$-connected spectra $X$ there is a natural isomorphism
\begin{align*}
	\underline{\pi}_0 (N_H^G X) \xrightarrow{\cong} N_H^G (\underline{\pi}_0 X).
\end{align*}
\end{cor}

It follows from this that we have isomorphisms as below.
\begin{align*}
	N_H^G (\underline{M}_1 \otimes \underline{M}_2) &= \underline{\pi}_0 N_H^G (\HH(\underline{M}_1 \otimes \underline{M}_2)) \\
	                                                                                  &\cong \underline{\pi}_0 N_H^G (\HH\underline{M}_1 \wedge \HH\underline{M}_2) \\
	                                                                                  &\cong \underline{\pi}_0 (N_H^G \HH\underline{M}_1 \wedge N_H^G \HH\underline{M}_2) \\
	                                                                                  &\cong (\underline{\pi}_0 N_H^G \HH\underline{M}_1) \otimes (\underline{\pi}_0 N_H^G \HH\underline{M}_2)
\end{align*}

Thus, the norm construction on Mackey functors is symmetric monoidal. Next, one easily calculates that
\begin{align*}
	N_H^G (\Sigma^{\infty} T_+) \cong \Sigma^{\infty} (N_H^G T)_+
\end{align*}

for any finite $H$-set $T$, where $N_H^G T$ denotes $\bigtimes_{G/H} T$ with $G$-action as in the definition of the spectrum level norm. Thus we have
\begin{align*}
	N_H^G ([\spacedash, T]) \cong [\spacedash, N_H^G T].
\end{align*}

We shall need the following.

\begin{lem}\label{lem:normpresurj}
If $f : \underline{M}_1 \to \underline{M}_2$ is a surjective map in $Mack(H)$ then
\begin{align*}
	N_H^G (f) : N_H^G \underline{M}_1 \to N_H^G \underline{M}_2
\end{align*}
is a surjective map in $Mack(G)$.
\end{lem}
\begin{proof}
Letting $\underline{K}$ be the kernel of $f$, we have a cofiber sequence
\begin{align*}
	\HH\underline{K} \to \HH\underline{M}_1 \xrightarrow{\HH(f)} \HH\underline{M}_2
\end{align*}
in $Ho(Sp_H)$. Theorem I.5.9 of~\cite{Ull} now implies that the cofiber of $N_H^G (\HH(f))$ is $0$-connected.
\end{proof}

Now let $\underline{M} \in Mack(H)$, and let $\HH\underline{M}$ be a positive cofibrant model for the corresponding Eilenberg MacLane spectrum. Recalling that the free commutative ring spectrum on $\HH\underline{M}$ is as below,
\begin{align*}
	\CC (\HH\underline{M}) = S \vee \HH\underline{M} \vee (\HH\underline{M})^{\wedge 2}/\Sigma_2 \vee ...
\end{align*}

we have a coretraction
\begin{align*}
	\iota_{\underline{M}} : \HH\underline{M} \to \CC(\HH\underline{M}).
\end{align*}

Applying $N_H^G$, we obtain a coretraction
\begin{align*}
	N_H^G (\iota_{\underline{M}}) : N_H^G \HH\underline{M} \to N_H^G \CC (\HH\underline{M}).
\end{align*}

Now $N_H^G$ on $comm_H$ is left adjoint to the restriction functor
\begin{align*}
	\Res_H^G : comm_G \to comm_H,
\end{align*}

so we have
\begin{align*}
	N_H^G \CC(\HH\underline{M}) \cong \CC(G_+ \wedge_H \HH\underline{M}).
\end{align*}

Applying the functor $\underline{\pi}_0$ to $N_H^G (\iota_{\underline{M}})$ and applying Corollary~\ref{cor:freetambara}, we obtain another coretraction
\begin{align*}
	\iota_{\underline{M}}^{G,H} : N_H^G \underline{M} \to \TT (\Ind_H^G \underline{M}).
\end{align*}

Our task is thus to identify the image of $\iota_{\underline{M}}^{G,H}$. Note that the existence of this retraction implies the following, since $\TT$ and $\Ind_H^G$ preserve direct limits and reflexive coequalizers.

\begin{lem}\label{lem:normsomecolim}
The norm construction on Mackey functors commutes with direct limits and reflexive coequalizers.
\end{lem}

For $X \in \Fin_G$ let $\TT^{G,H} \underline{M} (X)$ denote the subset of $\TT (\Ind_H^G \underline{M}) (X)$ whose elements are differences of pairs $(G/H \times V \xrightarrow{\pi_2} V \xrightarrow{j} X, u)$ such that $u$ is of the form $u = (u_1, 0)$ under the correspondence below.
\begin{align*}
	\Ind_H^G \underline{M} (G/H {\times} V) \cong \underline{M} (\Res_H^G V) \bigtimes \underline{M} \big((\Res_H^G (G/H) {-} H) {\times} \Res_H^G V\big)
\end{align*}

This is equivalent to the corresponding map
\begin{align*}
	G_+ \wedge_H (\Sigma^{\infty} (\Res_H^G V)_+) \cong \Sigma^{\infty} (G/H \times V)_+ \xrightarrow{u} G_+ \wedge_H \HH\underline{M}
\end{align*}

in $Ho(Sp_G)$ being of the form $G_+ \wedge_H (u_1)$. We shall write such elements as $u = G \times_H u_1$. It is clear that $\TT^{G,H} \underline{M}$ is closed under transfers. That $\TT^{G,H} \underline{M}$ is closed under restrictions follows from the fact that, for any map of finite $G$-sets $f : W \to V$, the diagram
\begin{align*}
\xymatrix{
 G/H \times V \ar[r]^-{\pi_2} & V \\
 G/H \times W \ar[u]^-{1 \times f} \ar[r]_-{\pi_2} & W \ar[u]_-{f} }
\end{align*}

is a pullback. We can now identify the image of $\iota_{\underline{M}}^{G,H}$.

\begin{thm}\label{thm:normimage}
The image of $\iota_{\underline{M}}^{G,H}$ is $\TT^{G,H} \underline{M}$. That is, we have a natural isomorphism
\begin{align*}
	\iota_{\underline{M}}^{G,H} : N_H^G \underline{M} \xrightarrow{\cong} \TT^{G,H} \underline{M}.
\end{align*}
\end{thm}
\begin{proof}
First we show that the image of $\iota_{\underline{M}}^{G,H}$ is contained in $\TT^{G,H} \underline{M}$. We begin by choosing a collection of finite $H$-sets $T_{\alpha}$ and a surjection
\begin{align*}
	f : \oplus_{\alpha} [\spacedash, T_{\alpha}] \to \underline{M}.
\end{align*}
The map $N_H^G (f)$ is a surjection by Lemma~\ref{lem:normpresurj}, so it suffices to show the inclusion for $\oplus_{\alpha} [\spacedash, T_{\alpha}]$. Letting $A$ denote an arbitrary, finite subset of $\alpha$'s we define
\begin{align*}
	T_A \defeq \coprod_{\alpha \in A} T_{\alpha},
\end{align*}
so that $\oplus_{\alpha} [\spacedash, T_{\alpha}] \cong \varinjlim_A [\spacedash, T_A]$. Now the norm functor on $Mack(H)$ commutes with direct limits by Lemma~\ref{lem:normsomecolim}. Then, examining the commutative diagram below,
\begin{align*}
\xymatrix{
 \varinjlim_A N_H^G ([\spacedash, T_A]) \ar[d]_-{\cong} \ar[rr]^-{\varinjlim_A \iota^{G,H}_{[\spacedash, T_A]}} && \varinjlim_A \TT (\Ind_H^G [\spacedash, T_A]) \ar[d]^-{\cong} \\
 N_H^G (\varinjlim_A [\spacedash, T_A]) \ar[rr]_-{\iota^{G,H}_{\varinjlim_A [\spacedash, T_A]}} && \TT (\Ind_H^G (\varinjlim_A [\spacedash, T_A])) }
\end{align*}
we are reduced to the case of a represented Mackey functor $[\spacedash, T]$. We need only check that the universal element of $N_H^G ([\spacedash, T]) \cong [\spacedash, N_H^G T]$ maps into $\TT^{G,H} ([\spacedash, T])$. By Proposition~\ref{prop:-1connsamesympow} and Corollary~\ref{cor:norm-1der}, we may work with the spectra $N_H^G (F_1 S^1 \wedge T_+)$ and $\CC(F_1 S^1 \wedge (G \times_H T)_+)$, rather than $N_H^G \HH([\spacedash, T])$ and $\CC(G_+ \wedge_H \HH([\spacedash, T]))$. Here, $F_1 S^1$ denotes the free orthogonal spectrum on $S^1$ in level $\RR^1$. Denoting by $\pi_H$ the ($H$-equivariant) projection below,
\begin{align*}
	\pi_H : \Res_H^G (N_H^G T) &\to T \\
	                       \{ t_{g_i H} \} &\mapsto t_H
\end{align*}
one can now see directly that our universal element maps to the norm of the composite
\begin{gather*}
	F_1 S^1 \wedge (G/H \times N_H^G T)_+ \cong G_+ \wedge_H (F_1 S^1 \wedge \Res_H^G (N_H^G T)_+) \\
	\hspace{1cm} \xrightarrow{G_+ \wedge_H (1 \wedge {\pi_H}_+)} G_+ \wedge_H (F_1 S^1 \wedge T_+) \to \CC(G_+ \wedge_H (F_1 S^1 \wedge T_+))
\end{gather*}
along the projection
\begin{align*}
	\pi_2 : G/H \times N_H^G T \to N_H^G T.
\end{align*}
\indent It remains to show that $\iota_{\underline{M}}^{G,H}$ maps \emph{onto} $\TT^{G,H} \underline{M}$. It suffices to show that the pairs $(G/H \times V \xrightarrow{\pi_2} V \xrightarrow{j} X, u)$ with $u = G \times_H u_1$ are in the image. Such an element is $t_j$ of $(G/H \times V \xrightarrow{\pi_2} V \xrightarrow{=} V, u)$, so we may assume that $V = X$ and $j = Id_V$. Observe that we have an isomorphism of $G$-sets as below.
\begin{align*}
	N_H^G (\Res_H^G V) &\xrightarrow{\cong} V^{G/H} \\
	            \{ v_{g_i H} \} &\mapsto \{ g_i \cdot v_{g_i H} \}
\end{align*}
We then have a map
\begin{align*}
	\Delta_V : V \to N_H^G (\Res_H^G V)
\end{align*}
which is adjoint to $G/H \to \ast$ under this identification. Note that the composite
\begin{align*}
	\Res_H^G V \xrightarrow{\Res_H^G \Delta_V} \Res_H^G N_H^G (\Res_H^G V) \xrightarrow{\pi_H} \Res_H^G V
\end{align*}
is the identity on $\Res_H^G V$, and that the diagram below is a pullback.
\begin{align}\label{eq:deltapullback}
\xymatrix{
 G/H \times V \ar[d]_-{1 \times \Delta_V} \ar[r]^-{\pi_2} & V \ar[d]^-{\Delta_V} \\
 G/H \times N_H^G (\Res_H^G V) \ar[r]_-{\pi_2} & N_H^G (\Res_H^G V) }
\end{align}
Now let $u = G \times_H u_1$ with $u_1 \in \underline{M} (\Res_H^G V)$. The first part of the proof implies that the element
\begin{align*}
	\big(G/H {\times} N_H^G (\Res_H^G V) {\xrightarrow{\pi_2}} N_H^G (\Res_H^G V) {\xrightarrow{=}} N_H^G (\Res_H^G V), G {\times_H} (r_{\pi_H} (u_1))\big)
\end{align*}
is in the image of $\iota_{\underline{M}}^{G,H}$. In fact, it is the image of the universal element under the composite
\begin{align*}
	[\spacedash, N_H^G (\Res_H^G V)] \cong N_H^G ([\spacedash, \Res_H^G V]) {\xrightarrow{N_H^G ({u_1}_*)}} N_H^G \underline{M} \xrightarrow{\iota_{\underline{M}}^{G,H}} \TT (\Ind_H^G \underline{M}).
\end{align*}
Thus, $r_{\Delta_V}$ of this element is also in the image, but since the diagram~\ref{eq:deltapullback} is a pullback this is equal to
\begin{align*}
	\big(G/H \times V \xrightarrow{\pi_2} V \xrightarrow{=} V, G \times_H (r_{\Res_H^G \Delta_V} r_{\pi_H} (u_1))\big).
\end{align*}
Finally, since $\pi_H \circ \Res_H^G \Delta_V$ is the identity, we have
\begin{align*}
	G \times_H (r_{\Res_H^G \Delta_V} r_{\pi_H} (u_1)) = G \times_H u_1 = u.
\end{align*}
\end{proof}

\indent \emph{Remark:} Note that we now have a description of the norm that makes no mention of a set of coset representatives. In fact, given any two sets of coset representatives, there is a natural isomorphism between the corresponding norm functors, and this isomorphism fits into a commutative triangle with the corresponding maps $\iota_{\underline{M}}^{G,H}$.\\
\indent Next we examine the geometric fixed points of the norm, using the description given in Theorem~\ref{thm:geomfixpfreetamb}. Consider
\begin{align*}
	\Ind_H^G \underline{M} \cong \underline{\pi}_0 (G_+ \wedge_H \HH\underline{M}),
\end{align*}

and let $K$ be a subgroup of $G$. One can see topologically that the geometric fixed points $\Phi^K$ of this Mackey functor are zero unless $K$ is subconjugate to $H$. Now suppose $K$ is subconjugate to $H$. If $|K|$ is smaller than $|H|$, then $|G/K| > |G/H|$. Otherwise $K$ is conjugate to $H$. Again, one can see topologically that
\begin{align*}
	\Phi^H (\Ind_H^G \underline{M}) \cong \bigoplus_{W(H)} \Phi^H \underline{M},
\end{align*}

where the Weyl group $W(H)$ acts by permuting the summands. Hence we have
\begin{align*}
	\Phi^H (\Ind_H^G \underline{M}) / W(H) \cong \Phi^H \underline{M}.
\end{align*}

Choosing $H$ to represent its conjugacy class, Theorem~\ref{thm:geomfixpfreetamb} gives
\begin{align*}
	\Phi^G \TT (\Ind_H^G \underline{M}) \cong \ZZ \oplus \Phi^H \underline{M} \oplus ...,
\end{align*}

where $\ZZ$ is in degree zero, $\Phi^H \underline{M}$ is in degree $|G/H|$ and the other terms have higher degree. It follows from Corollary~\ref{cor:geomfixpsympow} that the summand $\Phi^H \underline{M}$ is $\Phi^G \TT^{|G/H|} (\Ind_H^G \underline{M})$. Now the composite
\begin{align*}
	\iota_{\underline{M}}^{G,H} : N_H^G \underline{M} \xrightarrow{\cong} \TT^{G,H} \underline{M} \xrightarrow{\subseteq} \TT^{|G/H|} (\Ind_H^G \underline{M})
\end{align*}

is a coretraction, so by applying the functor $\Phi^G$ we obtain another coretraction
\begin{align*}
	(\Xi^{|G/H|}_{\Ind_H^G \underline{M}})^{-1} \circ \Phi^G (\iota_{\underline{M}}^{G,H}) : \Phi^G N_H^G \underline{M} \to \Phi^H \underline{M}.
\end{align*}

We can now identify the geometric fixed points of the norm.

\begin{prop}\label{prop:geomfixpnorm}
There is a natural isomorphism
\begin{align*}
	(\Xi^{|G/H|}_{\Ind_H^G \underline{M}})^{-1} \circ \Phi^G (\iota_{\underline{M}}^{G,H}) : \Phi^G N_H^G \underline{M} \xrightarrow{\cong} \Phi^H \underline{M}.
\end{align*}
\end{prop}
\begin{proof}
We know that this map is a natural coretraction. Hence, since $\Phi^H$ preserves direct sums and cokernels, so does $\Phi^G N_H^G$. Expressing an arbitrary $\underline{M}$ as a cokernel of a map between direct sums of represented Mackey functors then reduces us to the case $\underline{M} = [\spacedash, H/K]$, where $K$ is a subgroup of $H$. We have the following.
\begin{gather*}
	\Phi^G N_H^G ([\spacedash, H/K]) \cong \Phi^G ([\spacedash, N_H^G (H/K)]) \cong \ZZ \{ (N_H^G (H/K))^G \} \\
	\Phi^H ([\spacedash, H/K]) \cong \ZZ \{ (H/K)^H \}
\end{gather*}
Now, for arbitrary $H$-sets $T$ there is a natural isomorphism
\begin{align*}
	(N_H^G T)^G \cong T^H.
\end{align*}
Hence, when $K \neq H$ both groups are zero, and when $K = H$ the map is a coretraction from $\ZZ$ to $\ZZ$, and therefore an isomorphism.
\end{proof}

\indent \emph{Remark:} One can trace through the arguments above to find the inverse isomorphism
\begin{align*}
	\Phi^H \underline{M} \xrightarrow{\cong} \Phi^G N_H^G \underline{M}.
\end{align*}
It is induced by the function indicated below (which is \emph{not} a homomorphism).
\begin{align*}
	\underline{M} (H/H) &\to N_H^G \underline{M} (G/G) \\
	(S^0 \xrightarrow{u} \HH\underline{M}) &\mapsto (S^0 \cong N_H^G S^0 \xrightarrow{N_H^G (u)} N_H^G \HH\underline{M})
\end{align*}
\indent \emph{Remark:} For arbitrary $X \in Ho(Sp_H)$ there is a natural isomorphism
\begin{align*}
	\Phi^G N_H^G (X) \cong \Phi^H X
\end{align*}
in Ho(Sp). See~\cite{HHR}.

\section{A G-Symmetric Monoidal Structure on Mackey Functors}\label{sec:gsymmon}

In this section we put a \emph{$G$-symmetric monoidal structure} on $Mack(G)$, and give an algebraic description. A $G$-symmetric monoidal structure is, roughly speaking, a symmetric monoidal structure with additional data that allows us to form iterated monoidal products where the group $G$ acts on the factors. If $\mathcal{C}$ is any symmetric monoidal category, then the category of $G$-objects in $\emph{C}$, which we denote by $G\mathcal{C}$, has a $G$-symmetric monoidal structure. Letting $T$ be an arbitrary finite $G$-set and $X \in G\mathcal{C}$, we define
\begin{align*}
	X^{\otimes T} \defeq \bigotimes_{t \in T} X,
\end{align*}

with $G$ acting by permuting the factors and simultaneously acting on each factor. Now $Sp_G$ is equivalent to the category of $G$-objects in $Sp$, so we can define a $G$-symmetric monoidal structure on $Sp_G$ by
\begin{align*}
	(X, T) \mapsto X^{\wedge T}.
\end{align*}

To analyze this homotopically, we begin by breaking $T$ down into orbits. If $T \cong \coprod_i G/H_i$, then we have a natural isomorphism
\begin{align*}
	X^T \cong \bigwedge_i X^{\wedge G/H_i}.
\end{align*}

Furthermore, by choosing sets of coset representatives $\{ g_{ij} \}$ for each of the $G/H_i$'s, we obtain natural isomorphisms which we indicate below.
\begin{align*}
	N_{H_i}^G \Res_{H_i}^G X &\xrightarrow{\cong} X^{\wedge G/H_i} \\
	\wedge_{g_{ij} H_i} x_{g_{ij} H_i} &\mapsto \wedge_{g_{ij} H_i} (g_{ij} \cdot x_{g_{ij} H_i})
\end{align*}

Combining the above two isomorphisms, we find
\begin{align*}
	X^T \cong \bigwedge_i (N_{H_i}^G \Res_{H_i}^G X).
\end{align*}

Since norm functors preserve cofibrancy and weak equivalences between cofibrant spectra, we have the following.

\begin{lem}\label{lem:gsymmonder}
For any finite $G$-set $T$ the functor
\begin{align*}
	X \mapsto X^{\wedge T}
\end{align*}
preserves cofibrancy and weak equivalences between cofibrant $G$-spectra.
\end{lem}

Lemma~\ref{lem:norm-1conn} implies the following.

\begin{cor}\label{cor:gsymmon-1conn}
If $f : X \to Y$ is a map of cofibrant, $(-1)$-connected spectra which induces an isomorphism on $\underline{\pi}_0$, then $f^{\wedge T}$ induces an isomorphism on $\underline{\pi}_0$.
\end{cor}

Hence, for any $T \in \Fin_G$ we can define a functor as below,
\begin{align*}
	(\spacedash)^{\otimes T} : Mack(G) &\to Mack(G) \\
	                                    \underline{M} &\mapsto \underline{\pi}_0 (\HH\underline{M})^{\wedge T}
\end{align*}

and if $T \cong \coprod_i G/H_i$ then we have a natural isomorphism
\begin{align*}
	\underline{M}^{\otimes T} \cong \bigotimes_i N_{H_i}^G \Res_{H_i}^G \underline{M}.
\end{align*}

We obtain the following by considering Postnikov sections $Post^0$.

\begin{cor}\label{cor:gsymmon-1der}
For cofibrant, $(-1)$-connected spectra $X$ there is a natural isomorphism
\begin{align*}
	\underline{\pi}_0 (X^{\wedge T}) \xrightarrow{\cong} (\underline{\pi}_0 X)^{\otimes T}.
\end{align*}
\end{cor}

It follows that $(\spacedash)^{\otimes T}$ preserves ordinary tensor products. Next, for any $X \in \Fin_G$ one easily calculates that
 \begin{align*}
	(\Sigma^{\infty} X_+)^{\wedge T} \cong \Sigma^{\infty} (X^T)_+.
\end{align*}

Hence, we obtain
\begin{align*}
	[\spacedash, X]^{\otimes T} \cong [\spacedash, X^T].
\end{align*}

Also, since tensor products of surjections in $Mack(G)$ are surjective, Lemma~\ref{lem:normpresurj} implies the following.

\begin{lem}\label{lem:gsymmonsurj}
If $f : \underline{M}_1 \to \underline{M}_2$ is a surjective map in $Mack(G)$ and $T$ is any finite $G$-set then
\begin{align*}
	f^{\otimes T} : \underline{M}_1^{\otimes T} \to \underline{M}_2^{\otimes T}
\end{align*}
is surjective.
\end{lem}

Now let $T \cong \coprod_i G/H_i$. Since the norm functors $N_{H_i}^G$ commute with direct limits and reflexive coequalizers by Lemma~\ref{lem:normsomecolim}, and tensor products preserve these, we obtain the following.

\begin{lem}\label{lem:gsymmondirlim}
For any finite $G$-set $T$ the functor $(\spacedash)^{\otimes T}$ commutes with direct limits and reflexive coequalizers in $Mack(G)$.
\end{lem}

Now let $\underline{M} \in Mack(G)$, and let $\HH\underline{M}$ be a positive cofibrant model for the corresponding Eilenberg MacLane spectrum. Just as in Section~\ref{sec:norm}, we have a coretraction
\begin{align*}
	\iota_{\underline{M}} : \HH\underline{M} \to \CC (\HH\underline{M}).
\end{align*}

Letting $T \in \Fin_G$ and applying the functor $(\spacedash)^{\wedge T}$, we obtain a coretraction
\begin{align*}
	\iota^{\wedge T}_{\underline{M}} : (\HH\underline{M})^{\wedge T} \to (\CC (\HH\underline{M}))^{\wedge T} \cong \CC (\HH\underline{M} \wedge T_+).
\end{align*}

Applying the functor $\underline{\pi}_0$ and applying Corollary~\ref{cor:freetambara}, we obtain another coretraction
\begin{align*}
	\iota^T_{\underline{M}} : \underline{M}^{\otimes T} \to \TT (\underline{M} \otimes [\spacedash, T]).
\end{align*}

Our task is thus to identify the image of $\iota^T_{\underline{M}}$.\\
\indent For $X \in \Fin_G$ let $\TT^T \underline{M} (X)$ denote the subset of $\TT (\underline{M} \otimes [\spacedash, T]) (X)$ whose elements are differences of pairs of the form
\begin{align*}
	(V \times T \xrightarrow{\pi_1} V \xrightarrow{j} X, u),
\end{align*}

where $u$ is in the image of the natural map below, which we shall call the \emph{slant map}. Here, $\Delta$ denotes the diagonal map of $T$.
\begin{align*}
	\slant^{V,T}_{\underline{M}} : \underline{M} (V \times T) &\to (\underline{M} \otimes [\spacedash, T]) (V \times T) \\
	                                                                                          x &\mapsto r_{1 \times \Delta} (x \otimes Id_T)
\end{align*}

These correspond topologically to maps $u : \Sigma^{\infty} (V \times T)_+ \to \HH\underline{M} \wedge T_+$ that are of the form
\begin{align*}
	\Sigma^{\infty} (V \times T)_+ \xrightarrow{\Sigma^{\infty} (1 \times \Delta)_+} \Sigma^{\infty} (V \times T)_+ \wedge T_+ \xrightarrow{x \wedge 1} \HH\underline{M} \wedge T_+
\end{align*}

in $Ho(Sp_G)$. It is clear that $\TT^T \underline{M}$ is closed under transfers. That it is closed under restrictions follows from the fact that, for any map $f : W \to V$ of finite $G$-sets, the diagram
\begin{align*}
\xymatrix{
 V \times T \ar[r]^-{\pi_1} & V \\
 W \times T \ar[u]^-{f \times 1} \ar[r]_-{\pi_1} & W \ar[u]_-{f} }
\end{align*}

is a pullback. We can now identify the image of $\iota^T_{\underline{M}}$.

\begin{thm}\label{thm:identgsymmon}
The image of $\iota^T_{\underline{M}}$ is $\TT^T \underline{M}$. That is, we have a natural isomorphism
\begin{align*}
	\iota^T_{\underline{M}} : \underline{M}^{\otimes T} \xrightarrow{\cong} \TT^T \underline{M}.
\end{align*}
\end{thm}
\begin{proof}
First we show that the image of $\iota^T_{\underline{M}}$ is contained in $\TT^T \underline{M}$. Using the same argument as in the first part of the proof of Theorem~\ref{thm:normimage}, and using Lemmas~\ref{lem:gsymmonsurj} and~\ref{lem:gsymmondirlim}, we are reduced to the case of a represented Mackey functor $[\spacedash, X]$. We need only check that the universal element of $[\spacedash, X]^{\otimes T} \cong [\spacedash, X^T]$ maps into $\TT^T ([\spacedash, X])$. By Proposition~\ref{prop:-1connsamesympow} and Corollary~\ref{cor:gsymmon-1der}, we may work with the spectra $(F_1 S^1 \wedge X)^{\wedge T}$ and $\CC (F_1 S^1 \wedge (X \times T)_+)$, rather than the spectra $(\HH([\spacedash, X]))^{\wedge T}$ and $\CC (\HH([\spacedash, X]) \wedge T_+)$. Denoting by $ev_{X,T}$ the evaluation map
\begin{align*}
	ev_{X,T} : X^T \times T \to X,
\end{align*}
one can now see directly that our universal element maps to the norm of the composite
\begin{align*}
	F_1 S^1 \wedge (X^T \times T)_+ &\xrightarrow{1 \wedge ((ev_{X,T} \times 1) \circ (1 \times \Delta))_+} F_1 S^1 \wedge (X \times T)_+ \\ 
	&\to \CC (F_1 S^1 \wedge (X \times T)_+)
\end{align*}
along the projection
\begin{align*}
	\pi_1 : X^T \times T \to X^T.
\end{align*}
That is, our universal element maps to the pair
\begin{align*}
	\big(X^T \times T \xrightarrow{\pi_1} X^T \xrightarrow{=} X^T, \slant^{X^T,T}_{[\spacedash, X]} (ev_{X,T})\big).
\end{align*}
\indent It remains to show that $\iota^T_{\underline{M}}$ maps \emph{onto} $\TT^T \underline{M}$. It suffices to show that the image contains the pairs
$(V \times T \xrightarrow{\pi_1} V \xrightarrow{j} W, u)$ where $u$ is in the image of the slant map. Of course, the image is closed under transfers, so we may assume that $W = V$ and $j = Id_V$. Suppose that $u = \slant^{V,T}_{\underline{M}} (u_1)$. The element $u_1 \in \underline{M} (V \times T)$ corresponds to a unique map of Mackey functors $[\spacedash, V \times T] \to \underline{M}$. Naturality and the first part of the proof (putting $X = V \times T$) then imply that the pair
\begin{align*}
	\big((V \times T)^T \times T \xrightarrow{\pi_1} (V \times T)^T \xrightarrow{=} (V \times T)^T, \slant^{(V \times T)^T, T}_{\underline{M}} (r_{ev_{V \times T,T}} (u_1))\big)
\end{align*}
is in the image. Now let $coev_{V,T}$ denote the "coevaluation" map below.
\begin{align*}
	coev_{V,T} : V &\to (V \times T)^T \\
	                     v &\mapsto (t \mapsto (v,t))
\end{align*}
Applying $r_{coev_{V,T}}$ to the pair above, and using the fact that the diagram
\begin{align*}
\xymatrix{
 V \times T \ar[d]_-{coev_{V,T} \times 1} \ar[r]^-{\pi_1} & V \ar[d]^-{coev_{V,T}} \\
 (V \times T)^T \times T \ar[r]_-{\pi_1} & (V \times T)^T }
\end{align*}
is a pullback, we get that the pair below is in the image.
\begin{align*}
	\big(V \times T \xrightarrow{\pi_1} V \xrightarrow{=} V, r_{coev_{V,T} \times 1} \slant^{(V \times T)^T, T}_{\underline{M}} (r_{ev_{V \times T,T}} (u_1))\big)
\end{align*}
We have the identity $r_{coev_{V,T} \times 1} \circ \slant^{(V \times T)^T, T}_{\underline{M}} = \slant^{V, T}_{\underline{M}} \circ r_{coev_{V,T} \times 1}$, and one easily checks that $ev_{V \times T, T} \circ (coev_{V,T} \times 1) = Id_{V \times T}$ so that we have $r_{coev_{V,T} \times 1} r_{ev_{V \times T,T}} = 1$. Thus we have
\begin{align*}
	r_{coev_{V,T} \times 1} \slant^{(V \times T)^T, T}_{\underline{M}} (r_{ev_{V \times T,T}} (u_1)) = \slant^{V, T}_{\underline{M}} (u_1) = u.
\end{align*}
\end{proof}

\indent \emph{Remark:} The slant map has a left inverse. Any $T \in \Fin_G$ has a unique map $\pi : T \to \ast$ to the terminal object. The left inverse of the slant map is induced by the composite below.
\begin{align*}
	\underline{M} \otimes [\spacedash, T] \xrightarrow{1 \otimes [\spacedash, \pi]} \underline{M} \otimes [\spacedash, \ast] \cong \underline{M}
\end{align*}
This is due to the fact that the composite $T \xrightarrow{\Delta} T \times T \xrightarrow{\pi_1} T$ is $Id_T$.
\indent \emph{Remark:} Another $G$-symmetric monoidal structure on $Mack(G)$ is constructed in~\cite{Mazur}, where it is also shown that a Mackey functor $\underline{M}$ defines a functor
\begin{align*}
	\Fin_G &\to Mack(G) \\
	T &\mapsto \underline{M}^{\otimes T}
\end{align*}

with respect to this other $G$-symmetric monoidal structure if and only if $\underline{M}$ is a Tambara functor. It is an open question whether these two $G$-symmetric monoidal structures are the same.

\section{Intrinsic Descriptions}\label{sec:intrinsic}

\subsection{G-Symmetric Monoidal Structure}\label{subsec:gsymmon}

While the above description is complete, it is unsatisfying since it does not give intrinsic generators and relations. We now give such a description. First, for any map $f : W \to V \times T$ of finite $G$-sets, let the following diagram be exponential.
\begin{align}\label{eq:distrdef}
\xymatrix{
 D(W, f, V) \times T \ar[d]_-{e} \ar[rr]^-{\pi_1} && D(W, f, V) \ar[d]^-{p} \\
 W \ar[r]_-{f} & V \times T \ar[r]_-{\pi_1} & V }
\end{align}

We can now give our definition.

\begin{definition}\label{def:intrinsicpower}
Let $T \in \Fin_G$ and $\underline{M} \in Mack(G)$. For any finite $G$-set $X$ we define $F(T, \underline{M}) (X)$ to be the quotient of the free abelian group on the pairs $(V \xrightarrow{j} X, u \in \underline{M} (V \times T))$, where $j$ is a map in $\Fin_G$, by the relations
\begin{enumerate}[(i)]
\item $(V \xrightarrow{j} X, u) = (V' \xrightarrow{j'} X, u')$ when there is a commutative diagram
\begin{align*}
\xymatrix{
 V \ar[r]^-{j} & X \\
 V' \ar[u]^-{f}_-{\cong} \ar[ur]_-{j'} & }
\end{align*}
such that $f$ is an isomorphism and $r_{f \times 1} (u) = u'$,
\item $(V_1 \coprod V_2 \xrightarrow{j_1 \coprod j_2} X, (u_1, u_2)) = (V_1 \xrightarrow{j_1} X, u_1) + (V_2 \xrightarrow{j_2} X, u_2)$, and
\item $(V \xrightarrow{j} X, t_f (w)) = (D(W, f, V) \xrightarrow{j \circ p} X, r_e (w))$ for any $G$-map $f : W \to V \times T$, where $D(W, f, V)$, $p$ and $e$ are as in~\ref{eq:distrdef}.
\end{enumerate}
\end{definition}

We define restrictions and transfers for $F(T, \underline{M})$ as follows. Transfers are defined on the generators by composition; it is clear that these are well-defined and additive. If $f : Y \to X$ is a map in $\Fin_G$, we define $r_f ((V \xrightarrow{j} X, u))$ to be $(P \xrightarrow{k} Y, r_{q \times 1} (u))$, where the diagram
\begin{align*}
\xymatrix{
 V \ar[r]^-{j} & X \\
 P \ar[u]^-{q} \ar[r]_-{k} & Y \ar[u]_-{f} }
\end{align*}

is a pullback. Lemma~\ref{lem:distrpullback} implies that this is well-defined. To show that coproducts of $G$-sets are converted into products, one may argue as in Section~\ref{sec:sympow} with $s\TT_0$ and $s\TT$; imposing the first two relations in Definition~\ref{def:intrinsicpower} yields a Mackey functor, and then the third relation for $X = X_1 \coprod X_2$ becomes the product of the relations for $X_1$ and $X_2$. This argument uses the fact that $D(V \times T, Id, V) = V$. Hence, $F(T, \underline{M})$ is a Mackey functor.\\
\indent Next, we define a natural map $\Theta^T : F(T, \spacedash) \to \TT^T$ as below.
\begin{align*}
	\Theta^T_{\underline{M}} (X) : F(T, \underline{M}) (X) &\to \TT^T \underline{M} (X) \\
	(V \xrightarrow{j} X, u) &\mapsto (V \times T \xrightarrow{\pi_1} V \xrightarrow{j} X, \slant^{V, T}_{\underline{M}} (u))
\end{align*}

To show that this map respects relation (iii) of Definition~\ref{def:intrinsicpower}, we proceed as follows. Let $f : W \to V \times T$ be a $G$-map. Then since the diagram below is a pullback,
\begin{align*}
\xymatrix{
 W \ar[d]_-{f} \ar[r]^-{1 \times \pi_2 f} & W \times T \ar[d]^-{f \times 1} \\
 V \times T \ar[r]_-{1 \times \Delta} & V \times T \times T }
\end{align*}

we have the following for $w \in \underline{M} (W)$.
\begin{align*}
	\slant^{V, T}_{\underline{M}} (t_f (w)) &= r_{1 \times \Delta} (t_f (w) \otimes Id_T) \\
	                                                                &= r_{1 \times \Delta} t_{f \times 1} (w \otimes Id_T) \\
	                                                                &= t_f (r_{1 \times \pi_2 f} (w \otimes Id_T))
\end{align*}

It follows that in $\TT^T \underline{M}$, the element $(V {\times T} {\xrightarrow{\pi_1}} V {\xrightarrow{j}} X, \slant^{V, T}_{\underline{M}} (t_f (w)))$ is equal to $(D(W, f, V) \times T \xrightarrow{\pi_1} D(W, f, V) \xrightarrow{j \circ p} X, r_e r_{1 \times \pi_2 f} (w \otimes Id_T))$. Now one easily checks that the diagram below commutes,
\begin{align*}
\xymatrix{
 D(W, f, V) \times T \ar[d]_-{1 \times \Delta} \ar[r]^-{e} & W \ar[d]^-{1 \times \pi_2 f} \\
 D(W, f, V) \times T \times T \ar[r]_-{e \times 1} & W \times T }
\end{align*}

so we have the following.
\begin{align*}
	r_e r_{1 \times \pi_2 f} (w \otimes Id_T) &= r_{1 \times \Delta} r_{e \times 1} (w \otimes Id_T) = r_{1 \times \Delta} (r_e (w) \otimes Id_T) \\
	&= \slant^{D(W, f, V), T}_{\underline{M}} (r_e(w))
\end{align*}

It follows that $\Theta^T$ is a well-defined, natural map of Mackey functors. It is also clear that it is surjective. We shall prove the following theorem.

\begin{thm}\label{thm:intrpow}
For any $T \in \Fin_G$, the map $\Theta^T$ is an isomorphism.
\end{thm}

We shall proceed by induction on the order of $G$. First we check that $F(T, \spacedash)$ restricts appropriately. Let $H$ be a subgroup of $G$. To define $\Res_H^G F(T, \underline{M}) (X) = F(T, \underline{M}) (G \times_H X)$ we begin with the generators $(V \xrightarrow{j} G \times_H X, u \in \underline{M} (V \times T))$. Now $j$ identifies $V$ as $G \times_H V'$, where $V' = j^{-1} (X)$, so that $V \times T \cong (G \times_H V') \times T \cong G \times_H (V' \times \Res_H^G T)$. Thus, $u$ is an element of $\underline{M} (G \times_H (V' \times \Res_H^G T))$, which is equal to $\Res_H^G \underline{M} (V' \times \Res_H^G T)$. Thus we have a natural isomorphism
\begin{align*}
	\Res_H^G F(T, \underline{M}) \cong F(\Res_H^G T, \Res_H^G \underline{M}).
\end{align*}

We know topologically that $\TT^T$ must also restrict appropriately. In fact, for any $H$-set $V$ we have a commutative diagram as below.
\begin{align*}
\xymatrix{
 \underline{M} ((G \times_H V) \times T) \ar[d]_-{\cong} \ar[r]^-{\slant^{G \times_H V, T}_{\underline{M}}} & (\underline{M} \otimes [\spacedash, T]) ((G \times_H V) \times T) \ar[d]^-{\cong} \\
 \Res_H^G \underline{M} (V \times \Res_H^G T) \ar[r]_-{\slant^{V, \Res_H^G T}_{\Res_H^G \underline{M}}} & (\Res_H^G \underline{M} \otimes [\spacedash, \Res_H^G T]) (V \times \Res_H^G T) }
\end{align*}

Hence, we obtain a commutative diagram of natural maps as below.
\begin{align}\label{eq:thetarestr}
\xymatrix{
 \Res_H^G F(T, \spacedash) \ar[d]_-{\Res_H^G \Theta^T} \ar[r]^-{\cong} & F(\Res_H^G T, \Res_H^G (\spacedash)) \ar[d]^-{\Theta^{\Res_H^G T}} \\
 \Res_H^G \TT^T \ar[r]_-{\cong} & \TT^{\Res_H^G T} }
\end{align}

We now show that $\Theta^T$ is an isomorphism on representable Mackey functors.

\begin{lem}\label{lem:intrrepiso}
For any $Z \in \Fin_G$, the map $\Theta^T_{[\spacedash, Z]}$ is an isomorphism.
\end{lem}
\begin{proof}
In view of~\ref{eq:thetarestr}, we may assume inductively that $\Theta^T_{[\spacedash, Z]}$ is an isomorphism at $G/H$ for all proper subgroups $H$ of $G$. Next, recall from the proof of Theorem~\ref{thm:identgsymmon} that $\TT^T ([\spacedash, Z]) \cong [\spacedash, Z^T]$ with universal element
\begin{align*}
	(Z^T \times T \xrightarrow{\pi_1} Z^T \xrightarrow{=} Z^T, \slant^{Z^T, T}_{[\spacedash, Z]} (ev_{Z,T})).
\end{align*}
Thus, $\Theta^T_{[\spacedash, Z]}$ has a section, defined by sending the above universal element to
\begin{align*}
	(Z^T \xrightarrow{=} Z^T, ev_{Z,T}).
\end{align*}
It suffices to show that this section is surjective. Our inductive hypothesis implies that it is surjective at $G/H$ when $H \neq G$, and so it is surjective on sub-Mackey functors generated by these levels. The Five Lemma now implies that it suffices to show that the section is surjective on geometric fixed points. This is equivalent to the statement that $\Phi^G (\Theta^T_{[\spacedash, Z]})$ is an isomorphism. Hence we consider $\Phi^G F(T, \underline{M})$ for an arbitrary Mackey functor $\underline{M}$ (this will help to separate out the salient points). Any generator $(V \to \ast, u)$ with $V$ an orbit is a transfer, unless $V = \ast$. In this case $u \in \underline{M} (T)$. Breaking $T$ down into orbits to obtain $T \cong \coprod_i G/H_i$, we obtain a surjection
\begin{align*}
	\ZZ \{ \bigtimes_i \underline{M} (G/H_i) \} \to \Phi^G F(T, \underline{M}).
\end{align*}
To see that this map factors through $\otimes_i \underline{M} (G/H_i)$, we apply for each $i$ the relation coming from the "fold" map $T \coprod G/H_i \to T$. To see that it then factors through $\otimes_i \Phi^{H_i} (\underline{M})$, we apply for each $i$ and proper subgroup $K \subsetneq H_i$ the relation coming from the map below,
\begin{align*}
	\textstyle (\coprod_{j \neq i} G/H_j) \coprod G/K \xrightarrow{1 \coprod \pi} (\coprod_{j \neq i} G/H_j) \coprod G/H_i = T
\end{align*}
where $\pi : G/K \to G/H_i$ denotes the canonical projection. Hence we have a surjection
\begin{align}\label{eq:geomfixpftm}
	\bigotimes_i \Phi^{H_i} (\underline{M}) \to \Phi^G F(T, \underline{M}).
\end{align}
Now we have $\Phi^{H_i} ([\spacedash, Z]) \cong \ZZ \{ Z^{H_i} \}$, so that
\begin{align*}
	\bigotimes_i \Phi^{H_i} ([\spacedash, Z]) \cong \ZZ \{ \times_i Z^{H_i} \} \cong \ZZ \{ (Z^T)^G \}.
\end{align*}
It follows that the rank of $\Phi^G F(T, [\spacedash, Z])$ is at most $\# (Z^T)^G$, and that $\Phi^G F(T, [\spacedash, Z])$ is free if it attains this rank. Since $\Phi^G (\TT^T ([\spacedash, Z]))$ is a free abelian group of precisely this rank, the surjection $\Phi^G (\Theta^T_{[\spacedash, Z]})$ must be an isomorphism.
\end{proof}

Next, we dispose of two special cases.

\begin{lem}\label{lem:intrspeccase}
When $T = \emptyset$ we have $F(\emptyset, \spacedash) \equiv \underline{A}$. When $T \neq \emptyset$ we have $F(T, \underline{0}) = \underline{0}$.
\end{lem}
\begin{proof}
First take $T = \emptyset$. Then for any pair $(V \xrightarrow{j} X, u)$ we have $u \in \underline{M} (\emptyset) = 0$, so that $u = 0$. Examining Definition~\ref{def:intrinsicpower}, we see that relation (iii) is trivial, because for any map $f : W \to V \times T = \emptyset$ we have $W = \emptyset$, and hence $D(W, f, V) = V$. Hence the element $u$ is redundant, and $F(\emptyset, \underline{M}) (X)$ can be described as the Grothendieck group of isomorphism classes of $G$-maps $V \to X$. For the second part, take $T \neq \emptyset$ and $\underline{M} = \underline{0}$. Consider a pair $(V \xrightarrow{j} X, 0)$. We can regard this $0$ as the transfer of $0$ along $i : \emptyset \to V \times T$. Assuming $V$ is nonempty, the fibers of $\pi_1 : V \times T \to V$ are nonempty, so we have $D(\emptyset, i, V) = \emptyset$ (this also holds for $V = \emptyset$). Thus, the pair $(V \xrightarrow{j} X, 0)$ is equivalent to $(\emptyset \to X, 0)$, which is zero.
\end{proof}

The second part of the lemma implies the following.

\begin{cor}\label{cor:intrpreservezero}
When $T \neq \emptyset$ the functor $F(T, \spacedash)$ preserves zero maps.
\end{cor}

We now turn our attention to monoidal pairings. Let $T_1$ and $T_2$ be finite $G$-sets. We have a natural isomorphism $\TT^{T_1} \underline{M} \otimes \TT^{T_2} \underline{M} \to \TT^{T_1 \coprod T_2} \underline{M}$. One computes the pairing below.
\begin{gather*}
	\TT^{T_1} \underline{M} (X_1) \otimes \TT^{T_2} \underline{M} (X_2) \to \TT^{T_1 \coprod T_2} \underline{M} (X_1 \times X_2) \\
	(V_1 \times T_1 \xrightarrow{\pi_1} V_1 \xrightarrow{j_1} X_1, \slant^{V_1,T_1}_{\underline{M}} (u_1)) \otimes (V_2 \times T_2 \xrightarrow{\pi_1} V_2 \xrightarrow{j_2} X_2, \slant^{V_2,T_2}_{\underline{M}} (u_2)) \\
	\textstyle \mapsto \big(V_1 \times V_2 \times (T_1 \coprod T_2) \xrightarrow{\pi_1} V_1 \times V_2 \xrightarrow{j_1 \times j_2} X_1 \times X_2, \\ 		\slant^{V_1 \times V_2,T_1 \coprod T_2}_{\underline{M}} \big((r_{\pi_{V_1 \times T_1}} (u_1), r_{\pi_{V_2 \times T_2}} (u_2))\big)\big)
\end{gather*}

Hence, we define a pairing $F(T_1, \underline{M}) \otimes F(T_2, \underline{M}) \to F(T_1 \coprod T_2, \underline{M})$ analogously, as below.
\begin{gather*}
	F(T_1, \underline{M}) (X_1) \otimes F(T_2, \underline{M}) (X_2) \to \textstyle F(T_1 \coprod T_2, \underline{M}) (X_1 \times X_2) \\
	(V_1 \xrightarrow{j_1} X_1, u_1) \otimes (V_2 \xrightarrow{j_2} X_2, u_2) \\
	\mapsto \big(V_1 \times V_2 \xrightarrow{j_1 \times j_2} X_1 \times X_2, (r_{\pi_{V_1 \times T_1}} (u_1), r_{\pi_{V_2 \times T_2}} (u_2))\big)
\end{gather*}

We must show that this is well-defined.

\begin{prop}\label{prop:intrpairing}
There is a natural, associative and commutative system of pairings $F(T_1, \underline{M}) \otimes F(T_2, \underline{M}) \to F(T_1 \coprod T_2, \underline{M})$ making the following diagrams commute.
\begin{align*}
\xymatrix{
 F(T_1, \underline{M}) \otimes F(T_2, \underline{M}) \ar[d]_-{\Theta^{T_1}_{\underline{M}} \otimes \Theta^{T_2}_{\underline{M}}} \ar[r] & F(T_1 \coprod T_2, \underline{M}) \ar[d]^-{\Theta^{T_1 \coprod T_2}_{\underline{M}}} \\
 \TT^{T_1} \underline{M} \otimes \TT^{T_2} \underline{M} \ar[r]_-{\cong} & \TT^{T_1 \coprod T_2} \underline{M} }
\end{align*}
\end{prop}
\begin{proof}
We need only show that our pairing is compatible with relation (iii) of Definition~\ref{def:intrinsicpower}; the other parts are trivial to check. Given pairs $(V_i \xrightarrow{j_i} X_i, u_i)$ for $i = 1,2$, suppose that $u_1 = t_{f_1} (w_1)$ for some map $f_1 : W_1 \to V_1 \times T_1$. Now note that the diagram below is exponential, where we let $D_1 \defeq D(W_1, f_1, V_1)$.
\begin{align*}
\xymatrix{
 D_1 {\times} V_2 {\times} (T_1 \coprod T_2) \ar[d]^-{(e {\times} 1 \circ 1 {\times} \tau) \coprod (p {\times} 1 {\times} 1)} \ar[rrr]^-{\pi_1} &&& D_1 {\times} V_2 \ar[d]^-{p {\times} 1} \\
 W_1 {\times} V_2 \coprod V_1 {\times} V_2 {\times} T_2 \ar[rr]_-{(1 {\times} \tau \circ f_1 {\times} 1) \coprod 1} && V_1 {\times} V_2 {\times} (T_1 \coprod T_2) \ar[r]_-{\pi_1} & V_1 {\times} V_2 }
\end{align*}
Now $(r_{\pi_{V_1 \times T_1}} (u_1), r_{\pi_{V_2 \times T_2}} (u_2))$ is the transfer of $(r_{\pi_{W_1}} (w_1), r_{\pi_{V_2 \times T_2}} (u_2))$, so the above diagram tells us that
\begin{gather*}
	\big(V_1 \times V_2 \xrightarrow{j_1 \times j_2} X_1 \times X_2, (r_{\pi_{V_1 \times T_1}} (t_{f_1} (w_1)), r_{\pi_{V_2 \times T_2}} (u_2))\big) \\
	\sim \big( D_1 \times V_2 \xrightarrow{(j_1 \circ p) \times j_2} X_1 \times X_2,  (r_{\pi_{D_1 \times T_1}} (r_e (w_1)), r_{\pi_{V_2 \times T_2}} (u_2)) \big).
\end{gather*}
A symmetrical argument applies to $u_2$.
\end{proof}

Next we require an easy statement about direct limits.

\begin{lem}\label{lem:intrdirlim}
For any $T \in \Fin_G$, the functor $F(T, \spacedash)$ commutes with direct limits.
\end{lem}

Combined with Lemma~\ref{lem:gsymmondirlim} and the fact that every Mackey functor is the direct limit of its finitely generated sub-Mackey functors, this reduces Theorem~\ref{thm:intrpow} to the following lemma.

\begin{lem}\label{lem:intrpowfingen}
For any $T \in \Fin_G$ and finitely generated $\underline{M} \in Mack(G)$, the map $\Theta^T_{\underline{M}}$ is an isomorphism.
\end{lem}
\begin{proof}
By Lemma~\ref{lem:intrspeccase} we may assume that $T \neq \emptyset$. Since $\underline{M}$ is finitely generated, there is an exact sequence as below.
\begin{align*}
	[\spacedash, Y] \xrightarrow{\alpha} [\spacedash, Z] \xrightarrow{\beta} \underline{M} \to \underline{0}
\end{align*}
Now let $Y'$ and $Z'$ denote fibrant replacements for $\Sigma^{\infty} Y_+$ and $\Sigma^{\infty} Z_+$, respectively. Let $\alpha$ be represented by $\alpha' : Y' \to Z'$, and let $Z''$ denote the mapping cone of $\alpha'$. Then the inclusion $Y' \to Z''$ is a cofibration and we have $\underline{M} \cong \underline{\pi}_0 (Z'' / Y')$, so that $\underline{M}^{\otimes T} \cong \underline{\pi}_0 (Z'' / Y')^{\wedge T}$. Next, let $\partial^{i}_{Y'} {Z''}^{\wedge T}$ denote the $i$-coskeleton of ${Z''}^{\wedge T}$ with respect to $Y'$. We have a sequence of cofibrations
\begin{align*}
	{Y'}^{\wedge T} = \partial^{|T|}_{Y'} {Z''}^{\wedge T} \to \partial^{|T|-1}_{Y'} {Z''}^{\wedge T} \to ... \to \partial^{1}_{Y'} {Z''}^{\wedge T} \to {Z''}^{\wedge T}
\end{align*}
with quotients as below for $i = 1, ... , |T| - 1$, where the wedge sum is indexed.
\begin{align}\label{eq:coskelquot}
	\partial^{i}_{Y'} {Z''}^{\wedge T} / \partial^{i+1}_{Y'} {Z''}^{\wedge T} \cong \bigvee_{S \subseteq T, |S| = i} {Y'}^{\wedge S} \wedge (Z'' / Y')^{\wedge T - S}
\end{align}
Now, for $i = 1, ... , |T|$ we have maps
\begin{align*}
	\bigvee_{S \subseteq T, |S| \geq i} {Y'}^{\wedge S} \wedge {Z''}^{\wedge T - S} \to \partial^{i}_{Y'} {Z''}^{\wedge T}.
\end{align*}
Combining~\ref{eq:coskelquot} with Lemma~\ref{lem:gsymmonsurj}, one easily sees by downward induction on $i$ that these maps are surjective on $\underline{\pi}_0$. Letting $\{ S_j \}$ be a complete set of orbit representatives for the nonempty subsets of $T$, and $H_j$ the stabilizer of $S_j$, we obtain an exact sequence as below,
\begin{align*}
	\bigoplus_j {\Ind^G_{H_j}} \big( [\spacedash, {\Res^G_{H_j}} Y]^{\otimes S_j} {\otimes} [\spacedash, {\Res^G_{H_j}} Z]^{\otimes T - S_j} \big) \\
	\to [\spacedash, Z]^{\otimes T} \xrightarrow{\beta^{\otimes T}} \underline{M}^{\otimes T} \to \underline{0}
\end{align*}
where the first map on the $j$'th summand is adjoint to $(\Res^G_{H_j} \alpha)^{\otimes S_j} \otimes 1$. We now obtain the commutative square below, where we have simplified by omitting the notation for the restriction functors.
\begin{align*}
\xymatrix{
 \bigoplus_j {\Ind^G_{H_j}} \big( F(S_j, [\spacedash, Y]) {\otimes} F(T{-}S_j, [\spacedash, Z]) \big) \ar[r] \ar[d]^-{\cong}_-{\oplus_j \Ind^G_{H_j} (\Theta^{S_j}_{[\spacedash, Y]} \otimes \Theta^{T-S_j}_{[\spacedash, Z]})} & F(T, [\spacedash, Z]) \ar[d]^-{\Theta^T_{[\spacedash, Z]}}_-{\cong} \\
 \bigoplus_j {\Ind^G_{H_j}} \big( \TT^{S_j} ([\spacedash, Y]) {\otimes} \TT^{T{-}S_j} ([\spacedash, Z]) \big) \ar[r] & \TT^T ([\spacedash, Z]) }
\end{align*}
Now it is clear from the definition that $F(T, \spacedash)$ preserves surjections, so it suffices to show that the composite of the top map in the above diagram with $F(T, \beta)$ is zero. For each $j$, we check this by examining the following commutative diagram; we have again omitted the notation for the restriction functors, and have abbreviated $[\spacedash, Y]$ and $[\spacedash, Z]$ by $\overline{Y}$ and $\overline{Z}$, respectively, to further reduce clutter.
\begin{align*}
\xymatrix{
 F(S_j, \overline{Y}) {\otimes} F(T{-}S_j, \overline{Z}) \ar[dr]_-{F(S_j, 0) \otimes F(T{-}S_j, \beta) \hspace{1cm}} \ar[r]^-{F(S_j, \alpha) \otimes 1} & F(S_j, \overline{Z}) {\otimes} F(T{-}S_j, \overline{Z}) \ar[d]^-{F(S_j, \beta) \otimes F(T{-}S_j, \beta)} \ar[r] & F(T, \overline{Z}) \ar[d]^-{F(T, \beta)} \\
 & F(S_j, \underline{M}) {\otimes} F(T{-}S_j, \underline{M}) \ar[r] & F(T, \underline{M}) }
\end{align*}
The map in question is the composite along the top and right sides. Since $F(S_j, 0) = 0$ by Corollary~\ref{cor:intrpreservezero}, this composite is zero.
\end{proof}

We have completed the proof of Theorem~\ref{thm:intrpow}. Combining this with Proposition~\ref{prop:intrpairing}, we obtain the following.

\begin{cor}\label{cor:intrpairiso}
The maps $F(T_1, \underline{M}) \otimes F(T_2, \underline{M}) \to F(T_1 \coprod T_2, \underline{M})$ form an associative, commutative and unital system of natural isomorphisms. There is a natural isomorphism as below.
\begin{align*}
	F(\ast, \underline{M}) &\xrightarrow{\cong} \underline{M} \\
	(V \xrightarrow{j} X, u) &\mapsto t_j (u)
\end{align*}
\end{cor}

\begin{cor}\label{cor:geomfixpftm}
If $T \cong \coprod_i G/H_i$ then for any $\underline{M} \in Mack(G)$ the map shown below is an isomorphism.
\begin{align*}
	\mu^T_{\underline{M}} : \bigotimes_i \Phi^{H_i} (\underline{M}) &\to \Phi^G F(T, \underline{M}) \\
	\otimes_i [u_i] &\mapsto [(\ast \xrightarrow{=} \ast, \times_i u_i)]
\end{align*}
\end{cor}
\begin{proof}
This map is the same as~\ref{eq:geomfixpftm} from the proof of Lemma~\ref{lem:intrrepiso}. We know it is a surjection, and an isomorphism when $\underline{M}$ is representable. It is trivially an isomorphism when $T = \emptyset$. Suppose that $T = T_1 \coprod T_2$. Letting $T_1 \cong \coprod_i G/H_{1,i}$ and $T_2 \cong \coprod_j G/H_{2,j}$ so that $T \cong (\coprod_i G/H_{1,i}) \coprod (\coprod_j G/H_{2,j})$, one easily sees that the diagram below commutes.
\begin{align*}
\xymatrix{
 \big( \bigotimes_i \Phi^{H_i} (\underline{M}) \big) \otimes \big( \bigotimes_j \Phi^{H_j} (\underline{M}) \big) \ar[dr]_-{\mu^T_{\underline{M}}} \ar[r]^-{\mu^{T_1}_{\underline{M}} \otimes \mu^{T_2}_{\underline{M}}} & \Phi^G F(T_1, \underline{M}) \otimes \Phi^G F(T_2, \underline{M}) \ar[d]^-{\cong} \\
 & \Phi^G F(T, \underline{M}) }
\end{align*}
Thus, it suffices to prove the statement for $T \cong G/H$ an orbit. In this case we have $\Phi^G F(G/H, \underline{M}) \cong \Phi^G N_H^G (\Res_H^G \underline{M}) \cong \Phi^H \underline{M}$, so both the source and target of $\mu^T$ are right exact, additive functors. Hence we are reduced to the case where $\underline{M}$ is representable, which is handled in the proof of Lemma~\ref{lem:intrrepiso}.
\end{proof}

We can describe $\Phi^G F(T, \underline{M})$ in a canonical way, as follows. We first form the free abelian group on $\underline{M} (T)$. Then, for any $G$-map $f : W \to T$, we impose the relation below.
\begin{align*}
	[t_f (w)] = \sum_{j : T \to W, f \circ j = 1} [r_j (w)]
\end{align*}

Next, we know that there is a natural isomorphism of the form
\begin{align*}
	F(T, \underline{M}_1) \otimes F(T, \underline{M}_2) \xrightarrow{\cong} F(T, \underline{M}_1 \otimes \underline{M}_2),
\end{align*}

using the isomorphisms $\iota^T$ and $\Theta^T$. Letting $(V \xrightarrow{j} X, u) \in F(T, \underline{M})$, we compute $(\iota^T)^{-1} (\Theta^T)^{-1} (V \xrightarrow{j} X, u)$ as follows. First of all, this is $t_j$ of $(\iota^T)^{-1} (\Theta^T)^{-1} (V \xrightarrow{=} V, u)$. Let $u$ be represented by a map
\begin{align*}
	u : \Sigma^{\infty} (V \times T)_+ \to \HH\underline{M},
\end{align*}

where $\HH\underline{M}$ is cofibrant. Then $(\iota^T)^{-1} (\Theta^T)^{-1} (V \xrightarrow{=} V, u)$ is represented by the map
\begin{align*}
	\Sigma^{\infty} V_+ \to (\HH\underline{M})^{\wedge T}
\end{align*}

such that the underlying nonequivariant map on the summand corresponding to $v \in V$ is
\begin{align*}
	S^0 \cong (S^0)^{\wedge T} \xrightarrow{\wedge_{t \in T} u(v, t)} (\HH\underline{M})^{\wedge T}.
\end{align*}

We can indicate this schematically with the following.
\begin{align*}
	v \mapsto \prod_{t \in T} u(v, t)
\end{align*}

The following is now clear.

\begin{cor}\label{cor:intrtensmack}
There is an associative and commutative system of natural isomorphisms, as below.
\begin{gather*}
	F(T, \underline{M}_1) \otimes F(T, \underline{M}_2) \xrightarrow{\cong} F(T, \underline{M}_1 \otimes \underline{M}_2) \\
	(V_1 \xrightarrow{j_1} X_1, u_1) \otimes (V_2 \xrightarrow{j_2} X_2, u_2) \\
	\mapsto \big(V_1 \times V_2 \xrightarrow{j_1 \times j_2} X_1 \times X_2, r_{1 \times 1 \times \Delta} r_{1 \times \tau \times 1} (u_1 \otimes u_2)\big)
\end{gather*}
\end{cor}

There must also be a natural isomorphism as below.
\begin{align*}
	F(T_1, F(T_2, \underline{M})) \xrightarrow{\cong} F(T_1 \times T_2, \underline{M})
\end{align*}

We give a partial computation of this, as follows. An element of $F(T_1, F(T_2, \underline{M})) (X)$ may be represented by a pair $(V \xrightarrow{j} X, u)$ with $u \in F(T_2, \underline{M}) (V \times T_1)$. Suppose that $u$ is represented by a pair $(W \xrightarrow{f} V \times T_1, x)$ with $x \in \underline{M} (W \times T_2)$. Then we can represent $u$ schematically, as follows.
\begin{align*}
	(v, t_1) \mapsto \sum_{w \in W, f(w) = (v,t_1)} \Big( \prod_{t_2 \in T_2} x(w, t_2) \Big)
\end{align*}

Then $(V \xrightarrow{=} V, u)$ can be represented as below.
\begin{align*}
	v &\mapsto \prod_{t_1 \in T_1} \Big( \sum_{w \in W, f(w) = (v,t_1)} \Big( \prod_{t_2 \in T_2} x(w, t_2) \Big) \Big) \\
	   &= \sum_{(v,s) \in D(W,f,V)} \Big( \prod_{t_1 \in T_1} \prod_{t_2 \in T_2} x(s(v,t_1), t_2) \Big)
\end{align*}

It is not difficult to translate these formulas into topology. We obtain the following. In principle, a formula can be given for the case where $u$ is a difference of pairs, but we shall not attempt to write it down here.

\begin{cor}\label{cor:intrpowpow}
There is an associative and unital system of natural isomorphisms as below. The elements of the form below map as shown.
\begin{align*}
	F(T_1, F(T_2, \underline{M})) &\xrightarrow{\cong} F(T_1 \times T_2, \underline{M}) \\
	\big(V \xrightarrow{j} X, (W \xrightarrow{f} V \times T_1, x) \big) &\mapsto \big(D(W, f, V) \xrightarrow{j \circ p} X, r_{e \times 1} (x) \big)
\end{align*}
\end{cor}

Next we give an intrinsic description of the norm functor.

\subsection{Norm Functor}\label{subsec:norm}

Fix a subgroup $H$ of $G$, and a set of coset representatives for $G/H$, with $1$ representing the identity coset. We shall freely identify $G \times_H \Res_H^G V$ with $V \times G/H$ and $G/H \times V$ with $V \times G/H$ in this subsection, for convenience. Before we define our candidate for the norm functor, we introduce one small notation. If $V \in \Fin_G$ and $W \in \Fin_H$, and $f : W \to V$ is an $H$-map, we shall define $D_H (W, f, V)$ by the following exponential diagram.
\begin{align}\label{eq:intrnormdistr}
\xymatrix{
 D_H (W, f, V) \times G/H \ar[d]_-{e} \ar[rr]^-{\pi_1} && D_H (W, f, V) \ar[d]^-{p} \\
 G \times_H W \ar[r]_-{G \times_H f} & V \times G/H \ar[r]_-{\pi_1} & V }
\end{align}

We also define an $H$-map
\begin{align*}
	e_H : D_H (W, f, V) \to W
\end{align*}

to be the restriction of $e$ to $D_H (W, f, V) \times \{ H \}$. Thus $e = G \times_H e_H$.

\begin{definition}\label{def:intrnorm}
Let $\underline{M} \in Mack(H)$. For $X \in \Fin_G$ we define $N^{G,H} \underline{M} (X)$ to be the quotient of the free abelian group on the pairs $(V \xrightarrow{j} X, u \in \underline{M} (\Res_H^G V))$, where $j$ is a map in $\Fin_G$, by the relations
\begin{enumerate}[(i)]
\item $(V \xrightarrow{j} X, u) = (V' \xrightarrow{j'} X, u')$ when there is a commutative diagram
\begin{align*}
\xymatrix{
 V \ar[r]^-{j} & X \\
 V' \ar[u]^-{f}_-{\cong} \ar[ur]_-{j'} & }
\end{align*}
such that $f$ is an isomorphism and $r_{\Res_H^G f} (u) = u'$,
\item $(V_1 \coprod V_2 \xrightarrow{j_1 \coprod j_2} X, (u_1, u_2)) = (V_1 \xrightarrow{j_1} X, u_1) + (V_2 \xrightarrow{j_2} X, u_2)$, and
\item $(V \xrightarrow{j} X, t_f (w)) = (D_H (W, f, V) \xrightarrow{j \circ p} X, r_{e_H} (w))$ for any $H$-set $W$ and $H$-map $f : W \to \Res_H^G V$.
\end{enumerate}
\end{definition}

We define transfers on the generators by composition and restrictions by pullback, as before. Lemma~\ref{lem:distrpullback} implies that the restrictions are well-defined, and it is again easy to see that coproducts of finite $G$-sets are converted into products. Hence, we have a Mackey functor $N^{G,H} \underline{M}$.\\
\indent Next, we define a natural map $\Theta^{G,H} : N^{G,H} \to \TT^{G,H}$ as below.
\begin{align*}
	\Theta^{G,H}_{\underline{M}} (X) : N^{G,H} \underline{M} (X) &\to \TT^{G,H} \underline{M} (X) \\
	(V \xrightarrow{j} X, u) &\mapsto (G/H \times V \xrightarrow{\pi_2} V \xrightarrow{j} X, G \times_H u)
\end{align*}

To see that this is a natural map of Mackey functors, we need only show that it respects relation (iii) of Definition~\ref{def:intrnorm}. For this we apply~\ref{eq:intrnormdistr}, supposing $u = t_f (w)$. We have $G \times_H u = t_{G \times_H f} (G \times_H w)$, and also $r_e (G \times_H w) = r_{G \times_H e_H} (G \times_H w) = G \times_H (r_{e_H} (w))$. Thus we have
\begin{gather*}
	\big(G/H \times V \xrightarrow{\pi_2} V \xrightarrow{j} X, G \times_H (t_f (w))\big) \\
	\sim \big(G/H \times D_H (W, f, V) \xrightarrow{\pi_2} D_H (W, f, V) \xrightarrow{j \circ p} X, G \times_H (r_{e_H} (w)) \big).
\end{gather*}

It follows that $\Theta^{G,H}$ is well-defined, and it is clear that it is surjective. We shall prove the following.

\begin{thm}\label{thm:intrnorm}
The natural map $\Theta^{G,H}$ is an isomorphism.
\end{thm}

We shall require three lemmas.

\begin{lem}\label{lem:intrnormres}
For $\underline{M} \in Mack(G)$, the natural isomorphisms
\begin{align*}
	\Res_H^G \underline{M} (\Res_H^G V) \cong \underline{M} (V \times G/H)
\end{align*}
for $V \in \Fin_G$ induce a natural isomorphism
\begin{align*}
	N^{G,H} \Res_H^G \underline{M} \xrightarrow{\cong} F(G/H, \underline{M}).
\end{align*}
\end{lem}
\begin{proof}
Let $X \in \Fin_G$. The generators for $N^{G,H} \Res_H^G \underline{M} (X)$ are given by pairs $(V \xrightarrow{j} X, u)$ with $u \in \Res_H^G \underline{M} (\Res_H^G V)$, while the generators for $F(G/H, \underline{M}) (X)$ are given by such pairs with $u \in \underline{M} (V \times G/H)$. However, we have $\Res_H^G \underline{M} (\Res_H^G V) \cong \underline{M} (V \times G/H)$; hence, the generators of the two groups are in bijection. To see that we have an isomorphism of groups, one need only check that the relations (iii) of Definitions~\ref{def:intrinsicpower} and~\ref{def:intrnorm} correspond under this identification. For this, let $f : W \to V \times G/H$ be a $G$-map. Then $f$ identifies $W$ (canonically) as $G \times_H W'$, where $W' = f^{-1} (V \times \{ H \})$, and we have $f = G \times_H f'$, where $f'$ is the restriction of $f$ to $W'$. We now note the commutativity of the following diagram, and leave the rest up to the reader.
\begin{align*}
\xymatrix{
 \underline{M} (G \times_H W') \ar[r]^-{t_{G \times_H f'}} \ar[d]_-{=} & \underline{M} (V \times G/H) \ar[d]^-{\cong} \\
 \Res_H^G \underline{M} (W') \ar[r]_-{t_{f'}} & \Res_H^G \underline{M} (\Res_H^G V) }
\end{align*}
That these isomorphisms respect transfers is obvious. That they respect restrictions follows from the commutativity of the following diagram, supposing we have a map $k : P \to V$ in $\Fin_G$.
\begin{align*}
\xymatrix{
 \underline{M} (V \times G/H) \ar[d]_-{\cong} \ar[r]^-{r_{k \times 1}} & \underline{M} (P \times G/H) \ar[d]^-{\cong} \\
 \Res_H^G \underline{M} (\Res_H^G V) \ar[r]_-{r_{\Res_H^G k}} & \Res_H^G \underline{M} (\Res_H^G P) }
\end{align*}
\end{proof}

Next, one can use topology to define a map as below
\begin{align*}
	G \times_H (\spacedash) : \underline{M} (Z) \to \Ind_H^G \underline{M} (G \times_H Z)
\end{align*}

for $\underline{M} \in Mack(H)$ and $Z \in \Fin_H$, as well as an isomorphism
\begin{align*}
	\Ind_H^G \Res_H^G \underline{M} \cong \underline{M} \otimes [\spacedash, G/H]
\end{align*}

for $\underline{M} \in Mack(G)$. Of course, algebraic descriptions can be given for these maps, but these are not needed for our purposes. We can now relate the above maps $G \times_H (\spacedash)$ to the slant map.

\begin{lem}\label{lem:gtimeshslant}
For any $\underline{M} \in Mack(G)$ and $V \in \Fin_G$ the following diagram commutes.
\begin{align*}
\xymatrix{
 \Res_H^G \underline{M} (\Res_H^G V) \ar[d]_-{\cong} \ar[r]^-{G \times_H (\spacedash)} & \Ind_H^G \Res_H^G \underline{M} (V \times G/H) \ar[d]^-{\cong} \\
 \underline{M} (V \times G/H) \ar[r]_-{\slant^{V,G/H}_{\underline{M}}} & \underline{M} \otimes [\spacedash, G/H] (V \times G/H) }
\end{align*}
\end{lem}

The proof, which is a simple topological calculation, is left to the reader. Next we note that $\TT^{G,H} \Res_H^G \underline{M}$ is contained in $\TT (\Ind_H^G \Res_H^G \underline{M})$, while $\TT^{G/H} \underline{M}$ is contained in $\TT (\underline{M} \otimes [\spacedash, G/H])$. We have the following.

\begin{lem}\label{lem:extrnormres}
For $\underline{M} \in Mack(G)$, the natural isomorphism
\begin{align*}
	\Ind_H^G \Res_H^G \underline{M} \cong \underline{M} \otimes [\spacedash, G/H]
\end{align*}
induces a natural isomorphism
\begin{align*}
	\TT^{G,H} \Res_H^G \underline{M} \xrightarrow{\cong} \TT^{G/H} \underline{M}.
\end{align*}
\end{lem}
\begin{proof}
Let $\HH\underline{M}$ be positive cofibrant. Utilizing our set of coset representatives, we obtain a commutative diagram as below.
\begin{align*}
\xymatrix{
 N_H^G \Res_H^G \HH \underline{M} \ar[d]_-{\cong} \ar[r] & N_H^G \CC (\Res_H^G \HH\underline{M}) \ar[d]^-{\cong} \ar[r]^-{\cong} & \CC(G_+ \wedge_H (\Res_H^G \HH\underline{M})) \ar[d]^-{\cong} \\
 (\HH\underline{M})^{\wedge G/H} \ar[r] & (\CC(\HH\underline{M}))^{\wedge G/H} \ar[r]_-{\cong} & \CC(\HH\underline{M} \wedge G/H_+) }
\end{align*}
The image in $\underline{\pi}_0$ of the top composite is $\TT^{G,H} \Res_H^G \underline{M}$, while the image of the bottom composite is $\TT^{G/H} \underline{M}$. The right vertical map is $\CC$ of the canonical isomorphism $G_+ \wedge_H (\Res_H^G \HH\underline{M}) \cong \HH\underline{M} \wedge G/H_+$, which induces the isomorphism $\Ind_H^G \Res_H^G \underline{M} \cong \underline{M} \otimes [\spacedash, G/H]$.
\end{proof}

Now every $H$-Mackey functor is a retract of a restriction of a $G$-Mackey functor. In fact, $\underline{M}$ is canonically a retract of $\Res_H^G \Ind_H^G \underline{M}$. Since retracts of isomorphisms are isomorphisms, Theorem~\ref{thm:intrnorm} is reduced to the following.

\begin{lem}\label{lem:intrnormfinish}
For any $\underline{M} \in Mack(G)$ the diagram below commutes, where the left vertical map is given by Lemma~\ref{lem:intrnormres} and the right vertical map is given by Lemma~\ref{lem:extrnormres}.
\begin{align*}
\xymatrix{
 N^{G,H} \Res_H^G \underline{M} \ar[d]_-{\cong} \ar[rr]^-{\Theta^{G,H}_{\Res_H^G \underline{M}}} && \TT^{G,H} \Res_H^G \underline{M} \ar[d]^-{\cong} \\
 F(G/H, \underline{M}) \ar[rr]_-{\Theta^{G/H}_{\underline{M}}}^-{\cong} && \TT^{G/H} \underline{M} }
\end{align*}
\end{lem}
\begin{proof}
Let an element of $N^{G,H} \Res_H^G \underline{M}$ be given by a pair $(V \xrightarrow{j} X, u)$. Going around the diagram in either way, we obtain a pair with first coordinate as below.
\begin{align*}
	V \times G/H \xrightarrow{\pi_1} V \xrightarrow{j} X
\end{align*}
The fact that we obtain the same second coordinate is precisely the content of Lemma~\ref{lem:gtimeshslant}.
\end{proof}

We have completed the proof of Theorem~\ref{thm:intrnorm}. We now describe the geometric fixed points of this norm. A pair $(V \xrightarrow{j} \ast, u)$ with $V$ an orbit is a transfer unless $V = \ast$. Then $u \in \Res_H^G \underline{M} (\ast) \cong \underline{M} (G/H)$. Hence we have a surjection as below.
\begin{align*}
	\ZZ \{ \underline{M} (G/H) \} &\to \Phi^G N^{G,H} \underline{M} \\
	[u] &\mapsto [(\ast \xrightarrow{=} \ast, u)]
\end{align*}

To show that this map factors through $\Phi^H \underline{M}$, we examine relation (iii) of Definition~\ref{def:intrnorm}. Let $W$ be an $H$-set. Then the diagram below is exponential, where $\pi$ is the unique map $W \to \ast$ and $\pi_H$ is as in Section~\ref{sec:norm}.
\begin{align*}
\xymatrix{
 N_H^G W \times G/H \ar[rr]^-{\pi_1} \ar[d]_-{G \times_H \pi_H} && N_H^G W \ar[d] \\
 G \times_H W \ar[r]_-{G \times_H \pi} & G/H \ar[r] & \ast }
\end{align*}

Now the $G$-fixed points of $N_H^G W$ are the $G/H$-tuples where every element is some fixed $w \in W^H$. Hence, letting $W \cong (\coprod_i \ast) \coprod (\coprod_j H/K_j)$ with each $K_j$ a proper subgroup of $H$, we obtain the relation below in $\Phi^G N^{G,H} \underline{M}$.
\begin{align*}
	\big[\big(\ast \xrightarrow{=} \ast, \textstyle \sum_i u_i + \sum_j \displaystyle t_{K_j}^H u'_j \big)\big] = \sum_i [(\ast \xrightarrow{=} \ast, u_i)]
\end{align*}

We have obtained a natural surjection as below.
\begin{align*}
	\mu^{G,H}_{\underline{M}} : \Phi^H \underline{M} &\to \Phi^G N^{G,H} \underline{M} \\
	[u] &\mapsto [(\ast \xrightarrow{=} \ast, u)]
\end{align*}

Now we know that $\Phi^G N^{G,H}$ is naturally isomorphic to $\Phi^H$, so the source and target of $\mu^{G,H}$ are right exact, additive functors. Thus, to check it is an isomorphism we need only consider representable Mackey functors. In case $\underline{M} = [\spacedash, X]$, we have a surjection of free abelian groups of rank $\# X^H$; hence, an isomorphism.\\
\indent Next we turn to monoidal pairings. We know that there is a natural isomorphism $N^{G,H} \underline{M}_1 \otimes N^{G,H} \underline{M}_2 \cong N^{G,H} (\underline{M}_1 \otimes \underline{M}_2)$. To describe this algebraically, we first describe the elements of $N^{G,H} \underline{M}$ topologically. Let $\HH\underline{M}$ be cofibrant, and suppose that $u \in \underline{M} (\Res_H^G V)$ is represented by an $H$-map
\begin{align*}
	u : \Sigma^{\infty} (\Res_H^G V)_+ \to \Res_H^G \HH\underline{M}
\end{align*}

Letting $\{ g_i \}$ be our set of coset representatives for $G/H$, one calculates that the pair $(G/H \times V \xrightarrow{\pi_2} V \xrightarrow{=} V, G \times_H u)$ is represented by the $G$-map
\begin{align*}
	\Sigma^{\infty} V_+ \to N_H^G \HH\underline{M}
\end{align*}

whose underlying nonequivariant map on the wedge summand corresponding to $v \in V$ is as below.
\begin{align*}
	S^0 \cong \bigwedge_i S^0 \xrightarrow{\wedge_i u(g_i^{-1} \cdot v)} \bigwedge_i \HH\underline{M} = N_H^G \HH\underline{M}
\end{align*}

We can represent this schematically as below.
\begin{align*}
	v \mapsto \prod_i u(g_i^{-1} \cdot v)
\end{align*}

One now obtains the following by an easy calculation.

\begin{cor}\label{cor:intrnormsymmon}
There is a canonical isomorphism $\underline{A} \xrightarrow{\cong} N^{G,H} \underline{A}$ sending the universal element to $(\ast \xrightarrow{=} \ast, 1)$. (More generally, for any finite $H$-set $X$ there is an isomorphism $[\spacedash, N_H^G X] \xrightarrow{\cong} N^{G,H} ([\spacedash, X])$ sending the universal element to $(N_H^G X \xrightarrow{=} N_H^G X, \pi_H)$.) There is also an associative, commutative and unital natural isomorphism as below.
\begin{align*}
	N^{G,H} \underline{M}_1 \otimes N^{G,H} \underline{M}_2 &\xrightarrow{\cong} N^{G,H} (\underline{M}_1 \otimes \underline{M}_2) \\
	(V_1 \xrightarrow{j_1} X_1, u_1) \otimes (V_2 \xrightarrow{j_2} X_2, u_2) &\mapsto (V_1 \times V_2 \xrightarrow{j_1 \times j_2} X_1 \times X_2, u_1 \otimes u_2)
\end{align*}
That is, the functor $N^{G,H}$ is strong symmetric monoidal.
\end{cor}

Next let $K$ be a subgroup of $H$, and choose a set $\{ h_l \}$ of coset representatives for $H/K$, with $1$ representing the identity coset. Then $\{ g_i h_l \}$ is a set of coset representatives for $G/K$, with $1$ representing the identity coset, and we have a natural isomorphism of functors $N_K^G \cong N_H^G N_K^H$ on $K$-spectra. Hence we have a natural isomorphism $N^{G,K} \cong N^{G,H} N^{H,K}$ on $K$-Mackey functors. We give a partial computation of this, as follows. Suppose an element of $N^{G,H} N^{H,K} \underline{M}$ is represented by a pair $(V \xrightarrow{j} X, u)$, and that $u$ is represented by a pair $(W \xrightarrow{f} V, x)$. We can represent $u$ schematically as below.
\begin{align*}
	v \mapsto \sum_{w \in W, f(w) = v} \prod_l x(h_l^{-1} \cdot w)
\end{align*}

Then $(V \xrightarrow{=} V, u)$ can be represented as below,
\begin{align*}
	v &\mapsto \prod_i \Big( \sum_{w \in W, f(w) = g_i^{-1} \cdot v} \prod_l x(h_l^{-1} \cdot w) \Big) \\
	   &= \sum_{(v,s) \in D_H (W, f, V)} \prod_i \prod_l x(h_l^{-1} \cdot w_{i,v})
\end{align*}

where $w_{i,v}$ is such that $[g_i, w_{i,v}] = s(v, g_i H)$. This implies that we have $w_{i,v} = e_H (g_i^{-1} \cdot (v,s))$, so that
\begin{align*}
	x(h_l^{-1} \cdot w_{i,v}) = x(h_l^{-1} \cdot e_H (g_i^{-1} \cdot (v,s))) = r_{e_H} (x) ((g_i h_l)^{-1} \cdot (v,s)).
\end{align*}

It is not difficult to translate this argument into topology. The result is the following. In principle, one could write down a formula for the case where $u$ is a difference of pairs, but we shall not do so here.

\begin{cor}\label{cor:intrnormcomposite}
There is an associative system of natural isomorphisms of symmetric monoidal functors $N^{G,H} N^{H,K} \cong N^{G,K}$. The elements of the form below map as shown.
\begin{align*}
	N^{G,H} N^{H,K} \underline{M} &\xrightarrow{\cong} N^{G,K} \underline{M} \\
	\big(V \xrightarrow{j} X, (W \xrightarrow{f} V, x)\big) &\mapsto \big(D_H (W, f, V) \xrightarrow{j \circ p} X, r_{\Res_K^H e_H} (x)\big)
\end{align*}
\end{cor}

Now, since there is an isomorphism $(N_H^G X)^{\wedge T} \cong N_H^G (X^{\wedge \Res_H^G T})$ for $H$-spectra $X$ and $T \in \Fin_G$, indicated schematically below,
\begin{align}\label{eq:normpowiso}
	(N_H^G X)^{\wedge T} &\xrightarrow{\cong} N_H^G (X^{\wedge \Res_H^G T}) \\
	\wedge_t \wedge_i x_{t,i} &\mapsto \wedge_i \wedge_t x_{g_i \cdot t, i} \nonumber
\end{align}

we have a canonical isomorphism as below.
\begin{align*}
	F(T, N^{G,H} \underline{M}) \xrightarrow{\cong} N^{G,H} F(\Res_H^G T, \underline{M})
\end{align*}

We give a partial computation of this, as follows. Consider an element of $F(T, N^{G,H} \underline{M})$ represented by a pair $(V \xrightarrow{j} X, u)$, where $u$ is equal to $(W \xrightarrow{f} V \times T, x)$. Then we can represent $u$ schematically, as below.
\begin{align*}
	(v,t) \mapsto \sum_{w \in W, f(w) = (v,t)} \prod_i x(g_i^{-1} \cdot w)
\end{align*}

Then we can represent $(V \xrightarrow{=} V, u)$ as below.
\begin{align*}
	v &\mapsto \prod_{t \in T} \Big( \sum_{w \in W, f(w) = (v,t)} \prod_i x(g_i^{-1} \cdot w) \Big) \\
	   &= \sum_{(v,s) \in D(W, f, V)} \prod_{t \in T} \prod_i x(g_i^{-1} \cdot s(v,t))
\end{align*}

Careful examination of~\ref{eq:normpowiso} yields the following.
\begin{align*}
	\prod_{t \in T} \prod_i x(g_i^{-1} \cdot s(v,t)) = \prod_i \prod_{t \in T} x(g_i^{-1} \cdot s(v, g_i \cdot t))
\end{align*}

We also have $g_i^{-1} \cdot s(v, g_i \cdot t) = (g_i^{-1} \cdot s) (g_i^{-1} \cdot v, t)$, so that
\begin{align*}
	x(g_i^{-1} \cdot s(v, g_i \cdot t)) = r_e(x) (g_i^{-1} \cdot (v,s), t).
\end{align*}

Thus we obtain the following. Again, when $u$ is a difference of pairs one could obtain a formula, but we shall not do so here.
\begin{cor}\label{cor:normpow}
There is an associative and unital system of natural isomorphisms $F(T, N^{G,H} \underline{M}) \cong N^{G,H} F(\Res_H^G T, \underline{M})$. The elements of the form below map as shown.
\begin{align*}
	F(T, N^{G,H} \underline{M}) &\xrightarrow{\cong} N^{G,H} F(\Res_H^G T, \underline{M}) \\
	\big(V \xrightarrow{j} X, (W \xrightarrow{f} V \times T, x)\big) &\mapsto \big(D(W, f, V) \xrightarrow{j \circ p} X, \\
	\big(\Res_H^G D(W, f, V) &\xrightarrow{=} \Res_H^G D(W, f, V), r_{\Res_H^G e} (x)\big)\big)
\end{align*}
\end{cor}

In the next section we recap what we have learned about certain homotopy groups of spectra.

\section{Multiplicative Push Forward and Homotopy Groups}\label{sec:multpush}

In this section we use \emph{multiplicative push forwards} to give pleasing descriptions of the homotopy groups of spectra which we have calculated. For any finite $G$-set $X$, we denote by $\BB_G (X)$ the translation category of $X$, and let $Sp_{\BB_G (X)}$ denote the category of functors from $\BB_G (X)$ to $Sp$. If $X \cong \coprod_i G/H_i$ then we have an equivalence of categories as below.
\begin{align}\label{eq:transequiv}
	Sp_{\BB_G (X)} \cong \prod_i Sp_{H_i}
\end{align}

We can put a model structure on this category which corresponds to the product of the stable (or positive stable) model structures under this equivalence. Note that the result does not depend on the choice of orbit representatives. A map $Y \to Z$ in $Sp_{\BB_G (X)}$ is a weak equivalence (fibration) if and only if all the maps $Y(x) \to Z(x)$ in $Sp_{G_x}$ for $x \in X$ are weak equivalences (fibrations).\\
\indent Now, any map $f : U \to V$ induces a functor $f : \BB_G (U) \to \BB_G (V)$, so by precomposition we have a \emph{pullback} functor as below.
\begin{align*}
	f^* : Sp_{\BB_G (V)} \to Sp_{\BB_G (U)}
\end{align*}

Under appropriate equivalences~\ref{eq:transequiv}, this corresponds to a product of restriction functors, so it preserves fibrations and trivial fibrations. It also has a left adjoint, which we call \emph{additive push forward}, obtained by taking wedge sums over fibers as below.
\begin{align*}
	f_* : Sp_{\BB_G (U)} &\to Sp_{\BB_G (V)} \\
	(u \mapsto X_u) &\mapsto (v \mapsto \vee_{u \in f^{-1} (v)} X_u)
\end{align*}

Under appropriate equivalences~\ref{eq:transequiv}, in each factor this is a wedge of induction functors. Hence $(f_*, f^*)$ is a Quillen pair. (In fact, $f^*$ preserves (trivial) cofibrations as well.) We may similarly define the \emph{multiplicative push forward} as below.
\begin{align*}
	f_{\star} : Sp_{\BB_G (U)} &\to Sp_{\BB_G (V)} \\
	(u \mapsto X_u) &\mapsto (v \mapsto \wedge_{u \in f^{-1} (v)} X_u)
\end{align*}

Under appropriate equivalences~\ref{eq:transequiv}, in each factor this is a smash product of norm functors, so it preserves (positive) cofibrancy and weak equivalences between cofibrant objects. Hence, it is left derivable.\\
\indent Next note that $Sp_{\BB_G (\ast)} = Sp_G$. If $V$ is a $G$-set and $\pi : V \to \ast$ is the unique map to a point, then we have $\pi^* X (v) = X$ for all $v \in V$. We denote this $\BB_G (V)$-diagram by $const_V X$. Hence, $(\pi_*, const_V)$ is a Quillen pair. We note the following easy facts.
\begin{enumerate}[(i)]
\item If $\pi : V \to \ast$ then $\pi_* (const_V S^0) \cong \Sigma^{\infty} V_+$.
\item If $i : U \to V$ then $i_{\star} (const_U S^0) \cong const_V S^0$.
\item If $\pi^T = \pi_1 : V \times T \to V$ then $\pi^T_{\star} (const_{V \times T} X) \cong const_V X^{\wedge T}$.
\end{enumerate}

Part (i) gives an adjunction isomorphism as below.
\begin{align*}
	[\Sigma^{\infty} V_+, X] \cong [const_V S^0, const_V X]
\end{align*}

Hence we can define maps
\begin{align*}
	\pi^T_{\star} : \underline{\pi}_0 X (V \times T) \to \underline{\pi}_0 X^{\wedge T} (V)
\end{align*}

as below, where on the bottom right we have used (ii) and (iii).
\begin{align*}
\xymatrix{
 \underline{\pi}_0 X (V \times T) \ar@{-->}[r]^-{\pi^T_{\star}} \ar[d]_-{\cong} & \underline{\pi}_0 X^{\wedge T} (V) \ar[d]^-{\cong} \\
 [const_{V \times T} S^0, const_{V \times T} X] \ar[r]_-{\pi^T_{\star}} & [const_V S^0, const_V X^{\wedge T}] }
\end{align*}

With this notation and the calculations in Subsection~\ref{subsec:gsymmon}, we obtain the following.

\begin{thm}\label{thm:hmpowhtpy}
Let $\underline{M} \in Mack(G)$, and let $T$ and $X$ be finite $G$-sets. Let $Z (T, \underline{M}, X)$ be the free abelian group on the pairs $(V \xrightarrow{j} X, u)$, where $j$ is a map in $\Fin_G$ and $u \in \underline{M} (V \times T)$. Then the map
\begin{align*}
	Z (T, \underline{M}, X) &\to \underline{\pi}_0 (\HH \underline{M})^{\wedge T} (X) \\
	(V \xrightarrow{j} X, u) &\mapsto t_j (\pi^T_{\star} (u))
\end{align*}
is a surjection. The kernel is generated by the relations
\begin{enumerate}[(i)]
\item $(V \xrightarrow{j} X, u) = (V' \xrightarrow{j'} X, u')$ when there is a commutative diagram
\begin{align*}
\xymatrix{
 V \ar[r]^-{j} & X \\
 V' \ar[u]^-{f}_-{\cong} \ar[ur]_-{j'} & }
\end{align*}
such that $f$ is an isomorphism and $r_{f \times 1} (u) = u'$,
\item $(V_1 \coprod V_2 \xrightarrow{j_1 \coprod j_2} X, (u_1, u_2)) = (V_1 \xrightarrow{j_1} X, u_1) + (V_2 \xrightarrow{j_2} X, u_2)$, and
\item $(V \xrightarrow{j} X, t_f (w)) = (D \xrightarrow{j \circ p} X, r_e (w))$ for any map of finite $G$-sets $f : W \to V \times T$, where $D$, $p$ and $e$ are as in the exponential diagram below.
\begin{align*}
\xymatrix{
 D \times T \ar[d]_-{e} \ar[rr]^-{\pi^T} && D \ar[d]^-{p} \\
 W \ar[r]_-{f} & V \times T \ar[r]_-{\pi^T} & V }
\end{align*}
\end{enumerate}
\end{thm}

To describe the norm functor, we begin by noting that, for any $G$-set $V$ and subgroup $H \subseteq G$, we have a canonical inclusion
\begin{align*}
	\BB_H (\Res_H^G V) \subseteq \BB_G (G/H \times V)
\end{align*}

corresponding to the identity coset, and this is an equivalence. We obtain an equivalence of categories as below.
\begin{align*}
	Sp_{\BB_G (G/H \times V)} \xrightarrow{\cong} Sp_{\BB_H (\Res_H^G V)}
\end{align*}

After choosing a set of coset representatives $\{ g_i \}$ for $G/H$ (with $1$ representing the identity coset), we obtain an inverse equivalence by using the contraction below.
\begin{align*}
	\BB_G (G/H \times V) &\to \BB_H (\Res_H^G V) \\
	(g_i H, v) &\mapsto g_i^{-1} \cdot v
\end{align*}

Applying this equivalence to $const_{\Res_H^G V} X$ for some $H$-spectrum $X$, we obtain a $\BB_G (G/H \times V)$-diagram which we shall denote by $const_{G/H \times V}^{G,H} X$. Note the following.
\begin{enumerate}[(i)]
\item The diagram $const_{G/H \times V}^{G,H} S^0$ is equal to $const_{G/H \times V} S^0$, since $S^0$ has trivial action.
\item If $\pi^{G/H} = \pi_2 : G/H \times V \to V$ then we have an isomorphism
\begin{align*}
	\pi^{G/H}_{\star} (const_{G/H \times V}^{G,H} X) \cong const_V N_H^G X.
\end{align*}
\end{enumerate}

Thus, for $H$-spectra $X$ and $G$-sets $V$ we obtain a map
\begin{align*}
	\pi^{G,H}_{\star} : \underline{\pi}_0 X (\Res_H^G V) \to \underline{\pi}_0 N_H^G X (V)
\end{align*}

as below, where we use (i) on left and (ii) on the right.
\begin{align*}
\xymatrix{
 \underline{\pi}_0 X (\Res_H^G V) \ar[d]_-{\cong} \ar@{-->}[r]^-{\pi^{G,H}_{\star}} & \underline{\pi}_0 N_H^G X (V) \ar[d]^-{\cong} \\
 [const_{G/H \times V} S^0, const_{G/H \times V}^{G,H} X] \ar[r]_-{\pi^{G/H}_{\star}} & [const_V S^0, const_V N_H^G X] }
\end{align*}

With this notation and the calculations in Subsection~\ref{subsec:norm}, we obtain the following.

\begin{thm}\label{hmnormhtpy}
Let $H$ be a subgroup of $G$, $\underline{M} \in Mack(H)$ and $X$ a finite $G$-set. Let $Z_H^G (\underline{M}, X)$ be the free abelian group on the pairs $(V \xrightarrow{j} X, u)$, with $j$ a map in $\Fin_G$ and $u \in \underline{M} (\Res_H^G V)$. Then the map
\begin{align*}
	Z_H^G (\underline{M}, X) &\to \underline{\pi}_0 N_H^G (\HH \underline{M}) (X) \\
	(V \xrightarrow{j} X, u) &\mapsto t_j (\pi^{G,H}_{\star} (u))
\end{align*}
is a surjection. The kernel is generated by the relations
\begin{enumerate}[(i)]
\item $(V \xrightarrow{j} X, u) = (V' \xrightarrow{j'} X, u')$ when there is a commutative diagram
\begin{align*}
\xymatrix{
 V \ar[r]^-{j} & X \\
 V' \ar[u]^-{f}_-{\cong} \ar[ur]_-{j'} & }
\end{align*}
such that $f$ is an isomorphism and $r_{\Res_H^G f} (u) = u'$,
\item $(V_1 \coprod V_2 \xrightarrow{j_1 \coprod j_2} X, (u_1, u_2)) = (V_1 \xrightarrow{j_1} X, u_1) + (V_2 \xrightarrow{j_2} X, u_2)$, and
\item $(V \xrightarrow{j} X, t_f (w)) = (D \xrightarrow{j \circ p} X, r_{e_H} (w))$ for any $H$-set $W$ and $H$-map $f : W \to \Res_H^G V$, where $D$, $p$ and $e_H$ are as in the exponential diagram below.
\begin{align*}
\xymatrix{
 G/H \times D \ar[d]_-{G \times_H e_H} \ar[rr]^-{\pi^{G/H}} && D \ar[d]^-{p} \\
 G \times_H W \ar[r]_-{G \times_H f} & G/H \times V \ar[r]_-{\pi^{G/H}} & V }
\end{align*}
\end{enumerate}
\end{thm}

Finally, we can describe the symmetric powers of Mackey functors in similar terms. Let $n > 0$, and suppose that $i : U \to V$ is a $G$-map of degree $n$. Then for any $G$-spectrum $X$ we have a quotient map
\begin{align*}
	q^i : i_{\star} (const_U X) \to const_V (X^{\wedge n} / \Sigma_n).
\end{align*}

Then we can define a map
\begin{align*}
	i^{\Sigma}_{\star} : \underline{\pi}_0 X (U) \to \underline{\pi}_0 (X^{\wedge n} / \Sigma_n) (V)
\end{align*}

as below.
\begin{align*}
\xymatrix{
 \underline{\pi}_0 X (U) \ar[d]_-{\cong} \ar@{-->}[rr]^-{i^{\Sigma}_{\star}} && \underline{\pi}_0 (X^{\wedge n} / \Sigma_n) (V) \ar[d]^-{\cong} \\
 [const_U S^0, const_U X] \ar[rr]_-{[\spacedash, q^i] \circ i_{\star}} && [const_V S^0, const_V (X^{\wedge n} / \Sigma_n)] }
\end{align*}

We can now restate the description from Section~\ref{sec:sympow}.

\begin{thm}\label{thm:sympowmultpush}
Let $\underline{M} \in Mack(G)$, $n > 0$ and $X \in \Fin_G$. Let $Z^n (\underline{M}, X)$ be the free abelian group on the pairs $(U \xrightarrow{i} V \xrightarrow{j} X, u)$, where $i$ and $j$ are maps in $\Fin_G$, $i$ has degree $n$ and $u \in \underline{M} (U)$. Then the map
\begin{align*}
	Z^n (\underline{M}, X) &\to \underline{\pi}_0 ((\HH \underline{M})^{\wedge n} / \Sigma_n) (X) \\
	(U \xrightarrow{i} V \xrightarrow{j} X, u) &\mapsto t_j (i^{\Sigma}_{\star} (u))
\end{align*}
is a surjection. The kernel is generated by the relations
\begin{enumerate}[(i)]
\item $(U \xrightarrow{i} V \xrightarrow{j} X, u) = (U' \xrightarrow{i'} V' \xrightarrow{j'} X, u')$ whenever there is a commutative diagram
\begin{align*}
\xymatrix{
 U \ar[r]^-{i} & V \ar[r]^-{j} & X \\
 U' \ar[u]^-{f}_-{\cong} \ar[r]_-{i'} & V' \ar[u]^-{g}_-{\cong} \ar[ru]_-{j'} & }
\end{align*}
such that $f$ and $g$ are isomorphisms and $r_f (u) = u'$,
\item $(U_1 \textstyle \coprod U_2 \xrightarrow{i_1 \coprod i_2} V_1 \coprod V_2 \xrightarrow{j_1 \coprod j_2} X, (u_1, u_2)) =$ \\
$(U_1 \xrightarrow{i_1} V_1 \xrightarrow{j_1} X, u_1) + (U_2 \xrightarrow{i_2} V_2 \xrightarrow{j_2} X, u_2)$, and
\item $(U \xrightarrow{i} V \xrightarrow{j} X, t_k (w)) = (A \xrightarrow{g} B \xrightarrow{j \circ h} X, r_f (w))$ whenever the diagram
\begin{align*}
\xymatrix{
 W \ar[r]^-{k} & U \ar[r]^-{i} & V \ar[r]^-{j} & X \\
 A \ar[u]^-{f} \ar[rr]_-{g} && B \ar[u]_-{h} \ar[ur]_-{j \circ h} & }
\end{align*}
is commutative and the rectangle is exponential.
\end{enumerate}
\end{thm}

\section{The Norm / Restriction Adjunctions for Tambara and semi-Tambara Functors}\label{sec:normresadj}

One result of~\cite{UTamb} is that the norm functor on Mackey functors gives the left adjoint of restriction on Tambara functors. In this section, we demonstrate this adjunction algebraically. We begin with restriction functors. Let $H$ be a subgroup of $G$. Since the induction functor
\begin{align*}
	G \times_H (\spacedash) : \Fin_H \to \Fin_G
\end{align*}

preserves all colimits, as well as pullbacks and exponential diagrams, we may define a restriction functor
\begin{align*}
	\Res_H^G : Tamb(G) \to Tamb(H).
\end{align*}

For a $G$-Tambara functor $\underline{R}$, we define
\begin{align*}
	(\Res_H^G \underline{R})_{\star} \defeq \underline{R}_{\star} \circ G \times_H (\spacedash),
\end{align*}

just as with restrictions and transfers (this corresponds to topological restriction). As noted in~\cite{UTamb}, this functor has a right adjoint $\Ind_H^G$, defined similarly; however, we are interested in the left adjoint.\\
\indent In order to give norm maps to norms of $H$-Tambara functors, we must work with semi-Tambara functors and additive completion. For this purpose we introduce a norm construction on semi-Mackey functors
\begin{align*}
	sN^{G,H} : sMack(H) \to sMack(G).
\end{align*}

Let $\underline{M} \in sMack(H)$ and $X \in \Fin_G$. We define $sN_0^{G,H} \underline{M} (X)$ to be the set of isomorphism classes of pairs $(V \xrightarrow{j} X, u)$, with $j$ a map in $\Fin_G$ and $u \in \underline{M} (\Res_H^G V)$, where $(V \xrightarrow{j} X, u)$ is isomorphic to $(V' \xrightarrow{j'} X', u')$ if and only if there is a commutative diagram
\begin{align*}
\xymatrix{
 V \ar[r]^-{j} & X \\
 V' \ar[u]^-{f}_-{\cong} \ar[ur]_-{j'} & }
\end{align*}

such that $f$ is an isomorphism and $r_{\Res_H^G f} (u) = u'$. We define transfers by composition and restrictions by pullback, as before. It is clear that we now have a semi-Mackey functor $sN_0^{G,H} \underline{M}$; addition is achieved by taking disjoint unions of $V$'s, as below.
\begin{align*}
	(V_1 \xrightarrow{j_1} X, u_1) + (V_2 \xrightarrow{j_2} X, u_2) = \textstyle \big(V_1 \coprod V_2 \xrightarrow{j_1 \coprod j_2} X, (u_1, u_2)\big)
\end{align*}

Next, we define $sN^{G,H} \underline{M} (X)$ to be the quotient of $sN_0^{G,H} \underline{M} (X)$ by the smallest equivalence relation $\sim$ such that, for any $G$-map $j : V \to X$, $H$-map $f : W \to \Res_H^G V$ and $w \in \underline{M} (W)$, we have
\begin{align*}
	(V \xrightarrow{j} X, t_f (w)) \sim \big(D_H (W, f, V) \xrightarrow{j \circ p}, r_{e_H} (w)\big),
\end{align*}

where $D_H (W, f, V)$, $p$ and $e_H$ are as in the exponential diagram below.
\begin{align*}
\xymatrix{
 G/H \times D_H (W, f, V) \ar[d]_-{G \times_H e_H} \ar[rr]^-{\pi_2} && D_H (W, f, V) \ar[d]^-{p} \\
 G \times_H W \ar[r]_-{G \times_H f} & G/H \times V \ar[r]_-{\pi_2} & V }
\end{align*}

Lemma~\ref{lem:distrpullback} implies that the restrictions descend to this quotient. That the transfers descend is trivial. It remains to verify that $sN^{G,H} \underline{M}$ converts coproducts in $\Fin_G$ into products. One sees that the equivalence relation on $sN_0^{G,H} \underline{M} (X_1 \textstyle \coprod X_2) \cong sN_0^{G,H} \underline{M} (X_1) \bigtimes sN_0^{G,H} \underline{M} (X_2)$ corresponds to the product of the equivalence relations for $X_1$ and $X_2$. Note that when $\underline{M}$ is a Mackey functor, $N^{G,H} \underline{M}$ is the additive completion of $sN^{G,H} \underline{M}$.\\
\indent Now suppose that $\underline{R} \in sTamb(H)$. We define norm maps for $sN_0^{G,H} \underline{R}$ as follows. Let $f : X \to Y$ be a $G$-map. To define $n_f$ of $(V \xrightarrow{j} X, u)$, we form the exponential diagram
\begin{align*}
\xymatrix{
 E \ar[d]_-{e} \ar[rr]^-{k} && D \ar[d]^-{p} \\
 V \ar[r]_-{j} & X \ar[r]_-{f} & Y }
\end{align*}

and take $n_f ((V \xrightarrow{j} X, u))$ to be $\big(D \xrightarrow{p} Y, n_{\Res_H^G k} r_{\Res_H^G e} (u)\big)$. It is clear that this is well-defined. Lemma~\ref{lem:functorialnorm} implies that this norm is functorial, and Lemma~\ref{lem:distrlaw} implies that the distributive law holds. The commutation of norms and restrictions follows from Lemma~\ref{lem:distrpullback}.\\
\indent We must now show that these norm maps descend to $sN^{G,H} \underline{R}$. For this we point out the following interpretation of exponential diagrams. Let $f : X \to Y$ be a $G$-map. Then $f$ induces a pullback functor as below.
\begin{align*}
	f^* : \Fin_G / Y &\to \Fin_G / X \\
	(V \to Y) &\mapsto V \times_Y X
\end{align*}

This functor has a right adjoint, namely $(A \xrightarrow{i} X) \mapsto \big(\prod_{i,f} A \xrightarrow{p} Y\big)$. In an exponential diagram as below,
\begin{align*}
\xymatrix{
 D \times_Y X \ar[d]_-{e} \ar[rr]^-{\pi_1} && D \ar[d] \\
 A \ar[r]_-{i} & X \ar[r]_-{f} & Y }
\end{align*}

the map $e$ is the counit of the adjunction. Hence we have an isomorphism as below.
\begin{align*}
	Hom_{\Fin_G / Y} (B, D) &\xrightarrow{\cong} Hom_{\Fin_G /X} (B \times_Y X, A) \\
	j &\mapsto e \circ (j \times_Y 1)
\end{align*}

Also, if $A$ and $B$ are $H$-sets over $C$, we have an isomorphism as below.
\begin{align*}
	Hom_{\Fin_H / C} (A, B) &\xrightarrow{\cong} Hom_{\Fin_G / G \times_H C} (G \times_H A, G \times_H B) \\
	j &\mapsto G \times_H j
\end{align*}

We are now ready to prove the following.

\begin{lem}\label{lem:normtambnormdescend}
The norm maps for $sN_0^{G,H} \underline{R}$ descend to $sN^{G,H} \underline{R}$.
\end{lem}
\begin{proof}
Suppose given $G$-sets $A$, $B$ and $C$, and an $H$-set $F$. Suppose we have maps $a : A \to B$ and $b : B \to C$ in $\Fin_G$, and a map $f : F \to \Res_H^G A$ in $\Fin_H$. Now form the exponential diagram below.
\begin{align}\label{eq:firstmanyexp}
\xymatrix{
 J \ar[d]_-{j_1} \ar[rr]^-{j_2} && K \ar[d]^-{k} \\
 A \ar[r]_-{a} & B \ar[r]_-{b} & C }
\end{align}
Next form the pullback diagram below.
\begin{align}\label{eq:mainpullback}
\xymatrix{
 I \ar[d]_-{i_1} \ar[r]^-{i_2} & \Res_H^G J \ar[d]^-{\Res_H^G j_1} \\
 F \ar[r]_-{f} & \Res_H^G A }
\end{align}
It follows by Lemma~\ref{lem:distrpullback} that the exponential diagram
\begin{align*}
\xymatrix{
 G/H \times D_H (I, i_2, J) \ar[d]_-{G \times_H e_H^{i_2}} \ar[rr]^-{\pi_2} && D_H (I, i_2, J) \ar[d]^-{p^{i_2}} \\
 G \times_H I \ar[r]_-{G \times_H i_2} & G/H \times J \ar[r]_-{\pi_2} & J }
\end{align*}
is the pullback of the exponential diagram below over $j_1$.
\begin{align*}
\xymatrix{
 G/H \times D_H (F, f, A) \ar[d]_-{G \times_H e_H^f} \ar[rr]^-{\pi_2} && D_H (F, f, A) \ar[d]^-{p^f} \\
 G \times_H F \ar[r]_-{G \times_H f} & G/H \times A \ar[r]_-{\pi_2} & A }
\end{align*}
Hence we obtain two commutative squares as below (the left square is a pullback).
\begin{align}\label{eq:manytwosquareleftpullback}
\xymatrix{
 D_H (I, i_2, J) \ar[d]_-{d} \ar[r]^-{p^{i_2}} & J \ar[d]^-{j_1} & \Res_H^G D_H (I, i_2, J) \ar[d]_-{e_H^{i_2}} \ar[r]^-{\Res_H^G d} & \Res_H^G D_H (F, f, A) \ar[d]^-{e_H^f} \\
 D_H (F, f, A) \ar[r]_-{p^f} & A & I \ar[r]_-{i_1} & F }
\end{align}
Now form the exponential diagram below.
\begin{align}\label{eq:manyhexp}
\xymatrix{
 L \ar[d]_-{l_1} \ar[rr]^-{l_2} && M \ar[d]^-{m} \\
 I \ar[r]_-{i_2} & \Res_H^G J \ar[r]_-{\Res_H^G j_2} & \Res_H^G K }
\end{align}
Next form the exponential diagram below.
\begin{align*}
\xymatrix{
 G/H \times D_H (M, m, K) \ar[d]_-{G \times_H e_H^m} \ar[rr]^-{\pi_2} && D_H (M, m, K) \ar[d]^-{p^m} \\
 G \times_H M \ar[r]_-{G \times_H m} & G/H \times K \ar[r]_-{\pi_2} & K }
\end{align*}
We now have a chain of isomorphisms for $X \in \Fin_G / K$, as below.
\begin{align}\label{eq:overiso} \nonumber
	Hom_{/ K} (X, \prod_{p^{i_2}, j_2} D_H (I, i_2, J)) &\cong Hom_{/ J} (X {\times_K} J, D_H (I, i_2, J)) \\ \nonumber
	                                                                                               &\cong Hom_{/ (G/H \times J)} (G/H {\times} X {\times_K} J, G \times_H I) \\ \nonumber
	                                                                                               &\cong Hom_{/ \Res_H^G J} (\Res_H^G X {\times_{\Res_H^G K}} \Res_H^G J, I) \\
	                                                                                               &\cong Hom_{/ \Res_H^G K} (\Res_H^G X, M) \\ \nonumber
	                                                                                               &\cong Hom_{/ (G/H \times K)} (G/H {\times} X, G {\times_H} M) \\ \nonumber
	                                                                                               &\cong Hom_{/ K} (X, D_H (M, m, K)) \nonumber
\end{align}
It follows that we have an exponential diagram as below.
\begin{align}\label{eq:manysurprises}
\xymatrix{
 D_H (M, m, K) \times_K J \ar[d]_-{e^{!}} \ar[rr]^-{\pi_1} && D_H (M, m, K) \ar[d]^-{p^m} \\
 D_H (I, i_2, J) \ar[r]_-{p^{i_2}} & J \ar[r]_-{j_2} & K }
\end{align}
Next form the diagram below, where both squares are pullbacks.
\begin{align}\label{eq:manydubsquare}
\xymatrix{
 \Res_H^G (D_H (M, m, K) \times_K J) \ar[d]_-{\Res_H^G \pi_1} \ar[rr]^-{e_H^m \times_{\Res_H^G K} 1} && L \ar[d]^-{l_2} \ar[r]^-{i_2 \circ l_1} & \Res_H^G J \ar[d]^-{\Res_H^G j_2} \\
 \Res_H^G D_H (M, m, K) \ar[rr]_-{e_H^m} && M \ar[r]_-{m} & \Res_H^G K }
\end{align}
Next we claim that the diagram below commutes.
\begin{align}\label{eq:manyfinalsquare}
\xymatrix{
 \Res_H^G (D_H (M, m, K) \times_K J) \ar[d]_-{e_H^m \times_{\Res_H^G K} 1} \ar[r]^-{\Res_H^G e^{!}} & \Res_H^G D_H (I, i_2, J) \ar[d]^-{e_H^{i_2}} \\
 L \ar[r]_-{l_1} & I }
\end{align}
To see this, we consider the isomorphisms~\ref{eq:overiso}. We obtain $e^!$ by moving from the last line to the first, starting with the identity map. Restricting to $H$ and postcomposing with $e_H^{i_2}$ corresponds to moving to the third line of~\ref{eq:overiso}. Moving from the last line to the third-to-last line gives the map $e_H^m$. Moving up one more line corresponds to taking the pullback with $\Res_H^G J$ over $\Res_H^G K$ and postcomposing with $l_1$.\\
\indent Now every generating relation on $sN_0^{G,H} \underline{R} (B)$ is of the form
\begin{align*}
	(A \xrightarrow{a} B, t_f (x)) \sim (D_H (F, f, A) \xrightarrow{a \circ p^f} B, r_{e_H^f} (x))
\end{align*}
for some $A$, $a$, $x$, etc. First, using~\ref{eq:firstmanyexp} we calculate
\begin{align*}
	n_b ((A \xrightarrow{a} B, t_f (x))) = (K \xrightarrow{k} C, n_{\Res_H^G j_2} r_{\Res_H^G j_1} t_f (x)).
\end{align*}
Now $r_{\Res_H^G j_1} t_f = t_{i_2} r_{i_1}$ by~\ref{eq:mainpullback}, and $n_{\Res_H^G j_2} t_{i_2} = t_m n_{l_2} r_{l_1}$ by~\ref{eq:manyhexp}. It follows that we have
\begin{align*}
	n_b ((A \xrightarrow{a} B, t_f (x))) \sim (D_H (M, m, K) \xrightarrow{k \circ p^m} C, r_{e_H^m} n_{l_2} r_{l_1} r_{i_1} (x)).
\end{align*}
From~\ref{eq:manydubsquare} it follows that $r_{e_H^m} n_{l_2} = n_{\Res_H^G \pi_1} r_{e_H^m \times_{\Res_H^G K} 1}$, and by~\ref{eq:manyfinalsquare} we have $r_{e_H^m \times_{\Res_H^G K} 1} r_{l_1} = r_{\Res_H^G e^{!}} r_{e_H^{i_2}}$. Thus we have
\begin{align*}
	n_b ((A \xrightarrow{a} B, t_f (x))) \sim (D_H (M, m, K) \xrightarrow{k \circ p^m} C, n_{\Res_H^G \pi_1} r_{\Res_H^G e^{!}} r_{e_H^{i_2}} r_{i_1} (x)).
\end{align*}
Now by applying Lemma~\ref{lem:distrlaw} to~\ref{eq:firstmanyexp}, \ref{eq:manytwosquareleftpullback} and~\ref{eq:manysurprises}, we see that $d \circ e^!$, $\pi_1$, $k \circ p^m$ is a distributor for $a \circ p^f$, $b$. Hence we have
\begin{gather*}
	n_b ((D_H (F, f, A) \xrightarrow{a \circ p^f} B, r_{e_H^f} (x))) = \\
	(D_H (M, m, K) \xrightarrow{k \circ p^m} C, n_{\Res_H^G \pi_1} r_{\Res_H^G e^!} r_{\Res_H^G d} r_{e_H^f} (x)).
\end{gather*}
Finally, we have $r_{\Res_H^G d} r_{e_H^f} = r_{e_H^{i_2}} r_{i_1}$ by~\ref{eq:manytwosquareleftpullback}, so that
\begin{align*}
	n_b ((A \xrightarrow{a} B, t_f (x))) \sim n_b ((D_H (F, f, A) \xrightarrow{a \circ p^f} B, r_{e_H^f} (x))).
\end{align*}
\end{proof}

We have constructed a functor
\begin{align*}
	sN^{G,H} : sTamb(H) \to sTamb(G).
\end{align*}

We now construct a unit map
\begin{align*}
	\eta_{\underline{R}} : \underline{R} \to \Res_H^G sN^{G,H} \underline{R}
\end{align*}

for arbitrary $\underline{R} \in sTamb(H)$, as follows. We take our cue from topology. If $X$ is a commutative ring $H$-spectrum, then the unit map for the norm / restriction adjunction is as below.
\begin{align*}
	X \cong X \wedge (\wedge_{(G-H)/H} S^0) \xrightarrow{Id_X \wedge (\wedge_{(G-H)/H} 1)} \Res_H^G N_H^G X
\end{align*}

This suggests the following. For any $X \in \Fin_H$, let
\begin{align*}
	\eta_X : X \to \Res_H^G (G \times_H X)
\end{align*}

be the unit map of the induction / restriction adjunction. We shall also denote the counit by $\epsilon_X$ for $X \in \Fin_G$, and note that it corresponds to $\pi^{G/H}$ under the canonical isomorphism $G \times_H (\Res_H^G X) \cong G/H \times X$. We define $\eta_{\underline{R}} (X)$ as below.
\begin{align*}
	\eta_{\underline{R}} (X) : \underline{R} (X) &\to \Res_H^G sN^{G,H} \underline{R} (X) \\
	u &\mapsto (G \times_H X \xrightarrow{=} G \times_H X, n_{\eta_X} (u))
\end{align*}

We now check that this is a map of semi-Tambara functors.

\begin{lem}\label{lem:unitnormresadj}
For any $\underline{R} \in sTamb(H)$, the maps $\eta_{\underline{R}} (X)$ determine a map $\eta_{\underline{R}}$ in $sTamb(H)$.
\end{lem}
\begin{proof}
First we check that $\eta_{\underline{R}}$ commutes with restrictions; this follows from the fact that, for any $H$-map $f : Y \to X$, the diagram
\begin{align*}
\xymatrix{
 Y \ar[d]_-{\eta_Y} \ar[rr]^-{f} && X \ar[d]^-{\eta_X} \\
 \Res_H^G (G \times_H Y) \ar[rr]_-{\Res_H^G (G \times_H f)} && \Res_H^G (G \times_H X) }
\end{align*}
is a pullback. Next, suppose $f : W \to X$ is an $H$-map and $w \in \underline{R} (W)$. The following diagram is exponential.
\begin{align*}
\xymatrix{
 W \ar[d]_-{=} \ar[rr]^-{i_1} && W \coprod (G-H) \times_H X \ar[d]^-{f \coprod 1} \\
 W \ar[r]_-{f} & X \ar[r]_-{\eta_X} & \Res_H^G (G \times_H X) }
\end{align*}
Hence we obtain $n_{\eta_X} t_f (w) = t_{(f \coprod 1)} n_{i_1} (w)$. Next note that any diagram of the form below is exponential.
\begin{align*}
\xymatrix{
 A \coprod A \times_B C \ar[d]_-{1 \coprod \pi_2} \ar[rr]^-{1 \coprod \pi_1} && A \ar[d]^-{a} \\
 A \coprod C \ar[r]_-{a \coprod 1} & B \coprod C \ar[r]_-{1 \coprod c} & B }
\end{align*}
Applying the above with $A = G \times_H W$, $a = G \times_H f$, $B = G \times_H X$, $C = G \times_H ((G-H) \times_H X)$ and $c = \epsilon_{G \times_H X}|_C$, we obtain the exponential diagram below,
\begin{align*}
\xymatrix{
 G {\times_H} (\Res_H^G (G {\times_H} W)) \ar[d]^-{G \times_H (1 \coprod (G-H) \times_H f)} \ar[rr]^-{\epsilon_{G \times_H W}} && G {\times_H} W \ar[d]^-{G \times_H f} \\
 G {\times_H} (W \coprod (G{-}H) {\times_H} X) \ar[r]_-{G \times_H (f \coprod 1)} & G {\times_H} (\Res_H^G (G {\times_H} X)) \ar[r]_-{\epsilon_{G \times_H X}} & G {\times_H} X }
\end{align*}
where we have used the fact that the diagram
\begin{align*}
\xymatrix{
 G \times_H ((G-H) \times_H W) \ar[d]_-{G \times_H ((G-H) \times_H f)} \ar[r]^-{\epsilon_{G \times_H W}} & G \times_H W \ar[d]^-{G \times_H f} \\
 G \times_H ((G-H) \times_H X) \ar[r]^-{\epsilon_{G \times_H X}} & G \times_H X }
\end{align*}
is a pullback. This implies the following.
\begin{align*}
	\eta_{\underline{R}} (t_f (w)) &= (G \times_H X \xrightarrow{=} G \times_H X, t_{(f \coprod 1)} n_{i_1} (w)) \\
	                                                  &= (G \times_H W \xrightarrow{G \times_H f} G \times_H X, r_{(1 \coprod (G-H) \times_H f)} n_{i_1} (w))
\end{align*}
Then since the diagram
\begin{align*}
\xymatrix{
 W \ar[d]_-{\eta_W} \ar[rr]^-{=} && W \ar[d]^-{i_1} \\
 \Res_H^G (G \times_H W) \ar[rr]_-{1 \coprod (G-H) \times_H f} && W \coprod (G-H) \times_H X }
\end{align*}
is a pullback, we obtain the following.
\begin{align*}
	\eta_{\underline{R}} (t_f (w)) &= (G \times_H W \xrightarrow{G \times_H f} G \times_H X, n_{\eta_W} (w)) \\
	                                                  &= t_f ((G \times_H W \xrightarrow{=} G \times_H W, n_{\eta_W} (w))) \\
	                                                  &= t_f (\eta_{\underline{R}} (w))
\end{align*}
It remains to show that $\eta_{\underline{R}}$ commutes with norm maps. Again suppose we have an $H$-map $f : W \to X$. For $w \in \underline{R} (W)$ we have
\begin{align*}
	\eta_{\underline{R}} (n_f (w)) = (G \times_H X \xrightarrow{=} G \times_H X, n_{\eta_X} n_f (w)).
\end{align*}
We have $n_{\eta_X} n_f = n_{\Res_H^G (G \times_H f)} n_{\eta_W}$ by the naturality of $\eta$. Now the diagram below is exponential,
\begin{align*}
\xymatrix{
 G \times_H W \ar[d]_-{=} \ar[rr]^-{G \times_H f} && G \times_H X \ar[d]^-{=} \\
 G \times_H W \ar[r]_-{=} & G \times_H W \ar[r]_-{G \times_H f} & G \times_H X }
\end{align*}
so we obtain the following.
\begin{align*}
	n_f (\eta_{\underline{R}} (w)) &= n_f ((G \times_H W \xrightarrow{=} G \times_H W, n_{\eta_W} (w))) \\
	                                                  &= (G \times_H X \xrightarrow{=} G \times_H X, n_{\Res_H^G (G \times_H f)} n_{\eta_W} (w)) \\
	                                                  &= \eta_{\underline{R}} (n_f (w))
\end{align*}
\end{proof}

Next we define a counit map
\begin{align*}
	\epsilon_{\underline{R}} : sN^{G,H} \Res_H^G \underline{R} \to \underline{R}
\end{align*}

for $\underline{R} \in sTamb(G)$. Using that $\Res_H^G \underline{R} (\Res_H^G V) = \underline{R} (G \times_H (\Res_H^G V))$, we define $\epsilon_{\underline{R}}$ as below for $X \in \Fin_G$.
\begin{align*}
	\epsilon_{\underline{R}} (X) : sN^{G,H} \Res_H^G \underline{R} (X) &\to \underline{R} (X) \\
	(V \xrightarrow{j} X, u) &\mapsto t_j n_{\epsilon_V} (u)
\end{align*}

It is clear that this gives a well-defined function on $sN_0^{G,H} \Res_H^G \underline{R} (X)$. It then descends to the quotient $sN^{G,H} \Res_H^G \underline{R} (X)$ since diagrams of the form below are exponential and the distributive law holds in $\underline{R}$.
\begin{align*}
\xymatrix{
 G \times_H (\Res_H^G D_H (W, f, V)) \ar[d]_-{G \times_H e_H} \ar[rr]^-{\epsilon_{D_H (W, f, V)}} && D_H (W, f, V) \ar[d]^-{p} \\
 G \times_H W \ar[r]_-{G \times_H f} & G \times_H (\Res_H^G V) \ar[r]_-{\epsilon_V} & V }
\end{align*}

We now check that this is a map of semi-Tambara functors.

\begin{lem}\label{lem:counitnormresadj}
For any $\underline{R} \in sTamb(G)$, the maps $\epsilon_{\underline{R}} (X)$ determine a map $\epsilon_{\underline{R}}$ in $sTamb(G)$.
\end{lem}
\begin{proof}
First we check that $\epsilon_{\underline{R}}$ commutes with restrictions. Suppose given an element $(V \xrightarrow{j} X, u) \in sN^{G,H} \Res_H^G \underline{R} (X)$. Let $f : Y \to X$ be a $G$-map and form the diagram below, where both squares are pullbacks.
\begin{align*}
\xymatrix{
 G \times_H (\Res_H^G V) \ar[r]^-{\epsilon_V} & V \ar[r]^-{j} & X \\
 G \times_H (\Res_H^G P) \ar[u]^-{G \times_H (\Res_H^G p)} \ar[r]_-{\epsilon_P} & P \ar[r]_-{q} \ar[u]^-{p} & Y \ar[u]_-{f} }
\end{align*}
We now have the following.
\begin{align*}
	\epsilon_{\underline{R}} (r_f ((V \xrightarrow{j} X, u))) &= \epsilon_{\underline{R}} ((P \xrightarrow{q} Y, r_{G \times_H (\Res_H^G p)} (u))) \\
	                                                                                          &= t_q n_{\epsilon_P} r_{G \times_H (\Res_H^G p)} (u) \\
	                                                                                          &= t_q r_p n_{\epsilon_V} (u) \\
	                                                                                          &= r_f t_j n_{\epsilon_V} (u) \\
	                                                                                          &= r_f (\epsilon_{\underline{R}} ((V \xrightarrow{j} X, u)))
\end{align*}
That $\epsilon_{\underline{R}}$ commutes with transfers is trivial. Finally, we must show that it commutes with norms. Hence let $f : X \to Y$ be a $G$-map, and form the exponential diagram below.
\begin{align*}
\xymatrix{
 B \ar[d]_-{e} \ar[rr]^-{b} && C \ar[d]^-{p} \\
 V \ar[r]_-{j} & X \ar[r]_-{f} & Y }
\end{align*}
Then we have the following.
\begin{align*}
	\epsilon_{\underline{R}} (n_f ((V \xrightarrow{j} X, u))) &= \epsilon_{\underline{R}} ((C \xrightarrow{p} Y, n_{G \times_H (\Res_H^G b)} r_{G \times_H (\Res_H^G e)} (u))) \\
	                                                                                           &= t_p n_{\epsilon_C} n_{G \times_H (\Res_H^G b)} r_{G \times_H (\Res_H^G e)} (u) \\
\end{align*}
Now $n_{\epsilon_C} n_{G \times_H (\Res_H^G b)} = n_b n_{\epsilon_B}$ by the naturality of $\epsilon$, and the diagram below is a pullback,
\begin{align*}
\xymatrix{
 G \times_H (\Res_H^G B) \ar[d]_-{\epsilon_B} \ar[rr]^-{G \times_H (\Res_H^G e)} && G \times_H (\Res_H^G V) \ar[d]^-{\epsilon_V} \\
 B \ar[rr]_-{e} && V }
\end{align*}
so we also have $n_{\epsilon_B} r_{G \times_H (\Res_H^G e)} = r_e n_{\epsilon_V}$. Thus we obtain
\begin{align*}
	\epsilon_{\underline{R}} (n_f ((V \xrightarrow{j} X, u))) &= t_p n_b r_e n_{\epsilon_V} (u) \\
	                                                                                          &= n_f t_j n_{\epsilon_V} (u) \\
	                                                                                          &= n_f (\epsilon_{\underline{R}} ((V \xrightarrow{j} X, u))).
\end{align*}
\end{proof}

We can now demonstrate our adjunction for semi-Tambara functors.

\begin{thm}\label{thm:stambnormresadj}
The above natural maps $\eta_{\underline{R}} : \underline{R} \to \Res_H^G sN^{G,H} \underline{R}$ and $\epsilon_{\underline{R}} : sN^{G,H} \Res_H^G \underline{R}\to \underline{R}$ are the unit and counit, respectively, for an adjunction $(sN^{G,H}, \Res_H^G)$ on semi-Tambara functors.
\end{thm}
\begin{proof}
We will demonstrate the following.
\begin{enumerate}[(i)]
\item For any $\underline{R} \in sTamb(H)$ we have $\epsilon_{sN^{G,H} \underline{R}} \circ sN^{G,H} (\eta_{\underline{R}}) = 1$.
\item For any $\underline{R} \in sTamb(G)$ we have $\Res_H^G (\epsilon_{\underline{R}}) \circ \eta_{\Res_H^G \underline{R}} = 1$.
\end{enumerate}
We begin with (i); let $\underline{R} \in sTamb(H)$. Consider the composite below.
\begin{align*}
	sN^{G,H} \underline{R} \xrightarrow{sN^{G,H} (\eta_{\underline{R}})} sN^{G,H} \Res_H^G sN^{G,H} \underline{R} \xrightarrow{\epsilon_{sN^{G,H} \underline{R}}} sN^{G,H} \underline{R}
\end{align*}
First, we have the following.
\begin{gather*}
	sN^{G,H} (\eta_{\underline{R}}) ((V \xrightarrow{j} X, u)) = \\
	\big(V \xrightarrow{j} X, (G \times_H (\Res_H^G V) \xrightarrow{=} G \times_H (\Res_H^G V), n_{\eta_{\Res_H^G V}} (u))\big)
\end{gather*}
Applying $\epsilon_{sN^{G,H} \underline{R}}$, we obtain the following.
\begin{align*}
	t_j n_{\epsilon_V} \big((G \times_H (\Res_H^G V) \xrightarrow{=} G \times_H (\Res_H^G V), n_{\eta_{\Res_H^G V}} (u))\big)
\end{align*}
An easy calculation gives the following.
\begin{gather*}
	n_{\epsilon_V} \big((G \times_H (\Res_H^G V) \xrightarrow{=} G \times_H (\Res_H^G V), n_{\eta_{\Res_H^G V}} (u))\big) = \\
	(V \xrightarrow{=} V, n_{\Res_H^G (\epsilon_V)} n_{\eta_{\Res_H^G V}} (u))
\end{gather*}
Now $\Res_H^G (\epsilon_V) \circ \eta_{\Res_H^G V} = 1$ by the analogue of (ii) for the adjunction $(G \times_H (\spacedash), \Res_H^G)$, so the above element is $(V \xrightarrow{=} V, u)$. Applying $t_j$, we obtain $(V \xrightarrow{j} X, u)$, so (i) holds.\\
\indent Now we check (ii); let $\underline{R} \in sTamb(G)$. Consider the composite below.
\begin{align*}
	\Res_H^G \underline{R} \xrightarrow{\eta_{\Res_H^G \underline{R}}} \Res_H^G sN^{G,H} \Res_H^G \underline{R} \xrightarrow{\Res_H^G (\epsilon_{\underline{R}})} \Res_H^G \underline{R}
\end{align*}
We begin with an element $u \in \Res_H^G \underline{R} (X) = \underline{R} (G \times_H X)$. Applying $\eta_{\Res_H^G \underline{R}}$, we obtain the following.
\begin{align*}
	(G \times_H X \xrightarrow{=} G \times_H X, n_{G \times_H (\eta_X)} (u))
\end{align*}
Applying $\Res_H^G (\epsilon_{\underline{R}})$, we obtain
\begin{align*}
	n_{\epsilon_{G \times_H X}} n_{G \times_H (\eta_X)} (u) = u
\end{align*}
by the analogue of (i) for the adjunction $(G \times_H (\spacedash), \Res_H^G)$.
\end{proof}

\indent \emph{Remark:} Examining the above proof, we see that our norm / restriction adjunction for semi-Tambara functors is deeply related to the induction / restriction adjunction for sets with group action.\\
\indent Next, suppose that $\underline{R_1} \in Tamb(H)$. Then for $\underline{R}_2 \in Tamb(G)$ we have the following.
\begin{align}\label{eq:complnormresadj}
	Hom_{Tamb(H)} (\underline{R}_1, \Res_H^G \underline{R}_2) &= Hom_{sTamb(H)} (\underline{R}_1, \Res_H^G \underline{R}_2) \\
	                                                                                                      &\cong Hom_{sTamb(G)} (sN^{G,H} \underline{R}_1, \underline{R}_2) \nonumber
\end{align}

Hence we define the norm construction on Tambara functors
\begin{align*}
	N^{G,H} : Tamb(H) \to Tamb(G)
\end{align*}

to be the composite of $sN^{G,H}$ with the additive completion functor. For $\underline{R} \in Tamb(H)$ we define
\begin{align*}
	\eta_{\underline{R}} : \underline{R} \to \Res_H^G N^{G,H} \underline{R}
\end{align*}

to be the additive completion of $\eta_{\underline{R}} : \underline{R} \to \Res_H^G sN^{G,H} \underline{R}$, and
\begin{align*}
	\epsilon_{\underline{R}} : N^{G,H} \Res_H^G \underline{R} \to \underline{R}
\end{align*}

the additive completion of $\epsilon_{\underline{R}} : sN^{G,H} \Res_H^G \underline{R} \to \underline{R}$ for $\underline{R} \in Tamb(G)$. Then the following is immediate from~\ref{eq:complnormresadj} and Theorem~\ref{thm:stambnormresadj}.

\begin{thm}\label{thm:tambnormresadj}
The above natural maps $\eta_{\underline{R}}$ and $\epsilon_{\underline{R}}$ are the unit and counit, respectively, for an adjunction $(N^{G,H}, \Res_H^G)$ on Tambara functors.
\end{thm}

Next, recall that the functor $N^{G,H}$ is symmetric monoidal, so it preserves commutative Green functors. Corollary~\ref{cor:intrnormsymmon} describes the structure maps. This gives us two ways of obtaining a commutative Green functor structure on the norm of a Tambara functor. We now check that these are the same.

\begin{prop}\label{prop:twonormsamecomm}
The norm construction on Tambara functors coincides with the norm construction on underlying commutative Green functors.
\end{prop}
\begin{proof}
Let $\underline{R} \in Tamb(H)$, and let $X \in \Fin_G$. Suppose given two elements $(V_1 \xrightarrow{j_1} X, u_1)$ and $(V \xrightarrow{j_2} X, u_2)$ of $N^{G,H} \underline{R} (X)$. We first multiply them using the Tambara functor structure. This produces the element below.
\begin{align*}
	\textstyle n_{1 \coprod 1} \big((V_1 \coprod V_2 \xrightarrow{j_1 \coprod j_2} X \coprod X, (u_1, u_2))\big)
\end{align*}
To compute this, we form the exponential diagram below.
\begin{align*}
\xymatrix{
 (V_1 \times_X V_2) \coprod (V_1 \times_X V_2) \ar[d]_-{\pi_1 \coprod \pi_2} \ar[rr]^-{1 \coprod 1} && V_1 \times_X V_2 \ar[d]^-{j_1 \pi_1 = j_2 \pi_2} \\
 V_1 \coprod V_2 \ar[r]_-{j_1 \coprod j_2} & X \coprod X \ar[r]_-{1 \coprod 1} & X }
\end{align*}
We obtain the element below.
\begin{align*}
	(V_1 \times_X V_2 \xrightarrow{j_1 \pi_1} X, r_{\Res_H^G \pi_1} (u_1) \cdot r_{\Res_H^G \pi_2} (u_2))
\end{align*}
Next we multiply the two elements using the structure map from Corollary~\ref{cor:intrnormsymmon}. We obtain the element below.
\begin{align*}
	r_{\Delta} \big((V_1 \times V_2 \xrightarrow{j_1 \times j_2} X \times X, u_1 \otimes u_2)\big)
\end{align*}
Now the diagram below is a pullback,
\begin{align*}
\xymatrix{
 V_1 \times V_2 \ar[r]^-{j_1 \times j_2} & X \times X \\
 V_1 \times_X V_2 \ar[u]^-{k = \pi_1 \times \pi_2} \ar[r]_-{j_1 \pi_1} & X \ar[u]_-{\Delta} }
\end{align*}
so we obtain the following.
\begin{align*}
	(V_1 \times_X V_2 \xrightarrow{j_1 \pi_1} X, r_{\Res_H^G k} (u_1 \otimes u_2))
\end{align*}
Now we have $u_1 \otimes u_2 = (u_1 \otimes 1) \cdot (1 \otimes u_2) = r_{\Res_H^G \pi_1} (u_1) \cdot r_{\Res_H^G \pi_2} (u_2)$, so that $r_{\Res_H^G k} (u_1 \otimes u_2) = r_{\Res_H^G \pi_1} (u_1) \cdot r_{\Res_H^G \pi_2} (u_2)$.
\end{proof}

Next we wish to identify the adjunction of Theorem~\ref{thm:tambnormresadj} with the one obtained from topology in~\cite{UTamb}. We begin by describing the latter adjunction. Let $\underline{R} \in Tamb(H)$, and let $\HH \underline{R}$ be a cofibrant model for the corresponding Eilenberg MacLane commutative ring spectrum. Corollary~5.10 of~\cite{UTamb} gives the left adjoint of restriction from $Tamb(G)$ to $Tamb(H)$ as
\begin{align*}
	\underline{R} \mapsto \underline{\pi}_0 N_H^G \HH \underline{R},
\end{align*}

with unit map $\eta_{\underline{R}}^{top}$ induced by the map below.
\begin{align}\label{eq:induceetatop}
	\HH \underline{R} \cong \HH \underline{R} \wedge (S^0)^{\wedge (G-H)/H} &\xrightarrow{Id \wedge 1^{\wedge (G-H)/H}} \HH \underline{R} \wedge (\HH \underline{R})^{\wedge (G-H)/H} \\
	&\cong \Res_H^G N_H^G \HH \underline{R} \nonumber
\end{align}

To compare the two adjunctions, we begin by letting $(\HH \underline{R})_c$ be a cofibrant replacement, and considering the norm of the approximation map
\begin{align*}
	N_H^G (\HH \underline{R})_c \to N_H^G \HH \underline{R}.
\end{align*}

Applying $\underline{\pi}_0$, this induces a natural map
\begin{align}\label{eq:compareleftadj}
	\gamma_{\underline{R}} : N^{G,H} \underline{R} \to \underline{\pi}_0 (N_H^G \HH \underline{R})
\end{align}

of Mackey functors. Proposition~\ref{prop:twonormsamecomm} implies that it is a map of commutative Green functors. We must first prove the following.

\begin{lem}\label{lem:compmaptamb}
For any $\underline{R} \in Tamb(H)$, the map $\gamma_{\underline{R}}$ is a map of $G$-Tambara functors.
\end{lem}
\begin{proof}
We need only show that the precomposite of $\gamma_{\underline{R}}$ with the completion map $sN^{G,H} \underline{R} \to N^{G,H} \underline{R}$, which we denote by $s\gamma_{\underline{R}}$, is a map of semi-Tambara functors. Let $X \in \Fin_G$, and let $(V \xrightarrow{j} X, u)$ be an element of $sN^{G,H} \underline{R} (X)$. We compute the image of this element using the notation of Section~\ref{sec:multpush}. We obtain the following.
\begin{align*}
	s\gamma_{\underline{R}} \big((V \xrightarrow{j} X, u)\big) = t_j \pi^{G,H}_{\star} (u)
\end{align*}
Now let $f : X \to Y$ be a $G$-map, and form the exponential diagram below.
\begin{align*}
\xymatrix{
 D \times_Y X \ar[d]_-{e} \ar[rr]^-{\pi_1} && D \ar[d]^-{p} \\
 V \ar[r]_-{j} & X \ar[r]_-{f} & Y }
\end{align*}
We obtain the following.
\begin{align*}
	n_f \big(s\gamma_{\underline{R}} \big((V \xrightarrow{j} X, u)\big)\big) = t_p n_{\pi_1} r_e \pi^{G,H}_{\star} (u)
\end{align*}
A simple topological calculation yields
\begin{align*}
	n_{\pi_1} r_e \pi^{G,H}_{\star} (u) = n_{\pi_1} \pi^{G,H}_{\star} (r_{\Res_H^G e} (u)) = \pi^{G,H}_{\star} (n_{\Res_H^G \pi_1} r_{\Res_H^G e} (u)),
\end{align*}
so that we have the following.
\begin{align*}
	n_f \big(s\gamma_{\underline{R}} \big((V \xrightarrow{j} X, u)\big)\big) &= t_p \pi^{G,H}_{\star} (n_{\Res_H^G \pi_1} r_{\Res_H^G e} (u)) \\
	                                                                                                                     &= s\gamma_{\underline{R}} \big((D \xrightarrow{p} Y, n_{\Res_H^G \pi_1} r_{\Res_H^G e} (u))\big) \\
	                                                                                                                     &= s\gamma_{\underline{R}} \big(n_f \big((V \xrightarrow{j} X, u)\big)\big)
\end{align*}
\end{proof}

We can now relate our two adjunctions.

\begin{thm}\label{thm:twoadj}
The natural map~\ref{eq:compareleftadj} is the unique natural isomorphism relating the two adjunctions $(N^{G,H}, \Res_H^G)$ and $(\underline{\pi}_0 N_H^G \HH, \Res_H^G)$ on Tambara functors.
\end{thm}
\begin{proof}
The unique isomorphism relating the two adjunctions is the unique natural transformation $\sigma$ such that the following diagram commutes for all $\underline{R} \in Tamb(H)$.
\begin{align*}
\xymatrix{
 & \underline{R} \ar[dl]_-{\eta_{\underline{R}}} \ar[dr]^-{\eta^{top}_{\underline{R}}} & \\
 \Res_H^G N^{G,H} \underline{R} \ar[rr]_-{\Res_H^G \sigma_{\underline{R}}} && \Res_H^G \underline{\pi}_0 N_H^G \HH \underline{R} }
\end{align*}
Thus, we must show that this diagram commutes for $\sigma = \gamma$. Let $X$ be an $H$-set, and let $u \in \underline{R} (X)$. Then we have
\begin{align*}
	\Res_H^G (\gamma_{\underline{R}}) \eta_{\underline{R}} (u) = \pi^{G,H}_{\star} (n_{\eta_X} (u)) \in \underline{\pi}_0 N_H^G \HH \underline{R} (G \times_H X).
\end{align*}
To compute this, we may first assume that $\HH \underline{R}$ is a cofibrant and fibrant commutative ring spectrum. Then its underlying spectrum is positive fibrant, so $u$ may be represented by a map
\begin{align*}
	u : F_1 S^1 \wedge X_+ \to \HH \underline{R}.
\end{align*}
Then $n_{\eta_X} (u)$ may be represented by the map
\begin{align*}
	(u, 1) : (const_X F_1 S^1, const_{(G-H) \times_H X} S^0) \to const_{\Res_H^G (G \times_H X)} \HH \underline{R}.
\end{align*}
We now take the multiplicative push forward $\pi^{G,H}_{\star}$ and restrict to the copy of $X$ corresponding the identity coset. Letting $\{ g_i \}$ be our set of coset representatives, we note that for any $x \in X \subseteq \Res_H^G (G \times_H X)$, we have $g_i^{-1} \cdot (1,x) \in X$ only when $g_i = 1$. It follows from the explicit description of multiplicative push forwards in Section~\ref{sec:multpush} that we can describe the result precisely as the map $F_1 S^1 \wedge X_+ \to \Res_H^G N_H^G \HH \underline{R}$ which is equal, on the summand corresponding to $x \in X$, to the composite below.
\begin{align*}
	F_1 S^1 \cong F_1 S^1 \wedge (S^0)^{\wedge (G-H)/H} &\xrightarrow{u \wedge 1^{\wedge (G-H)/H}} \HH \underline{R} \wedge (\HH \underline{R})^{\wedge (G-H)/H} \\
	&\cong \Res_H^G N_H^G \HH \underline{R}
\end{align*}
Examining~\ref{eq:induceetatop}, we see that this is precisely $\eta^{top}_{\underline{R}} (u)$.
\end{proof}

\indent \emph{Remark:} Theorem~\ref{thm:twoadj} implies formally that the following diagram commutes for any $\underline{R} \in Tamb(G)$, where the map on the right is induced by the counit of the adjunction $(N_H^G, \Res_H^G)$ on commutative ring spectra.
\begin{align*}
\xymatrix{
 N^{G,H} \Res_H^G \underline{R} \ar[dr]_-{\epsilon_{\underline{R}}} \ar[rr]^-{\gamma_{\Res_H^G \underline{R}}}_-{\cong} && \underline{\pi}_0 N_H^G \Res_H^G \HH \underline{R} \ar[dl]^-{\epsilon_{\underline{R}}^{top}} \\ 
 & \underline{R} & }
\end{align*}
\indent We can now generalize a special case of Proposition~B.63 of~\cite{HHR}.

\begin{cor}\label{cor:normrighthtpytype}
Let $X \in comm_H$ be cofibrant and $(-1)$-connected, and let $Y \in Sp_H$ be cofibrant and $(-1)$-connected. Then for any map $f : Y \to X$ which induces an isomorphism on $\underline{\pi}_0$, the map
\begin{align*}
	N_H^G f : N_H^G Y \to N_H^G X
\end{align*}
induces an isomorphism on $\underline{\pi}_0$.
\end{cor}
\begin{proof}
Let $\underline{R} = \underline{\pi}_0 X$. Then by Proposition~3.7 of~\cite{UTamb}, we can construct a cofibration $p : X \to \HH \underline{R}$ in $comm_H$ which induces the identity on $\underline{\pi}_0$. Now let $k : Z \to \HH \underline{R}$ be a trivial fibration in $Sp_H$ with $Z$ cofibrant. Then we can complete the diagram below by the lifting axiom.
\begin{align*}
\xymatrix{
 Y \ar[d]_-{f} \ar[r]^-{q} & Z \ar[d]^-{k} \\
 X \ar[r]_-{p} & \HH \underline{R} }
\end{align*}
The map $q$ induces an isomorphism on $\underline{\pi}_0$ since the other three maps do. Then by Lemma~\ref{lem:norm-1conn}, $\underline{\pi}_0 N_H^G (q)$ is an isomorphism. Also, $\underline{\pi}_0 N_H^G (k)$ is an isomorphism by Theorem~\ref{thm:twoadj}. Thus, it suffices to show that $\underline{\pi}_0 N_H^G (p)$ is an isomorphism as well. Let $\underline{B} \in Tamb(G)$. We have the following chain of isomorphisms by Theorems~\ref{thm:mapstoemcomm} and~\ref{thm:homemcomm}, and the fact that the norm functor gives the left adjoint of restriction from $comm_G$ to $comm_H$.
\begin{align*}
	Hom_{Tamb(G)} (\underline{\pi}_0 N_H^G X, \underline{B}) &\cong Hom_{Ho(comm_G)} (N_H^G X, \HH \underline{B}) \\
	                                                                                                 &\cong Hom_{Ho(comm_H)} (X, \Res_H^G \HH \underline{B}) \\
	                                                                                                 &\cong Hom_{Tamb(H)} (\underline{\pi}_0 X, \Res_H^G \underline{B})
\end{align*}
We have similar isomorphisms for $\HH \underline{R}$. The result then follows from Yoneda's Lemma.
\end{proof}

\begin{cor}\label{cor:gsymmonrighthtpytype}
Let $X \in comm_G$ be cofibrant and $(-1)$-connected, and let $Y \in Sp_G$ be cofibrant and $(-1)$-connected. Then for any map $f : Y \to X$ which induces an isomorphism on $\underline{\pi}_0$ and any $T \in \Fin_G$, the map
\begin{align*}
	f^{\wedge T} : Y^{\wedge T} \to X^{\wedge T}
\end{align*}
induces an isomorphism on $\underline{\pi}_0$.
\end{cor}
\begin{proof}
Let $T \cong \coprod_i G/H_i$. Then we have a natural isomorphism for $Z \in Sp_G$ as below.
\begin{align}\label{eq:gsymmonrighthtpytype}
	Z^{\wedge T} \cong \bigwedge_i N_{H_i}^G \Res_{H_i}^G Z
\end{align}
For each $i$ the map $\Res_{H_i}^G f$ induces an isomorphism on $\underline{\pi}_0$, so Corollary~\ref{cor:normrighthtpytype} implies that $N_{H_i}^G \Res_{H_i}^G f$ induces an isomorphism on $\underline{\pi}_0$. Now the $N_{H_i}^G \Res_{H_i}^G Y$ are cofibrant in $Sp_G$ and the $N_{H_i}^G \Res_{H_i}^G X$ are cofibrant in $comm_G$, so all of these spectra are flat. It follows that the smash products~\ref{eq:gsymmonrighthtpytype} for $Z = Y$ and $Z = X$ are derived smash products. Since all of these spectra are $(-1)$-connected as well, these smash products induce tensor products on $\underline{\pi}_0$.
\end{proof}

\section{Tambara Functors and Multiplicative Push Forwards of Mackey Functors}\label{sec:multpushmackey}

In this section we give an alternative characterization of Tambara functors in terms of \emph{multiplicative push forwards} of Mackey functors. We then relate this topologically to $\underline{\pi}_0 \CC \HH$ and give an alternative formula for $\TT$ in these terms. We begin by recalling the notation and concepts of Section~\ref{sec:multpush}. Letting $T \in \Fin_G$, recall that $\BB_G (T)$ is the translation category of $T$, and that $Sp_{\BB_G (T)}$ denotes the category of functors from $\BB_G (T)$ to orthogonal spectra. We call these $\BB_G (T)$-spectra. We would like a corresponding notion of $\BB_G (T)$-Mackey functors. These should correspond to $\BB_G (T)$-spectra which are pointwise Eilenberg MacLane. Suppose that $X \in Sp_{\BB_G (T)}$ and that $g$ is a morphism in $\BB_G (T)$ from $t_1$ to $t_2$ (that is, $g \cdot t_1 = t_2$). Letting $H_t$ be the stabilizer of $t$ for all $t \in T$, we have $H_{t_2} = g H_{t_1} g^{-1}$. Hence, we have a conjugation isomorphism as below.
\begin{align*}
	I_g : H_{t_1} &\xrightarrow{\cong} H_{t_2} \\
	                   h &\mapsto g h g^{-1}
\end{align*}

This isomorphism clearly induces a corresponding isomorphism of categories $(I_g^{-1})^* : Sp_{H_{t_1}} \xrightarrow{\cong} Sp_{H_{t_2}}$. We can then describe our $\BB_G (T)$-spectrum $X$ equivalently as a collection $\{ X_t \in Sp_{H_t} \}_{t \in T}$ together with isomorphisms
\begin{align}\label{eq:bgtspectwopts}
	(I_g^{-1})^* X_t \xrightarrow[\cong]{g} X_{gt}
\end{align}

for each $t \in T$ and $g \in G$ which are suitably functorial, subject to the condition that the above map is the action by $g$ when $g \in H_t$. Now pullback of group action along $I_g$ induces an isomorphism of Burnside categories $Burn(H_{gt}) \xrightarrow{\cong} Burn(H_t)$ which we shall denote by $A \mapsto A^g$, and precomposition with this functor then induces an isomorphism $Mack(H_t) \xrightarrow{\cong} Mack(H_{gt})$ which we shall denote by
\begin{align*}
	\underline{M} \mapsto {}^g\underline{M}.
\end{align*}

Now for any $n \in \ZZ$ we clearly have a canonical isomorphism
\begin{align*}
	\underline{\pi}_n (I_g^{-1})^* X_t \cong {}^g(\underline{\pi}_n X_t),
\end{align*}

and hence we have isomorphisms as below.
\begin{align*}
	{}^g(\underline{\pi}_n X_t) \xrightarrow[\cong]{g} \underline{\pi}_n X_{gt}
\end{align*}

These satisfy the condition that, when $g \in H_t$, they coincide with the restriction maps
\begin{align*}
	{}^g(\underline{\pi}_n X_t) (V) = \underline{\pi}_n X_t (V^g) \xrightarrow{\cong} \underline{\pi}_n X_t (V)
\end{align*}

induced by the isomorphisms below for each $V \in \Fin_{H_t}$.
\begin{align*}
	V &\xrightarrow{\cong} V^g \\
	v &\mapsto g \cdot v
\end{align*}

Thus, a $\BB_G (T)$-spectrum $X$ is pointwise Eilenberg MacLane if and only if it is Eilenberg MacLane at any set of orbit representatives for $T$, and similarly for pointwise $(-1)$-connected objects. We now describe the analogue of the Burnside category. We define $Burn(\BB_G (T))$ to be the category whose objects are finite $G$-sets over $T$, where the hom sets are given by Grothendieck groups of correspondences over $T$. We now define $\BB_G (T)$-Mackey functors.

\begin{definition}\label{def:bgtmackey}
Let $T \in \Fin_G$. A $\BB_G (T)$-Mackey functor $\underline{M}$ is a contravariant, additive functor from $Burn(\BB_G (T))$ to the category of abelian groups. We denote the category of $\BB_G (T)$-Mackey functors by $Mack(\BB_G (T))$.
\end{definition}

Of course, if $T \cong \coprod_i G/H_i$ then $Burn(\BB_G (T)) \cong \prod_i Burn(H_i)$, so we have an equivalence of categories as below.
\begin{align}\label{eq:mackbgtsimple}
	Mack(\BB_G (T)) \cong \prod_i Mack(H_i)
\end{align}

Next note that $G$-sets over $T$ are equivalent to $\BB_G (T)$-sets: given a map $Y \to T$ we can take $Y_t$ to be the preimage of $t \in T$, and given $\{ Y_t \}_{t \in T}$ we can form $\coprod_{t \in T} Y_t \to T$. Hence we can associate suspension $\BB_G (T)$-spectra to the objects of $Burn(\BB_G (T))$. We have an equivalence $Sp_{\BB_G (T)} \cong \prod_i Sp_{H_i}$ corresponding to~\ref{eq:mackbgtsimple} which we can use to define cotransfer maps between these suspension $\BB_G (T)$-spectra. It is clear from~\ref{eq:bgtspectwopts} that these maps do not depend on the choice of orbit representatives. Hence we obtain the following.

\begin{prop}\label{prop:bgtmackeyspectra}
The homotopy category of suspension $\BB_G (T)$-spectra of finite $G$-sets over $T$ is canonically equivalent to $Burn(\BB_G (T))$. For each $n \in \ZZ$ there is a functor as below.
\begin{align*}
	\underline{\pi}_n : Ho(Sp_{\BB_G (T)}) &\to Mack(\BB_G (T)) \\
	                                                             X &\mapsto [ S^n \wedge \Sigma^{\infty} (\spacedash)_+, X]
\end{align*}
The functor $\underline{\pi}_0$ induces an equivalence from the homotopy category of pointwise Eilenberg MacLane $\BB_G (T)$-spectra to $Mack(\BB_G (T))$. For any pointwise $(-1)$-connected $\BB_G (T)$-spectrum $X$ and $\underline{M} \in Mack(\BB_G (T))$, we have a natural isomorphism as below.
\begin{align*}
	(\underline{\pi}_0)_* : Hom_{Ho(Sp_{\BB_G (T)})} (X, \HH \underline{M}) \xrightarrow{\cong} Hom_{Mack(\BB_G (T))} (\underline{\pi}_0 X, \underline{M})
\end{align*}
\end{prop}

Now let $\underline{M} \in Mack(G)$. Letting $const_T \underline{M} \defeq \underline{\pi}_0 (const_T \HH \underline{M})$, and recalling the adjunction between $const$ and additive push forward noted in Section~\ref{sec:multpush}, we obtain a canonical isomorphism as below for each $T \in \Fin_G$.
\begin{align*}
	\underline{M} (T) \cong Hom_{Mack(\BB_G (T))} (const_T \underline{A}, const_T \underline{M})
\end{align*}

We can now define additive and multiplicative push forwards of these objects. Let $i : U \to V$ be a $G$-map, and let $\underline{M} \in Mack(\BB_G (U))$. We define the \emph{additive push forward} associated to $i$ in terms of the topological push forward, as below.
\begin{align*}
	i_* : Mack(\BB_G (U)) &\to Mack(\BB_G (V)) \\
	            \underline{M} &\mapsto \underline{\pi}_0 i_* \HH\underline{M}
\end{align*}

We also define the \emph{pullback along $i$} by replacing $i_* \HH\underline{M}$ in the above definition with $i^* \HH\underline{M}$, and note that these are adjoint functors. We define the \emph{multiplicative push forward} similarly, as below.
\begin{align*}
	i_{\star} : Mack(\BB_G (U)) &\to Mack(\BB_G (V)) \\
	                      \underline{M} &\mapsto \underline{\pi}_0 i_{\star} \HH\underline{M}
\end{align*}

Note that in the above definitions we may replace $\HH\underline{M}$ with any $(-1)$-connected object with the correct $\underline{\pi}_0$. It follows that for any maps $i : U \to V$ and $j : V \to W$ in $\Fin_G$ we have canonical isomorphisms $j_* i_* \cong (j \circ i)_*$, $j_{\star} i_{\star} \cong (j \circ i)_{\star}$ and $i^* j^* \cong (j \circ i)^*$.\\
\indent We can describe these constructions algebraically. Let $i : U \to V$ be a map in $\Fin_G$. Then we obtain a restriction functor
\begin{align*}
	\Res(i) : Burn(\BB_G (V)) &\to Burn(\BB_G (U)) \\
	                                    W &\mapsto W \times_V U
\end{align*}

and an induction functor as below.
\begin{align*}
	\Ind(i) : Burn(\BB_G (U)) &\to Burn(\BB_G (V)) \\
	    (W \xrightarrow{k} U) &\mapsto (W \xrightarrow{i \circ k} V)
\end{align*}

Using the topological adjunction between $i_*$ and $i^*$, we calculate
\begin{align*}
	i^* \underline{M} \cong \underline{M} \circ Ind(i)
\end{align*}

for $\underline{M} \in Mack(\BB_G (V))$. Since $i_*$ is left adjoint to $i^*$, and $\Ind(i)$ is left adjoint to $\Res(i)$ (on opposite Burnside categories), it is formal that
\begin{align*}
	i_* \underline{M} \cong \underline{M} \circ \Res(i)
\end{align*}

for $\underline{M} \in Mack(\BB_G (U))$. (In fact, $i_*$ and $i^*$ are both left and right adjoint to one another, since $\Ind(i)$ and $\Res(i)$ are.) Note that these induction and restriction functors, for $i : G/K \to G/H$ the canonical projection for pairs of subgroups $K \subseteq H$, coincide with the usual induction and restriction functors.\\
\indent Next we give an algebraic description of the multiplicative push forward $i_{\star}$.

\begin{definition}\label{def:intrmultpushmackey}
Let  $\underline{M} \in Mack(\BB_G (U))$ and let $i : U \to V$ be a map in $\Fin_G$. Let $X$ be a finite $G$-set over $V$. We define $F(i, \underline{M}) (X)$ to be the quotient of the free abelian group on the pairs $(Y \xrightarrow{j} X, y)$, where $j$ is a map in $\Fin_G / V$ and $y \in \underline{M} (Y \times_V U)$, by the relations
\begin{enumerate}[(i)]
\item $(Y \xrightarrow{j} X, y) = (Y' \xrightarrow{j'} X, y')$ when there is a commutative diagram
\begin{align*}
\xymatrix{
 Y \ar[r]^-{j} & X \\
 Y' \ar[u]^-{f}_-{\cong} \ar[ur]_-{j'} & }
\end{align*}
such that $f$ is an isomorphism and $r_{f \times_V 1} (u) = u'$,
\item $(Y_1 \coprod Y_2 \xrightarrow{j_1 \coprod j_2} X, (y_1, y_2)) = (Y_1 \xrightarrow{j_1} X, y_1) + (Y_2 \xrightarrow{j_2} X, y_2)$, and
\item $(Y \xrightarrow{j} X, t_f (w)) = (D_i (W, f, V) \xrightarrow{j \circ p} X, r_e (w))$ for any $G$-map $f : W \to Y \times_V U$ over $U$, where the diagram below is exponential.
\begin{align*}
\xymatrix{
 D_i (W, f, V) \times_V U \ar[d]_-{e} \ar[rr]^-{\pi_1} && D_i (W, f, V) \ar[d]^-{p} \\
 W \ar[r]_-{f} & Y \times_V U \ar[r]_-{\pi_1} & Y }
\end{align*}
\end{enumerate}
\end{definition}

We define transfers for $F(i, \underline{M})$ by composition on the generators, and restrictions by pullback. Arguing as before with $F(T, \underline{M})$ and $N^{G,H}$, we see that these are well-defined and determine a $\BB_G (V)$-Mackey functor $F(i, \underline{M})$. Note that when $i : T \to \ast$ we have a precise identification
\begin{align*}
	F(i, i^* \underline{M}) \cong F(T, \underline{M})
\end{align*}

for any $\underline{M} \in Mack(G) = Mack(\BB_G (\ast))$, and for any subgroup $H \subseteq G$ and $i : G/H \to \ast$ we have a precise identification
\begin{align*}
	F\big(i, (Z \xrightarrow{k} G/H) \mapsto \underline{M} (k^{-1} (H))\big) \cong N^{G,H} \underline{M}
\end{align*}

for any $\underline{M} \in Mack(H)$.\\
\indent Next we require a comparison map from $F(i, \underline{M})$ to $i_{\star} \underline{M}$. We shall require the following proposition. The proof is left to the interested reader.

\begin{prop}\label{prop:basicmultpushfacts}
The following conclusions hold.
\begin{enumerate}[(i)]
\item For any $G$-map $i : U \to V$ there is a canonical isomorphism $i_{\star} (const_U S^0) \cong const_V S^0$.
\item For any pullback diagram of $G$-sets as below there are canonical natural isomorphisms $j^* i_{\star} X \cong q_{\star} p^* X$ and $j^* i_* X \cong q_* p^* X$ for $X \in Sp_{\BB_G (U)}$.
\begin{align*}
\xymatrix{
 P \ar[d]_-{q} \ar[r]^-{p} & U \ar[d]^-{i} \\
 W \ar[r]_-{j} & V }
\end{align*}
\item For any $G$-maps $j : W \to U$ and $i : U \to V$ there is a canonical natural isomorphism $i_{\star} j_* X \cong p_* f_{\star} e^* X$ for $X \in Sp_{\BB_G (W)}$, where the diagram below is exponential.
\begin{align}\label{eq:genericexp}
\xymatrix{
 E \ar[d]_-{e} \ar[rr]^-{f} && D \ar[d]^-{p} \\
 W \ar[r]_-{j} & U \ar[r]_-{i} & V }
\end{align}
\item If~\ref{eq:genericexp} is exponential then, under the isomorphisms (i) and (iii), $i_{\star}$ of the cotransfer $const_U S^0 \to j_* const_W S^0$ is equal to the cotransfer $const_V S^0 \to p_* const_D S^0$.
\end{enumerate}
\end{prop}

Now let $i : U \to V$ be a $G$-map, $X \in Sp_{\BB_G (U)}$ and $j : W \to V$ another $G$-map. Form the pullback diagram below.
\begin{align*}
\xymatrix{
 W \times_V U \ar[d]_-{q} \ar[r]^-{p} & U \ar[d]^-{i} \\
 W \ar[r]_-{j} & V }
\end{align*}

Utilizing Proposition~\ref{prop:basicmultpushfacts}, we define a homotopy operation
\begin{align*}
	\pi^i_{\star} : \underline{\pi}_0 X (W \times_V U) \to \underline{\pi}_0 i_{\star} X (W)
\end{align*}

by the commutative diagram below.
\begin{align*}
\xymatrix{
 [\Sigma^{\infty} (W \times_V U)_+, X] \ar[d]_-{\cong} \ar[r]^-{\pi^i_{\star}} & [\Sigma^{\infty} W_+, i_{\star} X] \ar[d]^-{\cong} \\
 [const_{W \times_V U} S^0, p^* X] \ar[r]_-{q_{\star}} & [const_W S^0, q_{\star} p^* X \cong j^* i_{\star} X] }
\end{align*}

Applying this to $X = \HH \underline{M}$ for $\underline{M} \in Mack(\BB_G (U))$, we now define our comparison map as below for $i : U \to V$ and arbitrary $X \in \Fin_G / V$.
\begin{align*}
	\Theta^i_{\underline{M}} (X) : F(i, \underline{M}) (X) &\to i_{\star} \underline{M} (X) \\
	                                                 (Y \xrightarrow{j} X, y) &\mapsto t_j \pi^i_{\star} (y)
\end{align*}

The fact that this map respects relation (iii) of Definition~\ref{def:intrmultpushmackey} follows directly from part (iv) of Proposition~\ref{prop:basicmultpushfacts}. It is clear that it commutes with transfers. Part (ii) of Proposition~\ref{prop:basicmultpushfacts} implies that $\Theta^i_{\underline{M}}$ commutes with restrictions. Hence we have a natural transformation $\Theta^i$. We can now identify multiplicative push forwards of Mackey functors algebraically.

\begin{thm}\label{thm:identmultpushmackey}
For any $i : U \to V$ and any $\underline{M} \in Mack(\BB_G (U))$, the map $\Theta^i_{\underline{M}}$ is an isomorphism.
\end{thm}
\begin{proof}
Firstly, suppose that $V = V_1 \coprod V_2$, so that $U = U_1 \coprod U_2$ and $i = i_1 \coprod i_2$. We may also write $\underline{M} = (\underline{M}_1, \underline{M}_2)$. One easily sees that we have a commutative diagram as below.
\begin{align*}
\xymatrix{
 F(i, \underline{M}) \ar[d]_-{\cong} \ar[rr]^-{\Theta^i_{\underline{M}}} && i_{\star} \underline{M} \ar[d]^-{\cong} \\
 \big(F(i_1, \underline{M}_1), F(i_2, \underline{M}_2)\big) \ar[rr]_-{(\Theta^{i_1}_{\underline{M}_1}, \Theta^{i_2}_{\underline{M}_2})} && ({i_1}_{\star} \underline{M}_1, {i_2}_{\star} \underline{M}_2) }
\end{align*}
Hence we are reduced to the case $V = G/H$, where $H$ is some subgroup of $G$. Now $G$-sets $X$ over $G/H$ correspond to $H$-sets under the correspondence
\begin{align*}
	(j : X \to G/H) \mapsto j^{-1} (H).
\end{align*}
It follows that $U \cong G \times_H U'$ and $i \cong G \times_H i'$ for some $H$-map $i'$. Also, $\underline{M}$ corresponds to an $\underline{M}' \in Mack(\BB_H (U'))$. One easily sees that we have a commutative diagram as below for any $G$-set $X$ over $G/H$.
\begin{align*}
\xymatrix{
 F(i, \underline{M}) (X \cong G \times_H X') \ar[d]_-{\cong} \ar[r]^-{\Theta^i_{\underline{M}}} & i_{\star} \underline{M} (X \cong G \times_H X') \ar[d]^-{\cong} \\
 F(i', \underline{M}') (X') \ar[r]_-{\Theta^{i'}_{\underline{M}'}} & i'_{\star} \underline{M}' (X') }
\end{align*}
Hence we are reduced to the case $V = \ast$.\\
\indent First suppose $U = \emptyset$. Then $i_{\star} \underline{M} \equiv \underline{A}$ and the result follows by inspection. Hence suppose $U$ is nonempty. We may suppose inductively that the theorem holds for all proper subgroups of $G$. If $H$ is a subgroup of $G$ and $X$ is an $H$-set then we may identify the map $\Theta^i_{\underline{M}} (G \times_H X)$ with $\Theta^{\Res_H^G i}_{\Res_H^G \underline{M}} (X)$ using similar arguments to those in the preceding paragraph. Hence, we may assume that $\Theta^i_{\underline{M}}$ is an isomorphism on all levels $G/H$ for $H$ a proper subgroup of $G$. We now reduce to the case where $U$ is an orbit. Suppose $U = U_1 \coprod U_2$, $\underline{M} = (\underline{M}_1, \underline{M}_2)$ and $i = i_1 \coprod i_2$. Then we have from topology an isomorphism
\begin{align*}
	{i_1}_{\star} \underline{M}_1 \otimes {i_2}_{\star} \underline{M}_2 \xrightarrow{\cong} i_{\star} \underline{M}.
\end{align*}
We now define a pairing as below.
\begin{align*}
	F(i_1, \underline{M}_1) \otimes F(i_2, \underline{M}_2) &\to F(i, \underline{M}) \\
	(Y_1 \xrightarrow{j_1} X_1, y_1) \otimes (Y_2 \xrightarrow{j_2} X_2, y_2) &\mapsto \\
	\big(Y_1 \times Y_2 \xrightarrow{j_1 \times j_2} &X_1 \times X_2, (r_{\pi_{Y_1 \times U_1}} (y_1), r_{\pi_{Y_1 \times U_2}} (y_2))\big) 
\end{align*}
Examining the proof of Proposition~\ref{prop:intrpairing}, we see that it applies equally well here, so that the above pairing is well-defined. Using the explicit description of $\Theta^i$, we obtain a commutative diagram as below.
\begin{align*}
\xymatrix{
 F(i_1, \underline{M}_1) \otimes F(i_1, \underline{M}_2) \ar[d]_-{\Theta^{i_1}_{\underline{M}_1} \otimes \Theta^{i_2}_{\underline{M}_2}} \ar[r] & F(i, \underline{M}) \ar[d]^-{\Theta^i_{\underline{M}}} \\
 {i_1}_{\star} \underline{M}_1 \otimes {i_2}_{\star} \underline{M}_2 \ar[r]_-{\cong} & i_{\star} \underline{M} }
\end{align*}
Assuming that the statement of the theorem holds for $i_1$ and $i_2$, the left vertical map above is an isomorphism, so $\Theta^i_{\underline{M}}$ has a section. By our inductive hypothesis, this section must be an isomorphism on all proper subgroups of $G$. Hence, to show that the top horizontal map above is surjective, it suffices to check that it is surjective on geometric fixed points. The geometric fixed points of $F(i, \underline{M})$ are generated by the elements of the form $(\ast \xrightarrow{=} \ast, (y_1, y_2))$. Such an element is the image under the top horizontal map of $(\ast \xrightarrow{=} \ast, y_1) \otimes (\ast \xrightarrow{=} \ast, y_2)$. Thus, the top horizontal map is an isomorphism, so $\Theta^i_{\underline{M}}$ is as well.\\
\indent Finally, we must prove the statement for $i$ of the form $G/H \to \ast$. In this case, $\underline{M}$ can be identified with an $H$-Mackey functor $\underline{M}'$, and $\Theta^i_{\underline{M}}$ can be identified with the map $\Theta^{G,H}_{\underline{M}'}$ of Subsection~\ref{subsec:norm}.
\end{proof}

We will not need this algebraic description in what follows; we will instead rely on the topological definitions. We will need the following consequences of Proposition~\ref{prop:basicmultpushfacts}.

\begin{cor}\label{cor:basicmultpushmackey}
The following conclusions hold.
\begin{enumerate}[(i)]
\item For any $G$-map $i : U \to V$ there is a canonical isomorphism $i_{\star} (const_U \underline{A}) \cong const_V \underline{A}$.
\item For any pullback diagram of $G$-sets as below there are canonical natural isomorphisms $j^* i_{\star} \underline{M} \cong q_{\star} p^* \underline{M}$ and $j^* i_* \underline{M} \cong q_* p^* \underline{M}$ for $\underline{M} \in Mack(\BB_G (U))$.
\begin{align*}
\xymatrix{
 P \ar[d]_-{q} \ar[r]^-{p} & U \ar[d]^-{i} \\
 W \ar[r]_-{j} & V }
\end{align*}
\item For any $G$-maps $j : W \to U$ and $i : U \to V$ there is a canonical natural isomorphism $i_{\star} j_* \underline{M} \cong p_* f_{\star} e^* \underline{M}$ for $\underline{M} \in Mack(\BB_G (W))$, where the diagram below is exponential.
\begin{align*}
\xymatrix{
 E \ar[d]_-{e} \ar[rr]^-{f} && D \ar[d]^-{p} \\
 W \ar[r]_-{j} & U \ar[r]_-{i} & V }
\end{align*}
\end{enumerate}
\end{cor}

We now give an alternative characterization of Tambara functors. In the following, we note that for any $G$-map $i : U \to V$ we have the identification $i^* (const_V \underline{M}) \cong const_U \underline{M}$, and we define the \emph{fold map} of $i$ to be the counit $\nabla_i : i_* (const_U \underline{M}) \to const_V \underline{M}$ of the $(i_*, i^*)$ adjunction.

\begin{definition}\label{def:multmackeyfunct}
A \emph{multiplicative Mackey functor} over $G$ is a $G$-Mackey functor $\underline{M}$ together with maps $\mu_i : i_{\star} (const_U \underline{M}) \to const_V \underline{M}$ for all $i : U \to V$ in $\Fin_G$ such that the following conditions hold.
\begin{enumerate}[(i)]
\item If $i : \ast \xrightarrow{=} \ast$ then $\mu_i = Id$.
\item If $i : U \to V$ and $j : V \to W$ then the diagram below commutes.
\begin{align*}
\xymatrix{
 j_{\star} (i_{\star} (const_U \underline{M})) \ar[d]_-{\cong} \ar[r]^-{j_{\star} (\mu_i)} & j_{\star} (const_V \underline{M}) \ar[d]^-{\mu_j} \\
 (j \circ i)_{\star} (const_U \underline{M}) \ar[r]_-{\mu_{j \circ i}} & const_W \underline{M} }
\end{align*}
\item If the square below is a pullback in $\Fin_G$ then the triangle commutes.
\begin{align*}
\xymatrix{
 P \ar[d]_-{q} \ar[r]^-{p} & U \ar[d]^-{i} & q_{\star} (const_P \underline{M}) \ar[d]_-{\mu_q} \ar[r]^-{\cong} & j^* i_{\star} (const_U \underline{M}) \ar[dl]^-{j^* (\mu_i)} \\
 W \ar[r]_-{j} & V & const_W \underline{M} & }
\end{align*}
\item If the rectangle below is exponential then the triangle commutes.
\begin{align*}
\xymatrix{
 E \ar[d]_-{e} \ar[rr]^-{f} && D \ar[d]^-{p} & p_* f_{\star} (const_E \underline{M}) \ar[r]^-{\cong} \ar[dr]_-{\nabla_p \circ p_* (\mu_f)} & i_{\star} j_* (const_W \underline{M}) \ar[d]^-{\mu_i \circ i_{\star} (\nabla_j)} \\
 W \ar[r]_-{j} & U \ar[r]_-{i} & V & & const_V \underline{M} }
\end{align*}
\end{enumerate}
We denote the category of multiplicative Mackey functors over $G$ by $MultMack(G)$.
\end{definition}

\begin{thm}\label{thm:multmackeqtamb}
For any finite group $G$, $Tamb(G)$ is isomorphic to $MultMack(G)$.
\end{thm}
\begin{proof}
First let $\underline{R} \in Tamb(G)$. We shall construct a family of maps $\mu_i$ as follows. Let $\HH \underline{R}$ be a cofibrant commutative ring spectrum. Corollaries~\ref{cor:normrighthtpytype} and~\ref{cor:gsymmonrighthtpytype} imply that repeated multiplicative push forwards of constant diagrams at $\HH \underline{R}$ have the correct $\underline{\pi}_0$. Hence, for each $G$-map $i : U \to V$ we may define $\mu_i$ to be $\underline{\pi}_0$ of the evident multiplication map $i_{\star} (const_U \HH \underline{R}) \to const_V \HH \underline{R}$. It is trivial to check that $(\underline{R}, \{ \mu_i \})$ is a multiplicative Mackey functor.\\
\indent Next, suppose that $(\underline{R}, \{ \mu_i \})$ is a multiplicative Mackey functor, and let $i : U \to V$. We define the norm map for $i$ by the commutative diagram below.
\begin{align*}
\xymatrix{
 \underline{R} (U) \ar[d]_-{\cong} \ar[rr]^-{n_i} && \underline{R} (V) \ar[d]^-{\cong} \\
 Hom (const_U \underline{A}, const_U \underline{R}) \ar[rr]_-{Hom(\spacedash, \mu_i) \circ i_{\star}} && Hom (const_V \underline{A}, const_V \underline{R}) }
\end{align*}
These norm maps respect composition and commute appropriately with restrictions by parts (ii) and (iii), respectively, of Definition~\ref{def:multmackeyfunct}. They convert identity maps into identity maps by parts (i) and (iii) of that definition, since any identity map is a pullback of the identity map of $G/G$. We verify the distributive law as follows. Let the diagram below be exponential.
\begin{align*}
\xymatrix{
 E \ar[d]_-{e} \ar[rr]^-{f} && D \ar[d]^-{p} \\
 W \ar[r]_-{j} & U \ar[r]_-{i} & V }
\end{align*}
For any $G$-map $k : B \to C$ we let $\Delta^k : const_C \underline{A} \to k_* (const_B \underline{A})$ denote the cotransfer associated to the fold map of $k$ for $S^0$. Then for any $w \in \underline{R} (W)$, the element $t_j (w)$ is represented by the composite below.
\begin{align*}
	const_U \underline{A} \xrightarrow{\Delta^j} j_* (const_W \underline{A}) \xrightarrow{j_* (w)} j_* (const_W \underline{R}) \xrightarrow{\nabla_j} const_U \underline{R}
\end{align*}
We now examine the diagram below.
\begin{align*}
\xymatrix{
 const_V \underline{A} \cong i_{\star} (const_U \underline{A}) \ar[r]^-{i_{\star} (\Delta^j)} & i_{\star} j_* (const_W \underline{A}) \ar[dl]_-{\cong} \ar[r]^-{i_{\star} j_* (w)} & i_{\star} j_* (const_W \underline{R}) \ar[dl]_-{\cong} \ar[d]^-{i_{\star} (\nabla_j)} \\
 p_* f_{\star} (const_E \underline{A}) \ar[d]_-{\cong} \ar[r]_-{p^* f_{\star} e^* (w)} & p_* f_{\star} (const_E \underline{R}) \ar[d]^-{p_* (\mu_f)} & i_{\star} (const_U \underline{R}) \ar[d]^-{\mu_i} \\
 p_* (const_D \underline{A}) \ar[r]_-{p_* (n_f r_e (w))} & p_* (const_D \underline{R}) \ar[r]_-{\nabla_p} & const_V \underline{R} }
\end{align*}
This diagram commutes: the lower square commutes by the definition of $n_f$, while the trapezoid commutes by part (iv) of Definition~\ref{def:multmackeyfunct}. The composite along the top and right sides is $n_i t_j (w)$. The composite along the left side is $\Delta^p$ by part (iv) of Proposition~\ref{prop:basicmultpushfacts}.\\
\indent Thus we have constructed functors $Tamb(G) \to MultMack(G)$ and $MultMack(G) \to Tamb(G)$. Examining the definition of the norm maps for a commutative ring spectrum, one easily sees that the composite functor from $Tamb(G)$ to itself is the identity. It remains to show that the Tambara functor associated to a multiplicative Mackey functor $(\underline{R}, \{ \mu_i \})$ determines the maps $\mu_i$. We proceed by induction on the order of $G$; hence, we may assume that the composite functor from $MultMack(H)$ to itself is the identity for all proper subgroups $H$ of $G$. Let $i : U \to V$ be a $G$-map. Suppose that $V = V_1 \coprod V_2$, so that $U = U_1 \coprod U_2$ and $i = i_1 \coprod i_2$. The diagram below is a pullback,
\begin{align*}
\xymatrix{
 U_1 \ar[d]_-{i_1} \ar[r]^-{\subseteq} & U \ar[d]^-{i} \\
 V_1 \ar[r]_-{\subseteq} & V }
\end{align*}
so part (iii) of Definition~\ref{def:multmackeyfunct} implies that $\mu_i \cong (\mu_{i_1}, \mu_{i_2})$. Part (iii) of that definition also implies that $\mu_i$ is determined by $\mu_{i'}$ for any map $i'$ which is isomorphic to $i$. Hence, we are reduced to the case where $V = G/H$ for some subgroup $H$ of $G$. In this case we have $U \cong G \times_H U'$ and $i \cong G \times_H i'$. Now, we may define a restriction functor as below,
\begin{align*}
	\Res_H^G : MultMack(G) &\to MultMack(H) \\
	(\underline{R}, \{ \mu_i \}) &\mapsto (\Res_H^G \underline{R}, \{ (\Res_H^G \mu)_i \})
\end{align*}
where for any $H$-map $i : U \to V$, we define $(\Res_H^G \mu)_i$ to be the map corresponding to $\mu_{G \times_H i}$ under the evident isomorphism
\begin{gather*}
	Hom ((G \times_H i)_{\star} (const_{G \times_H U} \underline{R}), const_{G \times_H V} \underline{R}) \cong \\
	Hom(i_{\star} (const_U \Res_H^G \underline{R}), const_V \Res_H^G \underline{R}).
\end{gather*}
We then have a commutative diagram as below.
\begin{align*}
\xymatrix{
 MultMack(G) \ar[d] \ar[r]^-{\Res_H^G} & MultMack(H) \ar[d] \\
 Tamb(G) \ar[r]_-{\Res_H^G} & Tamb(H) }
\end{align*}
It then follows by our inductive hypothesis that $\mu_i$ is determined by the associated Tambara functor when $V = G/H$ and $H$ is a proper subgroup of $G$. Hence we are reduced to the case $V = \ast$. Pulling $i$ back along the map $G/H \to \ast$ and applying the above argument, we see that all restrictions of $\mu_i$ to proper subgroups are determined. Hence, it suffices to show that
\begin{align*}
	\mu_i (G/G) : \underline{R}^{\otimes U} (G/G) \to \underline{R} (G/G)
\end{align*}
is determined on a set of elements which generate the geometric fixed points of $\underline{R}^{\otimes U}$. By Subsection~\ref{subsec:gsymmon}, we may take the elements of the form $(\ast \xrightarrow{=} \ast, u \in \underline{R} (U))$. Applying $\mu_i$ to this element, we obtain $n_i (u)$. Hence, $\mu_i$ is determined by the associated Tambara functor for all $G$-maps $i$.
\end{proof}

The notion of a multiplicative Mackey functor is formally somewhat similar to the notion of an algebra over a pair of operads, with the actions represented here by the $\mu_i$ and $\nabla_i$. Hence, one might expect the Tambara functors to be the algebras over some monad. We confirm this below.

\begin{prop}\label{prop:tambmonad}
For any finite group $G$ the category $Tamb(G)$ is isomorphic to the category of algebras over the free Tambara functor monad $\TT$.
\end{prop}
\begin{proof}
It is formal that we have a functor
\begin{align*}
	Tamb(G) \to Mack(G) [\TT].
\end{align*}
Now suppose we have an algebra $\underline{R}$ over $\TT$. Let $m : \TT \underline{R} \to \underline{R}$ denote the action map. Then for any map $f : X \to Y$ in $\Fin_G$ we define $n_f \defeq m (Y) \circ n_f \circ \theta_{\underline{R}} (X)$. (Recall that $\theta_{\underline{R}}$ is the unit map for the monad $\TT$.) Since $\theta_{\underline{R}}$ and $m$ commute with restrictions and transfers, and $m \circ \theta_{\underline{R}} = 1$, all the axioms for a Tambara functor are clear, except that the norm maps respect composition. Let $p : \TT \circ \TT \to \TT$ be the multiplication map of $\TT$. Then the diagram below commutes by the naturality of $\theta$ and the definition of $\TT$-algebras.
\begin{align*}
\xymatrix{
 \TT \underline{R} \ar[d]_-{m} \ar[r]^-{\theta_{\TT \underline{R}}} & \TT \TT \underline{R} \ar[d]_-{\TT (m)} \ar[r]^-{p_{\underline{R}}} & \TT \underline{R} \ar[d]^-{m} \\
 \underline{R} \ar[r]_-{\theta_{\underline{R}}} & \TT \underline{R} \ar[r]_-{m} & \underline{R} }
\end{align*}
Now let $f : X \to Y$ and $h : Y \to Z$ be $G$-maps, and let $x \in \underline{R} (X)$. We apply the above diagram to the element $n_h \theta_{\TT \underline{R}} (n_f \theta_{\underline{R}} (x)) \in \TT \TT \underline{R} (Z)$. The middle vertical map is a map of Tambara functors, so by using its commutativity with norms and the commutativity of the left square above we obtain the element $n_h \theta_{\underline{R}} m n_f \theta_{\underline{R}} (x)$. Applying $m$, we obtain $n_h (n_f (x))$. Now note that $n_h \theta_{\TT \underline{R}} (n_f \theta_{\underline{R}} (x))$ can be represented by the pair below.
\begin{align*}
	\big( Y \xrightarrow{h} Z \xrightarrow{=} Z, (X \xrightarrow{f} Y \xrightarrow{=} Y, x) \big)
\end{align*}
Applying $p_{\underline{R}}$, which is the counit of the adjunction between Mackey and Tambara functors for $\TT \underline{R}$, we obtain
\begin{align*}
	n_h \big( (X \xrightarrow{f} Y \xrightarrow{=} Y, x) \big) &= (X \xrightarrow{h \circ f} Z \xrightarrow{=} Z, x) \\
	                                                                                            &= n_{hf} \theta_{\underline{R}} (x).
\end{align*}
Hence, applying $m$ we obtain $n_{hf} (x)$. We have constructed a functor
\begin{align*}
	Mack(G)[\TT] \to Tamb(G).
\end{align*}
It is clear that the composite functor from $Tamb(G)$ to itself is the identity. Hence it suffices to show that the action map $m : \TT \underline{R} \to \underline{R}$ of an algebra is determined by the associated Tambara functor. We know from Section~\ref{sec:sympow} that any element of $\TT \underline{R} (X)$ is a difference of elements of the form $(U \xrightarrow{i} V \xrightarrow{j} X, u)$, where $u \in \underline{R}$. This element is $t_j n_i \theta_{\underline{R}} (u)$, so we must have
\begin{align*}
	m \big( (U \xrightarrow{i} V \xrightarrow{j} X, u) \big) = t_j m n_i \theta_{\underline{R}} (u) = t_j n_i (u).
\end{align*}
\end{proof}

Now we know by Corollary~\ref{cor:freetambara} that $\TT \cong \underline{\pi}_0 \CC \HH$. There is a fascinating relationship between the notion of a multiplicative Mackey functor and the structure of $\underline{\pi}_0 \CC \HH$, which we now explain. Let $\underline{M}$ be a Mackey functor, and let $\HH \underline{M}$ be positive cofibrant. Then there is a weak equivalence as below.
\begin{align*}
	\CC \HH \underline{M} \xleftarrow{\sim} \bigvee_{n \geq 0} E_G {\Sigma_n}_+ \wedge_{\Sigma_n} (\HH \underline{M})^{\wedge n}
\end{align*}

Here, $E_G \Sigma_n$ is the universal space for the family $\FF_G [n]$ of subgroups of $G \times \Sigma_n$ which have trivial intersection with $1 \times \Sigma_n$. Recall that these subgroups are the sets of the form
\begin{align*}
	H^{\phi} \defeq \{ (h, \phi(h)) : h \in H \}
\end{align*}

for subgroups $H$ of $G$ and homomorphisms $\phi : H \to \Sigma_n$. We denote the category of $(G \times \Sigma_n)$-orbits of the form $(G \times \Sigma_n) / H^{\phi}$ by $\OO(G;n)$. Now if we only wish to calculate $\underline{\pi}_0$, we may replace $E_G \Sigma_n$ with its $1$-skeleton. We may take the following,
\begin{align*}
	E_G \Sigma_n^{[1]} = \hocoeq \big( \coprod (G \times \Sigma_n) / L^{\lambda} \rightrightarrows \coprod_{H,\phi} (G \times \Sigma_n) / H^{\phi} \big)
\end{align*}

where the first coproduct is over pairs of distinct maps in $\OO(G;n)$ (with the same domain). Smashing over $\Sigma_n$ with $(\HH \underline{M})^{\wedge n}$ and taking $\underline{\pi}_0$, we obtain a coequalizer, which can be re-written as a colimit as below.
\begin{align*}
	\underline{\pi}_0 (\HH \underline{M})^{\wedge n} / \Sigma_n \cong \colim_{(G \times \Sigma_n) / H^{\phi} \in \OO(G;n)} \underline{\pi}_0 \big( ((G \times \Sigma_n) / H^{\phi})_+ \wedge_{\Sigma_n} (\HH \underline{M})^{\wedge n} \big)
\end{align*}

To understand this colimit, we shall require a few lemmas.

\begin{lem}\label{lem:freesigmaorbit}
Let $H^{\phi}$ be a subgroup of $G \times \Sigma_n$ in $\FF_G[n]$, and let $X$ be a $(G \times \Sigma_n)$-spectrum. Then there is a natural isomorphism
\begin{align*}
	(G \times \Sigma_n / H^{\phi})_+ \wedge_{\Sigma_n} X \hspace{.5cm} \cong \hspace{.5cm} G_+ \wedge_H X^{\phi},
\end{align*}
where $X^{\phi}$ is $X$ with $H$-action multiplied by the pullback of the $\Sigma_n$-action along $\phi$.
\end{lem}
\begin{proof}
The above spectrum is
\begin{align*}
	((G \times \Sigma_n / H^{\phi})_+ \wedge X)/\Sigma_n &\cong ((G \times \Sigma_n)_+ \wedge_{H^{\phi}} X)/\Sigma_n \\
	&\cong (\Sigma_n \SLASH G \times \Sigma_n)_+ \wedge_{H^{\phi}} X \cong G_+ \wedge_{H^{\phi}} X,
\end{align*}
where $H^{\phi}$ acts on $G$ via its projection onto $H$. The last spectrum above can be described equivalently as $G_+ \wedge_H X^{\phi}$.
\end{proof}

We leave the next two lemmas to the reader.

\begin{lem}\label{lem:orbitmaps}
Let $H$ and $L$ be subgroups of $G$, and let $\phi : H \to \Sigma_n$ and $\lambda : L \to \Sigma_n$ be homomorphisms. There is a map in $\Fin_{G \times \Sigma_n}$
\begin{align*}
	G \times \Sigma_n / L^{\lambda} \to G \times \Sigma_n / H^{\phi}
\end{align*}
sending the identity coset to $(g,\sigma)H^{\phi}$ if and only if $L \subseteq gHg^{-1}$ and
\begin{align*}
	\lambda(l) = \sigma \phi(g^{-1} l g) \sigma^{-1}
\end{align*}
for all $l \in L$.
\end{lem}

\begin{lem}\label{lem:orbitmapsinduce}
Let $H$ and $L$ be subgroups of $G$, and let $\phi : H \to \Sigma_n$ and $\lambda : L \to \Sigma_n$ be homomorphisms. Suppose there is a map in $\Fin_{G \times \Sigma_n}$
\begin{align*}
	f : G \times \Sigma_n / L^{\lambda} \to G \times \Sigma_n / H^{\phi}
\end{align*}
sending the identity coset to $(g,\sigma)H^{\phi}$. Let the following diagram commute, where the vertical isomorphisms are given by Lemma~\ref{lem:freesigmaorbit} and $X \in Sp_G$.
\begin{align*}
\xymatrix{
 (G \times \Sigma_n / L^{\lambda})_+ \wedge_{\Sigma_n} X^{\wedge n} \ar[d]_-{\cong} \ar[rr]^-{f_+ \wedge_{\Sigma_n} Id} & & (G \times \Sigma_n / H^{\phi})_+ \wedge_{\Sigma_n} X^{\wedge n} \ar[d]^-{\cong} \\
 G_+ \wedge_{L} (X^{\times n})^{\lambda} \ar[rr]_-{f'} & & G_+ \wedge_{H} (X^{\wedge n})^{\phi} }
\end{align*}
Then $f'$ may be represented schematically as below.
\begin{align*}
	f' : G_+ \wedge_{L} (X^{\wedge n})^{\lambda} &\to G_+ \wedge_{H} (X^{\wedge n})^{\phi} \\
	[k, \wedge_j x_j] &\mapsto [kg, \wedge_j (g^{-1} \cdot x_{\sigma(j)})]
\end{align*}
\end{lem}

Now when $H^{\phi} \in \FF_G [n]$, let us denote the set $\{ 1, ..., n \}$ with $H$-action through $\phi$ by $\{ 1, ..., n\}_{\phi}$. Using Lemma~\ref{lem:freesigmaorbit}, we now have the following,
\begin{align}\label{eq:topcolim}
	\underline{\pi}_0 (\HH \underline{M})^{\wedge n} / \Sigma_n \cong \colim_{(G \times \Sigma_n) / H^{\phi} \in \OO(G;n)} \Ind_H^G (\Res_H^G \underline{M})^{\otimes \{ 1, ..., n \}_{\phi}}
\end{align}

where the maps in the colimit are determined by Lemma~\ref{lem:orbitmapsinduce}. We now construct a similar colimit from the definition of a multiplicative Mackey functor. We focus on (iii) of Definition~\ref{def:multmackeyfunct} (the assumption on pullback squares). We define $P(G)$ to be the category whose objects are morphisms in $\Fin_G$, and where the morphisms are morphisms of arrows which produce pullback squares. If $(\underline{M}, \{ \mu_i \})$ is a multiplicative Mackey functor then the maps
\begin{align*}
	\mu_i : i_{\star} (const_U \underline{M}) \to const_V \underline{M}
\end{align*}

are adjoint, via the maps $\pi^V : V \to \ast$, to maps
\begin{align*}
	\mu_i' : \pi^V_* i_{\star} (const_U \underline{M}) \to \underline{M}.
\end{align*}

If the diagram below is a pullback,
\begin{align*}
\xymatrix{
 B \ar[d]_-{b} \ar[r]^-{f} & D \ar[d]^-{d} \\
 C \ar[r]_-{c} & E }
\end{align*}

then the commutative triangles
\begin{align*}
\xymatrix{
 b_{\star} (const_B \underline{M}) \ar[d]_-{\mu_b} \ar[r]^-{\cong} & c^* d_{\star} (const_D \underline{M}) \ar[dl]^-{c^* (\mu_d)} \\
 const_C \underline{M} }
\end{align*}

are adjoint to commutative triangles,
\begin{align*}
\xymatrix{
 c_* b_{\star} (const_B \underline{M}) \ar[d] \ar[r] & d_{\star} (const_D \underline{M}) \ar[dl]^-{\mu_d} \\
 const_E \underline{M} }
\end{align*}

and hence to commutative triangles as below.
\begin{align*}
\xymatrix{
 \pi^C_* b_{\star} (const_B \underline{M}) \ar[d]_-{\mu_b'} \ar[r] & \pi^E_* d_{\star} (const_D \underline{M}) \ar[dl]^-{\mu_d'} \\
 \underline{M} }
\end{align*}

Note that the top maps above are defined for all $\underline{M}$, and give a functor from $P(G)$ to $Mack(G)$. The diagram above then gives a cocone from this functor to $\underline{M}$. Hence, a reasonable candidate for the "free multiplicative Mackey functor monad" is the colimit below.
\begin{align*}
	\colim_{i : U \to V \in P(G)} \pi^V_* i_{\star} (const_U \underline{M})
\end{align*}

We now identify this colimit with~\ref{eq:topcolim}.

\begin{thm}\label{thm:pi0ceqmultmack}
For any $\underline{M} \in Mack(G)$ there is an isomorphism
\begin{align*}
	F : \colim_{i : U \to V \in P(G)} \pi^V_* i_{\star} (const_U \underline{M}) \xrightarrow{\cong} \pi_0 \CC \HH \underline{M}.
\end{align*}
\end{thm}
\begin{proof}
We begin by defining $F$. One can mimic the notion of a multiplicative Mackey functor, replacing Mackey functors with spectra. Then commutative ring spectra are naturally "multiplicative" in this sense. Taking $\HH \underline{M}$ to be positive cofibrant, we take $F$ to be the map induced by applying $\underline{\pi}_0$ to the composites
\begin{align*}
	\pi_* i_{\star} (const_U \HH \underline{M}) \to \pi_* i_{\star} (const_U \CC \HH \underline{M}) \xrightarrow{\mu_i'} \CC \HH \underline{M},
\end{align*}
where the left maps are induced by the inclusions $\HH \underline{M} \xrightarrow{\subseteq} \CC \HH \underline{M}$.\\
\indent To show that $F$ is an isomorphism, we begin by examining the structure of $P(G)$. First of all, if $V = V_1 \coprod V_2$ then $U = U_1 \coprod U_2$ and $i = i_1 \coprod i_2$. Then we have
\begin{align*}
	\pi^V_* i_{\star} (const_U \underline{M}) \cong \pi^{V_1}_* {i_1}_{\star} (const_{U_1} \underline{M}) \oplus \pi^{V_2}_* {i_2}_{\star} (const_{U_2} \underline{M}),
\end{align*}
and since the diagrams
\begin{align*}
\xymatrix{
 U_j \ar[d]_-{i_j} \ar[r]^-{\subseteq} & U \ar[d]^-{i} \\
 V_j \ar[r]_-{\subseteq} & V }
\end{align*}
are pullbacks for $j = 1, 2$, if one has a cocone then the map for $i$ must be the direct sum of the maps for $i_1$ and $i_2$. Thus one sees that we obtain the same colimit if we restrict ourselves to the full subcategory of $P(G)$ such that the targets $V$ are orbits. In this case, the maps $i$ have degrees. Since pullbacks preserve fibers, this full subcategory is a disjoint union over components, by degree. We may further restrict to a full subcategory containing all isomorphism classes. Hence, for each $n \geq 0$ we let $P(G;n)$ denote the full subcategory of $P(G)$ consisting of the arrows of the form
\begin{align*}
	G \times_H \{ 1, ..., n \}_{\phi} \xrightarrow{G \times_H \pi} G/H,
\end{align*} 
where $\phi : H \to \Sigma_n$ is a homomorphism. Thus, our colimit is a direct sum over $n \geq 0$ of colimits of $\Ind_H^G (\Res_H^G \underline{M})^{\otimes \{ 1, ..., n \}_{\phi}}$ over the $P(G;n)$. We will identify these precisely with the colimits~\ref{eq:topcolim}.\\
\indent First of all, it is clear that the objects of $P(G;n)$ correspond to the objects of $\OO(G;n)$. We need only show that the morphisms in the two diagrams correspond as well; it is then easy to see using the explicit isomorphisms in Lemma~\ref{lem:freesigmaorbit} that the two cocones are identical. First we dispose of the special cases $n = 0, 1$. For $n = 0$, the sources $U$ are empty, and we see that $\emptyset \to \ast$ is the terminal object. Hence, our colimit over $P(G,0)$ is $\underline{\pi}_0 S^0 \cong \underline{A}$, and this is clearly mapped isomorphically to $\underline{\pi}_0 (\HH \underline{M})^{\wedge 0}$. Similarly, for $n = 1$ the maps $i$ are isomorphisms and we have a terminal object $\ast \xrightarrow{=} \ast$. Hence, our colimit over $P(G;1)$ is $\underline{M}$, which is clearly mapped isomorphically to $\underline{\pi}_0 (\HH \underline{M})^{\wedge 1}$.\\
\indent Finally, fix $n \geq 2$. Let $L$ and $H$ be two subgroups of $G$, and let $\lambda : L \to \Sigma_n$ and $\phi : H \to \Sigma_n$ be homomorphisms. Now consider a diagram as below.
\begin{align*}
\xymatrix{
 G \times_L \{ 1, ..., n \}_{\lambda} \ar[d] \ar[r]^-{p} & G \times_H \{ 1, ..., n \}_{\phi} \ar[d] \\
 G/L \ar[r]_-{q} & G/H }
\end{align*}
This diagram is a pullback if it commutes and induces isomorphisms on corresponding fibers. Since $G/L$ is a transitive $G$-set, the condition on the fibers may be checked on the single fiber over $L$. The map $q$ corresponds to an $L$-fixed coset of $G/H$. Suppose that $q (L) = gH$, so that $L \subseteq gHg^{-1}$. Then the fiber over $gH$ may be identified as $\{ [g, j] : 1 \leq j \leq n \}$. Now let $\sigma \in \Sigma_n$. Then we attempt to define a map $p$ by
\begin{align*}
	p ([1, j]) \defeq [g, \sigma^{-1} (j)].
\end{align*}
This defines a pullback diagram if and only if the above definition is $L$-equivariant. Hence let $l \in L$. We have $l \cdot [1, j] = [1, \lambda(l) (j)]$, so that $p (l \cdot [1, j]) = [g, \sigma^{-1} \lambda(l) (j)]$. We also have
\begin{align*}
	l \cdot p([1, j]) = [lg, \sigma^{-1} (j)] = [g \cdot g^{-1} l g, \sigma^{-1} (j)] = [g, \phi(g^{-1} l g) \sigma^{-1} (j)].
\end{align*}
Hence, we have a pullback diagram if and only if
\begin{align*}
	\sigma^{-1} \lambda(l) = \phi(g^{-1} l g) \sigma^{-1},
\end{align*}
for all $l \in L$, which is equivalent by Lemma~\ref{lem:orbitmaps} to the existence of a $(G \times \Sigma_n)$-map from $(G \times \Sigma_n)/L^{\lambda}$ to $(G \times \Sigma_n)/H^{\phi}$ sending the identity coset to $(g, \sigma) H^{\phi}$. Now for $h \in H$ we have
\begin{align*}
	[g, \sigma^{-1} (j)] = [gh, \phi(h)^{-1} \sigma^{-1} (j)] = [gh, (\sigma \phi(h))^{-1} (j)],
\end{align*}
so the map specified by $(g, \sigma)$ is the same as the map specified by $(gh, \sigma \phi(h))$; that is, two element of $G \times \Sigma_n$ specify the same pullback diagram exactly when they represent the same coset of $H^{\phi}$, and hence the same $(G {\times} \Sigma_n)$-map from $(G {\times} \Sigma_n)/L^{\lambda}$ to $(G {\times} \Sigma_n)/H^{\phi}$. Using this correspondence, one easily checks that $P(G;n)$ is isomorphic to $\OO(G;n)$. Let $i : G {\times_L} \{ 1, ..., n \}_{\lambda} \to G/L$ and $k : G {\times_H} \{ 1, ..., n \}_{\phi} \to G/H$ be the projection maps. Then we have, schematically,
\begin{align*}
	(gH, \wedge_j x_j) = g \cdot (H, \wedge_j (g^{-1} \cdot x_j))
\end{align*}
in $k_{\star} (const_{G \times_H \{ 1, ..., n \}_{\phi}} X)$ for any $X \in Sp_G$, so one easily sees from Lemma~\ref{lem:orbitmapsinduce} that the maps
\begin{gather*}
	G_+ \wedge_L (\Res_L^G \HH \underline{M})^{\wedge \{ 1, ..., n \}_{\lambda}} \cong \pi^{G/L}_* i_{\star} (const_{G \times_L \{ 1, ..., n \}_{\lambda}} \HH \underline{M}) \to \\
	\pi^{G/H}_* k_{\star} (const_{G \times_H \{ 1, ..., n \}_{\phi}} \HH \underline{M}) \cong G_+ \wedge_H (\Res_H^G \HH \underline{M})^{\wedge \{ 1, ..., n \}_{\phi}}
\end{gather*}
that we used to define $F$ coincide, after applying $\underline{\pi}_0$, with those occuring in the colimit~\ref{eq:topcolim}.
\end{proof}

\begin{cor}\label{cor:tfuncmultmack}
Let $\underline{M} \in Mack(G)$. For each $G$-map $i : U \to V$, let $\mu_i$ be the corresponding structure map for $\TT \underline{M}$ regarded as a multiplicative Mackey functor, and let $\mu_i'$ be the adjoint of $\mu_i$. Then the composite maps
\begin{align*}
	\pi^V_* i_{\star} (const_U \underline{M}) \xrightarrow{\pi^V_* i_{\star} (const_U \theta_{\underline{M}})} \pi^V_* i_{\star} (const_U \TT \underline{M}) \xrightarrow{\mu_i'} \TT \underline{M}
\end{align*}
determine an isomorphism
\begin{align*}
	\colim_{i : U \to V \in P(G)} \pi^V_* i_{\star} (const_U \underline{M}) \xrightarrow{\cong} \TT \underline{M}.
\end{align*}
\end{cor}

\indent \emph{Remark:} The way that these comparison maps are defined is precisely how one would define a map of monads whose algebra categories were isomorphic. Unfortunately, it appears to be difficult to construct a monad algebraically from the definition of multiplicative Mackey functors, though it should be possible in principle.\\
\indent \emph{Remark:} If $\underline{R} \in Tamb(G)$ and $i : U \to V$ is a map in $\Fin_G$, then we can compute the adjoint structure maps $\mu_i'$ for $\underline{R}$ in terms of the description given by Definition~\ref{def:intrmultpushmackey}. For any $X \in \Fin_G$, $\mu_i' (X)$ is as below.
\begin{align*}
	\mu_i' (X) : \pi^V_* i_{\star} \underline{R} (X) &\to \underline{R} (X) \\
	(Y \xrightarrow{j \times k} X \times V, y \in \underline{R} (Y \times_V U)) &\mapsto t_j n_{\pi_Y} (y)
\end{align*}

It follows that the maps in Corollary~\ref{cor:tfuncmultmack} are as below.
\begin{align*}
	(Y \xrightarrow{j \times k} X \times V, y) \mapsto (Y \times_V U \xrightarrow{\pi_Y} Y \xrightarrow{j} X, y)
\end{align*}

\vspace{1cm}

\indent \emph{Acknowledgement:} I would like to thank Mike Hill for some key insights on this topic.

\vspace{1cm}

\bibliographystyle{alphanum}
\bibliography{sympownormbiblio}

\end{document}